\documentclass[a4paper,twoside,12pt,leqno]{article}
\usepackage{amsmath,amsthm,amssymb,enumerate}
\usepackage{eucal} 
\usepackage{epsfig, float, afterpage, array}
\usepackage{hhline}
\usepackage{graphpap}

\swapnumbers
\theoremstyle{plain}
\newtheorem{theorem}{Theorem}[section]
\newtheorem{lemma}[theorem]{Lemma}
\newtheorem{corollary}[theorem]{Corollary}
\newtheorem{proposition}[theorem]{Proposition}

\theoremstyle{definition}
\newtheorem{definition}[theorem]{Definition}

\newtheorem{definitions}[theorem]{Definitions}
\newtheorem{example}[theorem]{Example}
\newtheorem{examples}[theorem]{Examples}
\newtheorem{notation}[theorem]{Notation}

\newtheorem{remark}[theorem]{Remark}
\newtheorem{remarks}[theorem]{Remarks}
\newtheorem{history}[theorem]{Historical Remarks}

\newtheorem{review}[theorem]{Review}
\numberwithin{equation}{theorem}
\numberwithin{figure}{theorem}

{\addvspace{\baselineskip}}

\newcommand{\abs}[1]{\vert#1\vert} 
\newcommand{\norm}[1]{\|#1\|} 

\newcommand{\gen}[1]{\langle#1\rangle}

\newcommand{\ray}{\,\,\rule[3pt]{.6cm}{0.4pt}\hskip 2pt\,\,}
\newcommand{\doubleray}{\begin{aligned}[t]\ray\\[-20pt]\ray\end{aligned}}

\DeclareMathOperator{\Out}{Out}
\DeclareMathOperator{\Aut}{Aut}
\DeclareMathOperator{\Sym}{Sym}
\DeclareMathOperator{\Stab}{Stab}
\DeclareMathOperator{\Artin}{Artin}
\DeclareMathOperator{\lab}{label}
\def \B {\mathcal{B}}
\def \naturals{\mathbb{N}}
\def \integers {\mathbb{Z}}
\def \reals{\mathbb{R}}
\def \complexes{\mathbb{C}}

\def\d1{\discretionary{-}{}{-}}

\begin{document}

\pagestyle{myheadings} \markboth{Actions of the braid group} {Llu\'{\i}s Bacardit and Warren Dicks}  
\title{Actions of the braid group, and  new algebraic proofs of results of  Dehornoy and Larue.}

\author
{Llu\'{\i}s Bacardit and Warren Dicks}

\date{\small\today}

\maketitle
 
 \bigskip

\begin{abstract} This article surveys many standard results 
about the braid group with emphasis on simplifying the usual algebraic proofs. 

We use van der Waerden's trick to illuminate the Artin-Magnus 
proof of the classic presentation of the algebraic mapping-class group
of a punctured disc.  

We give a simple, new proof of the
Dehornoy-Larue braid-group trichotomy, and, hence, recover the Dehornoy right-ordering of the braid group.  

We then turn to the Birman-Hilden theorem 
concerning braid-group actions on free products of cyclic groups, 
and the consequences derived by Perron-Vannier, and the connections with the Wada representations.
We recall the very simple Crisp-Paris proof of the Birman-Hilden theorem 
that uses the Larue-Shpilrain technique.
Studying ends of free groups permits a deeper understanding of the braid group; 
this gives us a generalization of the Birman-Hilden theorem.  
Studying Jordan curves in the punctured disc permits a still deeper understanding
of the braid group; this gave Larue, in his PhD thesis, correspondingly deeper results, and, in
an appendix, we recall  the essence of Larue's thesis, giving simpler combinatorial proofs.

\bigskip

{\footnotesize
\noindent \emph{{\normalfont 2000}\,Mathematics Subject Classification.} Primary: 20F36;
Secondary: 20F34, 20E05, 20F60.

}
\end{abstract}

\section{General Notation}

Let $\naturals$ denote the set of finite cardinals, $\{0,1,2,\ldots\}$.

Throughout, we fix an element $n$ of $\naturals$.

Let $i$, $j \in \integers$ and let $v$ be a symbol.  We define 
\begin{align*}
&[i{\uparrow}j]:=  \{k \in \integers \mid i \le k \text{ and } k \le j\},
\\
&[i{\downarrow}j]:= \{k \in \integers \mid i \ge k \text{ and } k \ge j\},
\\
&([i{\uparrow}j]):= \begin{cases}
(i,i+1,\ldots, j-1, j) \in \integers^{j-i+1} &\text{if $i \le j$,}\\
() \in \integers^{0}  &\text{if $i > j$,}
\end{cases}
\\
\\
&([i{\downarrow}j]):= \begin{cases}
(i,i-1,\ldots, j+1, j) \in \integers^{i-j+1} &\text{if $i \ge j$,}\\
() \in \integers^{0} &\text{if $i < j$.}
\end{cases}
\end{align*} 
Also, $v_{[i{\uparrow}j]} := \{v_k \mid k \in [i{\uparrow}j]\}$,
and this will usually be a subset of some ambient set, $G$.
If $i \le j$, $v_{([i{\uparrow}j])}:=(v_i,v_{i+1}, \ldots, v_{j-1},v_j) \in G^{j-i+1}$, and, if $G$ is a group,
$\Pi v_{[i{\uparrow}j]} := v_i   v_{i+1}   \cdots v_{j-1}   v_j \in G$.
If $i > j$, $v_{([i{\uparrow}j])}:= ()$, the $0$-tuple,
and $\Pi v_{[i{\uparrow}j]}:=1$, the empty product.  
We define $v_{[i{\downarrow}j]}$,  $v_{([i{\downarrow}j])}$ and $\Pi v_{[i{\downarrow}j]}$, analogously. 
Thus, if $i \ge  j$,    $\Pi v_{[i{\downarrow}j]}:= v_{i}v_{i-1} \cdots v_{j+1}v_j$.
Finally, $[i{\uparrow}\infty[\,\,\,\, :=\{k \in \integers \mid i \le k \}$.

For elements $a$, $b$ of a group $G$,   $\overline a := a^{-1}$,   $a^b := \overline b a b$,
   $a^{nb}:= \overline b a^n b$, and $[a]: = \{a^g\mid g \in G\}$,  the conjugacy class of $a$ in $G$.
The group of all automorphisms of $G$ will be denoted by~$\Aut(G)$.

An {\it ordering} of a set will mean a total order for the set.  An {\it ordered} set is one endowed with a
specific ordering. We will speak of $n$-tuples {\it for} a given set and  $n$-tuples {\it of  elements of} a given set.

\section{Outline}
\label{sec:intro}

 Let $\Sigma_{0,1,n} := \gen{ \{z_1\} \cup t_{[1{\uparrow}n]} \mid z_1 \Pi t_{[1{\uparrow}n]} = 1}$.  
Then $\Sigma_{0,1,n}$ is a one-relator group which is freely generated by the set $t_{[1{\uparrow}n]}$.

Let $\Out^+_{0,1,n}$ denote the subgroup of $\Aut(\Sigma_{0,1,n})$ consisting 
of all automorphisms of $\Sigma_{0,1,n}$ 
which map the set $\{z_1\} \cup \{[t_i]\}_{i\in[1{\uparrow}n]}$  to itself.
Let $\Out_{0,1,0}$ denote  $\Aut(\integers)$, and, for $n \ge 1$,  
let $\Out_{0,1,n}$ denote the group of all automorphisms of $\Sigma_{0,1,n}$ 
which map the set   $\{z_1, \overline z_1 \} \cup \{[t_i], [\overline t_i]\}_{i\in[1{\uparrow}n]}$ to itself.
Then $\Out^+_{0,1,n}$  is a subgroup 
of index two in $\Out_{0,1,n}$.  We call $\Out_{0,1,n}$ the algebraic mapping-class group of
the surface of genus $0$ with $1$ boundary component and $n$ punctures; 
see~\cite{DF05} for background on algebraic mapping-class groups.

Frequently, $\Out^+_{0,1,n}$  will be denoted $\B_n$ and called the {\it $n$-string braid group}. 
(The similar symbol $B_n$ denotes a certain Coxeter diagram.)

In Section~\ref{sec:gen}, we define $\sigma_{[1{\uparrow}n-1]} \subseteq \Out^+_{0,1,n}$,
we review Artin's 1925 proof that $\sigma_{[1{\uparrow}n-1]}$ 
generates $\Out^+_{0,1,n}$, and we present intermediate results that we shall apply in
subsequent sections.
In Section~\ref{sec:Artin}, we recall the definition of Artin groups, specifically $\Artin\gen{A_n}$, $\Artin\gen{B_n}$
and $\Artin\gen{D_n}$. 
In Section~\ref{sec:pres}, we verify Artin's 1925 result that $\Out^+_{0,1,n} \simeq \Artin\gen{A_{n-1}}$, by combining
Magnus' 1934 proof, Manfredini's observation that $\Out^+_{0,1,(n-1)\perp 1} \simeq \Artin\gen{B_{n-1}}$, 
and the van der Waerden trick, to condense the calculations involved.

In Section~\ref{sec:trich}, we use results of Section~\ref{sec:Artin}
 to recover the Dehornoy-Larue trichotomy for $\B_n$ and the Dehornoy right-ordering of $\B_n$; 
this represents a substantial simplification.  Let us emphasize
that we verify directly that $\Out^+_{0,1,n}$ satisfies the trichotomy, 
in contrast with the approach by Larue~\cite{Larue94}  of using
the trichotomy for  $\Artin\gen{A_{n-1}}$ to verify that $\Artin\gen{A_{n-1}}$ acts faithfully on~$\Sigma_{0,1,n}$.

In Section~\ref{sec:ends}, we review the action of $\B_n$ on the set of ends of $\Sigma_{0,1,n}$.
The argument of Thurston given in~\cite{ShortWiest} yields the Dehornoy right-ordering of $\B_n$, 
but not the trichotomy.  By analysing further, we obtain
new results about the $\B_n$-orbit of $t_1$ in $\Sigma_{0,1,n}$.

In  Section~\ref{sec:freeprods},  for  each $m \ge 2$, we introduce $\Out_{0,1,n^{(m)}}$, the algebraic mapping-class
group of the disc with $n$ $C_m$-points. 
We recall the Crisp-Paris proof of the Birman-Hilden result that the natural map from
$\Out_{0,1,n}$ to $\Out_{0,1,n^{(m)}}$ is injective, and then
modify an argument of Steve Humphries to show that there is a natural 
identification $\Out_{0,1,n^{(m)}}=\Out_{0,1,n}$.
The new results obtained in Section~\ref{sec:ends} then provide 
additional information in this context.

In Section~\ref{sec:finite index}, we review some applications by Perron-Vannier~\cite{PerronVannier}
of the above Birman-Hilden result, and  discuss connections with  
the actions given by Wada~\cite{Wada} and studied by
Shpilrain~\cite{Shpilrain} and Crisp-Paris~\cite{CrispParis1},~\cite{CrispParis2}.

Following a kind suggestion of Patrick Dehornoy, we studied 
the analysis of the  $\B_n$-orbit of $t_1$ in $\Sigma_{0,1,n}$ given by David Larue~\cite{LarueThesis94}.  
Larue's approach is combinatorial and uses polygonal curves in the punctured disc. 
By combining Larue's approach with Whitehead's use of graphs, we were able to simplify
 Larue's main arguments, and we record our combinatorial approach in an appendix.  We also show how
Larue's results imply the results we had obtained, more easily, by studying ends, in Section~\ref{sec:ends}.

\section{Artin's generators of $\B_n$}\label{sec:gen}

In this section we describe the famous generating set of $\B_n$.
Let us fix more notation related to 
$\Sigma_{0,1,n} = \gen{ \{z_1\} \cup t_{[1{\uparrow}n]} \mid z_1 \Pi t_{[1{\uparrow}n]} = 1}$
and $\B_n \le \Aut(\Sigma_{0,1,n})$.

\begin{notation}\label{not:basic} Let $m \in \naturals$.  
Consider an $m$-tuple $a_{([1{\uparrow}m])}$ for $t_{[1{\uparrow}n]} \cup \overline t_{[1{\uparrow}n]}$, and 
an element $w$ of $\Sigma_{0,1,n}$. 

 If
$\Pi a_{[1{\uparrow}m]} = w$ in $\Sigma_{0,1,n}$, 
we say that  $a_{([1{\uparrow}m])}$ is an {\it expression}
for~$w$.  We say that the expression $a_{([1{\uparrow}m])}$ is  {\it reduced} if, for all $j \in [1{\uparrow}n-1]$, 
$a_{j+1}\ne \overline a_j$ in $t_{[1{\uparrow}n]} \cup \overline t_{[1{\uparrow}n]}$.
For each element of $\Sigma_{0,1,n}$, there exists a unique reduced expression called the {\it normal form}.

Suppose that  $a_{([1{\uparrow}m])}$
is the normal form for $w$.  We define the {\it length} of $w$ to be $\abs{w}:= m$.
The set of elements of $\Sigma_{0,1,n}$ whose normal forms have $a_{([1{\uparrow}m])}$
 as an initial segment is denoted $(w{\star})$; and, the set of
elements of $\Sigma_{0,1,n}$ whose normal forms have $a_{([1{\uparrow}m])}$
 as a terminal segment is denoted  $({\star} w)$.  The elements of $(w{\star})$ are said to {\it begin with } $w$,
and the elements of $({\star} w)$  are said to {\it end with}~$w$.

Let $\Sym_n$ denote the group of permutations of $[1{\uparrow}n]$ acting on the right (on $[1{\uparrow}n]$).

Let $\phi \in \B_n$. 
  There exists
a unique permutation $\pi \in \Sym_n$, and a unique $(n+2)$-tuple $(w_{([0{\uparrow}n+1])})$ for $\Sigma_{0,1,n}$
 such that
$w_0 = 1$ and $w_{n+1} = 1$, and,
for each $i \in [1{\uparrow}n]$,  $w_i \not\in (t_{i^\pi}{\star})\cup (\overline t_{i^\pi}{\star})$ 
and  $t_i^\phi = t_{i^\pi}^{w_i}$.
For each $i \in [0{\uparrow}n]$, let $u_i = w_i \overline w_{i+1}$. 
If $j \in [i{\uparrow}n]$, then $\Pi u_{[i{\uparrow}j]} =  w_i \overline w_{j+1}$.
In particular, $\Pi u_{[i{\uparrow}n]} =   w_i$.
We define    $ \pi(\phi) := \pi$, $w_i(\phi) :=  w_i$, $i \in [0{\uparrow}n+1]$, and 
$u_i(\phi):=u_i$, $i \in [0{\uparrow}n]$.  
We write $\norm{\phi}:= \sum\limits_{i \in [1{\uparrow}n]} \abs{t_i^\phi}
= n + 2 \sum\limits_{i \in [1{\uparrow}n]} \abs{w_i(\phi)}$.

Let $\sigma_{[1{\uparrow}n-1]} \subseteq \B_n$ be the subset determined 
by, for all $i \in [1{\uparrow}n-1]$ and all
$k \in [1{\uparrow}n]$, 
$$t_k^{\sigma_{i}} = \begin{cases}
 t_k &\text{if } k \in [1{\uparrow}i-1]\cup[i+2{\uparrow}n],\\
t_{i+1} &\text{if } k = i,\\
t_i^{t_{i+1}} &\text{if } k = i+1.
\end{cases}$$
In the literature, $\sigma_i$ is sometimes represented in  $2\times n$-matrix notation, for example, 
in the format
$$\sigma_i =  \begin{pmatrix}
t_1 &\ldots &t_{i-1} &t_i &t_{i+1} &t_{i+2} &\ldots &t_n\\
t_1 &\ldots &t_{i-1} &t_{i+1} &   t_i^{t_{i+1}} &t_{i+2} &\ldots &t_n
\end{pmatrix}.$$
We shall often find it convenient to compress the dots and say that
 $\sigma_i$ and $\overline \sigma_i$ are {\it determined by} the  expressions
\medskip

\centerline{
\begin{tabular}
{
>{$}r<{$} 
@{} 
>{$}l<{$} 
@{\hskip .4cm} 
>{$}c<{$} 
@{\hskip .8cm}  
>{$}c<{$}
@{\hskip .8cm}
>{$}r<{$} 
@{} 
>{$}l<{$}
}
\setlength\extrarowheight{3pt}
 &\hskip -.4cm\underline{\scriptstyle k \in [1{\uparrow}i-1]}  &   & &
 &\hskip -1cm\underline{\scriptstyle k \in [i+1{\uparrow}n]}
\\[.15cm]
(&t_k&  t_i& t_{i+1}&  t_k&)^{\sigma_i}\\  
=(&t_k & t_{i+1}& t_i^{t_{i+1}}&  t_k&),
\end{tabular}
\quad \text{and} \quad
\begin{tabular}
{
>{$}r<{$} 
@{} 
>{$}l<{$} 
@{\hskip .4cm} 
>{$}c<{$} 
@{\hskip .8cm}  
>{$}c<{$}
@{\hskip .8cm}
>{$}r<{$} 
@{} 
>{$}l<{$}
}
\setlength\extrarowheight{3pt}
 &\hskip -.4cm\underline{\scriptstyle k \in [1{\uparrow}i-1]}  &   & &
 &\hskip -1cm\underline{\scriptstyle k \in [i+1{\uparrow}n]}
\\[.15cm]
(&t_k&  t_i& t_{i+1}&  t_k&)^{\overline\sigma_i}\\  
=(&t_k & t_{i+1}^{\overline t_i}& t_i&  t_k&). \hfill\qed
\end{tabular}}

\bigskip

\end{notation}
We shall apply the following result in different situations.

\begin{lemma}[Artin~\cite{Artin25}]\label{lem:art2}
 Let $\phi \in \B_n$.  Let $\pi = \pi(\phi)$ and, 
for each $i \in[0{\uparrow}n]$, let  $u_i = u_i(\phi)$.
\begin{enumerate}[\normalfont (i).] 
\vskip-0.7cm \null
\item  Suppose that there exists some $i \in [1{\uparrow}n-1]$ such that $u_i \in ({\star} \overline t_{(i+1)^\pi})$.
Then $\norm{{\sigma_i\phi}} \le \norm{\phi} -2$; moreover, for each $j \in[1{\uparrow}i]$,
$t_j^{\sigma_i\phi}$ and $t_j^{\phi}$ both begin with the same element of
$t_{[1{\uparrow}n]} \cup \overline t_{[1{\uparrow}n]}$. 
\vskip-0.7cm \null
\item  Suppose that there exists some $i \in [1{\uparrow}n-1]$ such that
$ u_i \in ( \overline t_{i^\pi}{\star})$.  Then   $\norm{\overline\sigma_i\phi}
 \le \norm{\phi} -2$;  moreover, for each $j \in[1{\uparrow}i-1]$,
$t_j^{\overline\sigma_i\phi}$ and $t_j^{\phi}$ both begin with the same element of
$t_{[1{\uparrow}n]} \cup \overline t_{[1{\uparrow}n]}$.
\vskip-0.7cm \null
\item  Suppose that, for each $i \in [1{\uparrow}n-1]$, $u_i \not\in ( \overline t_{i^\pi}{\star}) \cup
({\star} \overline t_{(i+1)^\pi})$.  Then $\phi = 1$.
\end{enumerate}
\end{lemma}

\begin{proof} 

\medskip

(i).  There exists some  
$v \in \Sigma_{0,1,n}-({\star} t_{(i+1)^\pi})$ such that $u_i = v\overline t_{(i+1)^\pi}$. 
Since $w_i(\phi) = u_i w_{i+1}(\phi)$, we have 
\begin{equation}\label{eq:ws}
w_i(\phi)  = v\overline t_{(i+1)^\pi}w_{i+1}(\phi).
\end{equation}
Since  $v \not \in ({\star} t_{(i+1)^\pi})$ and $w_{i+1}(\phi) \not\in (t_{(i+1)^\pi} {\star})$,
there is no cancellation in the expression $t_{i^\pi}^{v\overline t_{(i+1)^\pi}w_{i+1}(\phi)}$ for 
$t_i^\phi$; hence 
\begin{equation}\label{eq:i}
t_i^\phi \in 
(\overline w_{i+1}(\phi) t_{(i+1)^{\pi}}{\star}) \text{ and }
\abs{t_i^\phi} = 1 + 2\abs{v} + 2 +  2\abs{w_{i+1}(\phi)}.
\end{equation}

For all $j \in [1{\uparrow}i-1]\cup[i+2{\uparrow}n]$, $t_j^{\sigma_i\phi} = t_j^\phi$; hence, 
$t_j^{\sigma_i\phi}$ has the same first letter as $t_j^\phi$, and, 
$\abs{t_j^{\sigma_i\phi}}  =  \abs{t_{j}^{\phi}}$.

Since $t_i^{\sigma_i\phi} = t_{i+1}^\phi\in (\overline w_{i+1}(\phi)t_{(i+1)^\pi}{\star})$, 
we see, from~\eqref{eq:i}, that
$t_i^{\sigma_i\phi}$ has the same first letter as $t_i^\phi$.  Also, 
 $\abs{t_i^{\sigma_i\phi}}  =  \abs{t_{i+1}^{\phi}}$.

By~\eqref{eq:ws}, $w_i(\phi) \overline w_{i+1}(\phi)t_{(i+1)^{\pi}} =  v$; hence
$$t_{i+1}^{\sigma_i\phi} = (t_i^{t_{i+1}})^{\phi} = (t_{i^\pi}^{w_i(\phi)})^{(t_{(i+1)^\pi}^{w_{i+1}(\phi)})}
 = t_{i^\pi}^{ v w_{i+1}(\phi)}.$$
Hence, $\abs{t_{i+1}^{\sigma_i\phi}} 
\le 1 + 2\abs{v} + 2\abs{w_{i+1}(\phi)} \overset{\eqref{eq:i}}{=} \abs{t_i^\phi} -2$.

It now follows that $\norm{\sigma_i\phi} \le \norm{\phi} -2$, and (i) is proved.

\medskip

(ii).  There exists some  
$v \in \Sigma_{0,1,n} - ( t_{i^\pi}{\star})$ such that $u_i = \overline t_{i^\pi}v$. 
Since $w_{i+1}(\phi) = \overline u_i w_{i}(\phi)$, we have 
\begin{equation}\label{eq:wss}
w_{i+1}(\phi)  = \overline v t_{i^\pi}w_{i}(\phi).
\end{equation}
Since  $\overline v \not \in ({\star} \overline t_{i^\pi})$ and $w_{i}(\phi) \not\in (\overline t_{i^\pi} {\star})$,
there is no cancellation in the expression $t_{(i+1)^\pi}^{\overline v  t_{i^\pi}w_{i}(\phi)}$ for 
$t_{i+1}^\phi$; hence 
\begin{equation}\label{eq:ii}
\abs{t_{i+1}^\phi} = 1 + 2\abs{\overline v} + 2 +  2\abs{w_{i}(\phi)}.
\end{equation}

For all $j \in [1{\uparrow}i-1]\cup[i+2{\uparrow}n]$, $t_j^{\overline\sigma_i\phi} = t_j^\phi$; hence, 
$t_j^{\overline\sigma_i\phi}$ has the same first letter as $t_j^\phi$, and, 
$\abs{t_j^{\overline\sigma_i\phi}}  =  \abs{t_{j}^{\phi}}$.

Since $t_{i+1}^{\overline\sigma_i\phi} = t_{i}^\phi$, 
we see  that  $\abs{t_{i+1}^{\overline \sigma_i\phi}}  =  \abs{t_{i}^{\phi}}$.

By~\eqref{eq:wss}, $w_{i+1}(\phi) \overline w_{i}(\phi)\overline t_{i^{\pi}} =  \overline v$; hence
$$t_{i}^{\overline\sigma_i\phi} = (t_{i+1}^{\overline t_{i}})^{\phi} 
= (t_{(i+1)^\pi}^{w_{i+1}(\phi)})^{(\overline t_{i^\pi}^{w_{i}(\phi)})}
 = t_{i^\pi}^{ \overline v w_{i}(\phi)}.$$
Hence, $\abs{t_{i}^{\overline \sigma_i\phi}} 
\le 1 + 2\abs{\overline v} + 2\abs{w_{i}(\phi)} \overset{\eqref{eq:ii}}{=} \abs{t_{i+1}^\phi} -2$.

It now follows that $\norm{\overline\sigma_i\phi} \le \norm{\phi} -2$, and (ii) is proved.

\medskip
(iii).  Since $u_0 = \overline w_1(\phi)\not \in ({\star} \overline t_1^\pi)$ and
$u_{n} = w_n(\phi) \not \in (\overline t_n^\pi {\star})$, we see that
 there is no cancellation anywhere in the expression
$u_0 \mathop{\Pi}\limits_{i\in[1{\uparrow}n]}(t_{i^{\pi}} u_i)$.  Hence,  
$$\abs{u_0 \mathop{\Pi}\limits_{i\in[1{\uparrow}n]}(t_{i^{\pi}} u_i)}
 = \textstyle\sum\limits_{i\in[0{\uparrow}n]} \abs{u_{i}} + n,
\text{ that is, } \textstyle\sum\limits_{i\in[0{\uparrow}n]} \abs{u_{i}} 
= \abs{u_0\mathop{\Pi}\limits_{i\in[1{\uparrow}n]}(t_{i^{\pi}} u_i)} -n.$$

Recall that
$u_0\mathop{\Pi}\limits_{i\in[1{\uparrow}n]}(t_{i^{\pi}} u_i)
 = \mathop{\Pi}\limits_{i\in[1{\uparrow}n]} (t_{i^\pi}^{w_{i}(\phi)})
= (\mathop{\Pi}\limits_{i\in[1{\uparrow}n]}t_i)^\phi = \mathop{\Pi}\limits_{i\in[1{\uparrow}n]}t_i.$
Hence $$\abs{u_0\mathop{\Pi}\limits_{i\in[1{\uparrow}n]}(t_{i^{\pi}} u_i)} = n
\text{ and } \textstyle\sum\limits_{i\in[0{\uparrow}n]} \abs{u_{i}}= n-n =0.$$ 
Hence, all the elements of $ u_{[0{\uparrow}n]}$ are trivial.  

For each $i \in [0{\uparrow}n+1]$,
 $w_i = \Pi u_{[i{\uparrow}n]}$; hence, all the elements of  $w_{[1{\uparrow}n]}$ are trivial. 
 Also,
$ \mathop{\Pi}\limits_{i\in[1{\uparrow}n]}t_{i^\pi}
= u_0\mathop{\Pi}\limits_{i\in[1{\uparrow}n]}(t_{i^{\pi}} u_i)=  \mathop{\Pi}\limits_{i\in[1{\uparrow}n]}t_i$.
Hence $\pi$ is trivial.  Thus $\phi = 1$.
\end{proof}

The following is then immediate.

\begin{proposition}[Artin~\cite{Artin25}]\label{prop:artin} 
 For each  $\phi \in \B_n$, either  $\phi = 1$, or there exists some 
$\sigma_i^\epsilon \in \sigma_{[1{\uparrow}n-1]} \cup
\overline \sigma_{[1{\uparrow}n-1]}$  such that
 $\norm{\sigma_i^\epsilon \phi} \le \norm{\phi} -2$.  
Hence, $\gen{\sigma_{[1{\uparrow}n-1]}} =\B_n$.
\hfill\qed
\end{proposition}

\begin{remarks}\label{rems:length} 
If  $w \in \Sigma_{0,1,n}$ has odd length, then $w^{\sigma_i}$ has odd length, and
$\abs{w^{\sigma_i}} \le 2\abs{w} + 1$,
with equality being achieved only if every odd letter of $w$ equals $t_{i+1}$.  
Similar statements hold with $\overline \sigma_i$ in place of $\sigma_i$.

Let $\phi \in \B_n$ and let $\abs{\phi}$ denote the minimum length of $\phi$ as a word in  
$\sigma_{[1{\uparrow}n-1]}$.
Thus,  $\abs{t_i^{\phi}} \le 2^{\abs{\phi}+1}-1$. Hence, $\norm{\phi} \le n 2^{\abs{\phi}+1}- n$. 
Proposition~\ref{prop:artin} gives an algorithm for writing  $\phi$ as a word in 
$\sigma_{[1{\uparrow}n-1]}$ of length at most $\frac{\norm{\phi}-n}{2}$, and we have now seen that 
 $\frac{\norm{\phi}-n}{2} \le \frac{n 2^{\abs{\phi}+1}-2n}{2} = n 2^{\abs{\phi}} - n.$
\hfill\qed
\end{remarks}

\section{Definition of  Artin groups}\label{sec:Artin}

\begin{definition} 
A {\it Coxeter diagram} $X$ consists of a set $V$
 together with a function $V \times V \to \naturals \cup\{\infty\},
\quad (x,y) \mapsto m_{x,y},$  such that, for all $x$, $y \in V$, $m_{x,x} = 0$ and $m_{x,y} = m_{y,x}$. 
The elements of $V$ are called the {\it vertices} of $X$, and,
for $(x,y) \in V \times V$, we say that $m_{x,y}$ is the {\it number of edges joining $x$ and $y$};
 we can depict $X$ in a natural way. 
We then define the {\it Artin group} of $X$, denoted $\Artin\gen{X}$, to be the group presented with generating set 
$V$ and 
relations saying that, for all $(x,y) \in V \times V$,
$$\begin{array}{rrlll}
&xy &= &yx &\text{ if \,\, } m_{x,y} = 0, \\
&xyx &= &yxy &\text{ if \,\, } m_{x,y} = 1, \\
&xyxy &=&yxyx &\text{ if \,\, } m_{x,y} = 2,\\
&&&\hskip -1.5cm\text{etc.}
\end{array}$$
Notice that if $m_{x,y} = \infty$, then no relation is imposed. Notice also that if $V$ is  
empty, then $\Artin\gen{X}$ is the trivial group. \hfill\qed
\end{definition}

\begin{notation} (i).  Let $A_{n}$ denote the Coxeter diagram
$$\,\,a_1  \ray  a_2   \ray  \cdots \ray a_{n-1}  \ray \,\,   a_{n}.$$ 
Clearly, $A_0$ is  empty.  We define $A_{-1}$ to be  empty also.  

Thus, in $A_{n}$, the vertex set   is $a_{[1{\uparrow}n]}$, and, 
for each $(i, j) \in [1{\uparrow}n]^2$, the number of edges joining 
$a_i$ to $a_j$ is $\begin{cases}
1 &\text{if } \abs{i-j} = 1,\\
0 &\text{if } \abs{i-j} \ne 1.
\end{cases}$

Thus, $\Artin\gen{A_n}$  has a presentation with generating set $a_{[1{\uparrow}n]}$ and relations saying that,
for each $(i, j) \in [1{\uparrow}n]^2$, 
$$\begin{array}{rrlll}
&a_ia_j &= &a_ja_i &\text{ if }  \abs{i-j} \ne 1,\\
&a_ia_ja_i &= &a_ja_ia_j &\text{ if } \abs{i-j} = 1.
\end{array}$$

(ii). Let $B_n$ denote the Coxeter diagram 
$$\,\,b_1  \ray  b_2   \ray  \cdots \ray b_{n-1}  \doubleray \,\,   b_n.$$
Here, the vertex set is $b_{[1{\uparrow}n]}$, and, 
for each $(i,j) \in [1{\uparrow}n]^2$, the number of edges joining 
$b_i$ to $b_j$ is $\begin{cases}
2 &\text{if } \{i,j\} = \{n-1,n\},\\
1 &\text{if } \abs{i-j} = 1 \text{ and } \{i,j\} \ne \{n-1,n\},\\
0 &\text{if } \abs{i-j} \ne 1.
\end{cases}$

(iii). For $n\ge 2$, let $D_n$ denote the Coxeter diagram 
\begin{align*} 
&\,\,d_1  \ray  d_2   \ray  \cdots \ray
 d_{n-3}\begin{aligned}
[b]d&\null_{n}\\[-2pt]\rule[-4pt]{0.4pt}{.6cm}\hskip2pt&\\[0pt]\ray d_n&\null_{-2}\ray
\end{aligned}    d_{n-1}.
\end{align*}
Here, the vertex set is $d_{[1{\uparrow}n]}$, and, 
for each $(i,j) \in [1{\uparrow}n]^2$, the number of edges joining 
$d_i$ to $d_j$ is $$\begin{cases}
1 &\text{if } \{i,j\} \in \{\{1,2\},  \{2,3\}, \ldots , \{n-2,n-1\},\{n-2,n\}\},\\
0 &\text{otherwise.}
\end{cases}$$
\vskip -0.8cm \hfill\qed
\end{notation}

\section{Artin's  presentation of $\B_n$}\label{sec:pres}

In this section, we verify Artin's result that
there exists an isomorphism 
$\gamma_{n} \colon \Artin\gen{A_{n-1}} \to \B_{n}$ determined by 
\begin{tabular}
{
>{$}r<{$}  
@{} 
>{$}c<{$} 
@{} 
>{$}l<{$}
}
\setlength\extrarowheight{3pt}
&&\hskip-.8cm\underline{\scriptstyle i \in [1{\uparrow}n-1]}
\\[.15cm]
(&a_i&)^{\gamma_{n}}\\  
=(&\sigma_i&)
\end{tabular}.
We express this by writing
$\B_n = \Artin\gen{\sigma_1 \ray \sigma_2 \ray \cdots \ray \sigma_{n-1}}.$

\begin{proposition}\label{prop:artin2}
There exists a homomorphism $\gamma_n\colon\Artin\gen{A_{n-1}}\to~\B_{n}$
determined by\begin{tabular}
{
>{$}r<{$}  
@{} 
>{$}c<{$} 
@{} 
>{$}l<{$}
}
\setlength\extrarowheight{3pt}
&&\hskip-.8cm\underline{\scriptstyle i \in [1{\uparrow}n-1]}
\\[.15cm]
(&a_i&)^{\gamma_{n}}\\  
=(&\sigma_i&)
\end{tabular}, and $\gamma_n$ is surjective.
\end{proposition}

\begin{proof} (a). Suppose that $1 \le i \le i+2\le j \le n-1$.  We have the following.

\medskip

\centerline{
\begin{tabular}
{
>{$}r<{$}  
@{\hskip0cm} 
>{$}l<{$} 
@{\hskip.5cm} 
>{$}c<{$} 
@{\hskip.9cm} 
>{$}c<{$} 
@{\hskip.5cm}  
>{$}c<{$} 
@{\hskip.5cm} 
>{$}c<{$} 
@{\hskip.9cm} 
>{$}c<{$}
@{\hskip .9cm}
>{$}r<{$} 
@{\hskip 0cm} 
>{$}l<{$}
}
\setlength\extrarowheight{100pt}
& \hskip-.4cm\underline{\scriptstyle k \in [1{\uparrow} i-1]}& &&\underline{\scriptstyle k\in[i+2{\uparrow} j-1]}& 
& &&\hskip-.9cm\underline{\scriptstyle 0k\in[j+2{\uparrow} n]}
\\[.15cm]
(&t_k  & t_i& t_{i+1}&  t_k &t_{j}&t_{j+1}& t_k&)^{\sigma_i\sigma_j}
\\[.08cm]=(&t_k& t_{i+1} & t_i^{t_{i+1}} &t_k 
& t_{j}& t_{j+1}& t_k&)^{\sigma_j}
\\[.08cm]=(&  t_k&  t_{i+1} 
& t_i^{t_{i+1}} &  t_k 
& t_{j+1}&  t_j^{t_{j+1}}&   t_k&)
\\[.1cm]=(&t_k&  t_{i} &  t_{i+1} &   t_k 
&t_{j+1}& t_j^{t_{j+1}}&  t_k&)^{\sigma_i}
\\[.08cm]=(&t_k  & t_i& t_{i+1}&  t_k &t_{j}&t_{j+1}& t_k&)^{\sigma_j\sigma_i}.
\end{tabular}}

\medskip

(b). Suppose that $1 \le i \le n-2$.  We have the following.

\medskip

\centerline{
\begin{tabular}
{
>{$}r<{$}  
@{\hskip0cm} 
>{$}l<{$} 
@{\hskip.5cm} 
>{$}c<{$} 
@{\hskip.9cm} 
>{$}c<{$}
@{\hskip.9cm} 
>{$}c<{$} 
@{\hskip1cm}  
>{$}r<{$}
@{\hskip 0cm} 
>{$}l<{$}
}
\setlength\extrarowheight{3pt}
&\hskip-.4cm\underline{\scriptstyle k \in [1{\uparrow} i-1]}& 
 &&&&\hskip-.8cm\underline{\scriptstyle  k\in[i+3{\uparrow} n]}
\\[.15cm]
(&t_k  & t_i& t_{i+1}&  t_{i+2} &t_k&)^{\sigma_i\sigma_{i+1}\sigma_i}
\\[0.08cm]=(&t_k  & t_{i+1}& t_i^{t_{i+1}}&  t_{i+2} & t_k& )^{\sigma_{i+1}\sigma_i}
\\[0.08cm]=(&t_k  & t_{i+2}&t_i^{t_{i+2}}& t_{i+1}^{t_{i+2}} &  t_k &)^{\sigma_i}
\\[0.08cm]=(&t_k  &  t_{i+2}& t_{i+1}^{t_{i+2}}& t_{i}^{t_{i+1}t_{i+2}} &  t_k &)
\\[0.1cm]=(&t_k  &  t_{i+1}&      t_{i+2}& t_{i}^{t_{i+1}t_{i+2}} &  t_k &)^{\sigma_{i+1}}  
\\[0.08cm]=(&t_k  &  t_{i}& t_{i+2}&  t_{i+1}^{t_{i+2}} &  t_k& )^{\sigma_i\sigma_{i+1}}
\\[0.08cm]=(&t_k  & t_i& t_{i+1}&  t_{i+2} &t_k &)^{\sigma_{i+1}\sigma_i\sigma_{i+1}}
\end{tabular}}

\medskip

Together, (a) and (b) show that there exists 
a homomorphism\newline $\gamma_n\colon\Artin\gen{A_{n-1}}\to~\B_{n}$
determined by\begin{tabular}
{
>{$}r<{$}  
@{} 
>{$}c<{$} 
@{} 
>{$}l<{$}
}
\setlength\extrarowheight{3pt}
&&\hskip-.8cm\underline{\scriptstyle i \in [1{\uparrow}n-1]}
\\[.15cm]
(&a_i&)^{\gamma_{n}}\\  
=(&\sigma_i&)
\end{tabular}.
By Artin's Proposition~\ref{prop:artin}, $\gen{\sigma_{[1{\uparrow}n-1]}} =\B_n$, and, hence, $\gamma_n$
is surjective.
\end{proof}

In the remainder of this section, we shall use induction on $n$ to show that the surjective homomorphism 
$ \gamma_n \colon \Artin\gen{A_{n-1}} \to \B_{n}$  
of Proposition~\ref{prop:artin2} is an isomorphism.  Notice that $\gamma_n$ endows 
 $\Artin\gen{A_{n-1}}$ with a canonical
action on~$\Sigma_{0,1,n}$.

\hskip-4.3pt The following is precisely~\cite[Proposition~1]{Manfredini} and, also, \cite[Proposition~2.1(2)]{CrispParis1}.

\begin{lemma}[Manfredini~\cite{Manfredini}]\label{lem:man} 
If $n\ge1$, then $$\Artin\gen{A_{n-1}} \ltimes \Sigma_{0,1,n}  = \Artin\gen{\,
  a_1 \ray  a_2   \ray \hskip-2pt \cdots  \hskip-2pt \ray  a_{n-1} \doubleray  \overline t_n\,} \simeq \Artin\gen{B_n}.$$
\end{lemma}

\begin{proof}  
 For $n = 1$, the result is clear.

For $n = 2$, we have the following. 
\begin{align*}
 &\Artin\gen{A_1}
               \ltimes \Sigma_{0,1,2} = \gen{  \{a_{1}\} \cup t_{[1{\uparrow}2]}  \mid t_{1}^{a_{1}} = t_{2}, \, 
t_2^{a_{1}} = \overline  t_2 t_{1} t_2}
\\ 
&
=   \gen{  a_{1},  t_2 \mid t_2^{a_{1}} = \overline  t_2 t_2^{\overline  a_{1}} t_2}
=    \gen{ a_{1},  t_2  \mid (\overline  a_{1} t_2)(a_{1})  
=   (\overline  t_2 a_{1})(t_2 \overline  a_{1} t_2)} 
\\&= \gen{ a_{1},  t_2  \mid  (a_{1})( \overline  t_2 a_{1} \overline  t_2 )
= (\overline  t_2 a_{1})(\overline  t_2 a_{1})} 
= \Artin\gen{\,\,  a_{1}   \doubleray \,\, \overline  t_2  \,\, }.
\end{align*}

From the case $n = 2$, we see that there exists a  homomorphism\newline
$
\mu\colon \Artin\gen{\, B_n\,}\to \Artin\gen{A_{n-1}\,} \ltimes \Sigma_{0,1,n}
\text{ determined by \begin{tabular}{ 
>{$}r<{$} 
@{\hskip0cm}    
>{$}l<{$}   
@{\hskip0.2cm}    
>{$}r<{$}   
@{\hskip0cm}  
>{$}l<{$}}
\setlength\extrarowheight{3pt}
\\[-0.6cm]&  \hskip -.5cm \underline{\scriptstyle  i\in[1{\uparrow}n-1]}&& 
\\ [.15cm]
(&b_i& b_n&)^{\mu}\\
=(&a_i&\overline t_n&)
\end{tabular}}.
$

For each $k \in [1{\uparrow}n]$,
let $\mathfrak{t}_k$ denote the element $\overline b_n^{\Pi \overline b_{[n-1{\downarrow}k]}} $ of 
$\Artin\gen{B_n}$.  For each $i \in [1{\uparrow}n-1]$ and $k \in [1{\uparrow}n]$,
let $\mathfrak{t}_k^{\overline \sigma_i}$ denote 
$\mathfrak{t}_k$, resp. $\mathfrak{t}_{i}$, resp. $\mathfrak{t}_{i+1}^{\overline{\mathfrak{t}}_{i}}$,
if $k \in [1{\uparrow}i-1]\cup[i+2{\uparrow}n]$, resp. $k = i+1$, resp. $k = i$.
We shall see that   $\mathfrak{t}_k^{\overline b_{i}} = \mathfrak{t}_k^{\overline \sigma_{i}}$; 
this immediately implies that
 there exists a  homomorphism \newline
$
\overline \mu\colon \Artin\gen{A_{n-1}\,}
               \ltimes  \Sigma_{0,1,n}
 \to \Artin\gen{B_n\,}$  determined by 
\begin{tabular}
{ 
>{$}r<{$} 
@{\hskip0cm}    
>{$}l<{$}   
@{\hskip0.5cm}    
>{$}r<{$}   
@{\hskip0cm}  
>{$}l<{$}}
\setlength\extrarowheight{3pt}
\\[-0.6cm]&  \hskip -.5cm \underline{\scriptstyle  i\in[1{\uparrow}n-1]}&& \hskip -.5cm 
\underline{\scriptstyle  k\in[1{\uparrow}n]}
\\ [.15cm]
(&a_i&t_k&)^{\overline \mu}\\
=(&b_i&\mathfrak{t}_k&)
\end{tabular},
which is then clearly inverse to $\mu$, and the result will be proved.

For each $m \in [n{\downarrow}1]$,  we shall show, by decreasing induction on $m$,
that, for each $k \in [n{\downarrow}m]$ and each $i \in [n-1{\downarrow}m]$,
 $\mathfrak{t}_k^{\overline b_{i}} = \mathfrak{t}_k^{\overline \sigma_{i}}$.  For $m = n$,
this is trivial, and, for $m = n-1$, it follows from the case 
$n=2$.  Suppose that $m \in  [n-2{\downarrow}1]$.
\begin{enumerate}[(a).]
\vskip -0.7cm \null
\item For each $k \in [n{\downarrow}m+1]$ and each $i \in [n-1{\downarrow}m+1]$, 
$\mathfrak{t}_k^{\overline b_{i}} = \mathfrak{t}_k^{\overline \sigma_{i}}$, by hypothesis.
\vskip -0.7cm \null
\item For each $k \in [n{\downarrow}m+2]$,  $\mathfrak{t}_k \in \gen{  b_{[n{\downarrow}m+2]}}$ and, hence,
$\mathfrak{t}_k^{\overline b_{m}} = \mathfrak{t}_k = \mathfrak{t}_k^{\overline \sigma_{m}}$.
\vskip -0.7cm  \null 
\item $\mathfrak{t}_{m+1}^{\overline b_{m}}  =  \overline b_n^{\Pi \overline 
b_{[n-1{\downarrow}m+1]}\overline b_m}  =   \mathfrak{t}_m = \mathfrak{t}_{m+1}^{\overline \sigma_{m}}$.
\vskip -0.7cm  \null 
\item For each $i \in [n-1{\downarrow}m+2]$,  $\mathfrak{t}_{m}^{\overline b_{i}} \overset{\rm(c)}{=} 
\mathfrak{t}_{m+1}^{\overline b_{m}\,\overline b_{i}} 
= \mathfrak{t}_{m+1}^{\overline b_{i}\, \overline b_{m}}  \overset{\rm(a)}{=} \mathfrak{t}_{m+1}^{\overline b_{m}}
 \overset{\rm(c)}{=} \mathfrak{t}_{m}
= \mathfrak{t}_{m}^{\overline \sigma_{i}}$.
\vskip -0.7cm  \null 
\item $\mathfrak{t}_m^{\overline b_{m+1}} \hskip -2pt \overset{\rm(c)}{=} \hskip -2pt \mathfrak{t}_{m+1}^{\overline b_{m}\overline b_{m+1}} 
\overset{\rm(a)}{=}  \mathfrak{t}_{m+2}^{\overline b_{m+1}\overline b_{m}\overline b_{m+1}} 
\hskip -2pt= \hskip -2pt \mathfrak{t}_{m+2}^{\overline b_{m}\overline b_{m+1}\overline b_{m}}  
\hskip -2pt\overset{\rm(b)}{=} \hskip -2pt\mathfrak{t}_{m+2}^{\overline b_{m+1}\overline b_{m}} 
 \overset{\rm(a)}{=} \mathfrak{t}_{m+1}^{\overline b_{m}} \overset{\rm(c)}{=} \mathfrak{t}_m
\hskip -2pt = \mathfrak{t}_{m}^{\overline \sigma_{m+1}}.$
\vskip -0.7cm  \null 
\item  $\mathfrak{t}_m^{\overline b_{m}} = \mathfrak{t}_{m}^{ b_{m+1}   b_{m}\overline  b_{m+1}\overline b_{m}\overline b_{m+1}}
\overset{\rm(e)}{=} \mathfrak{t}_m^{ b_{m}\overline  b_{m+1}\overline b_{m}\overline b_{m+1}}
\overset{\rm(c)}{=} \mathfrak{t}_{m+1}^{\overline  b_{m+1}\overline b_{m}\overline b_{m+1}}\newline
\overset{\rm(a)}{=}\hskip-2pt ( \mathfrak{t}_{m+1} \mathfrak{t}_{m+2} \overline{\mathfrak{t}}_{m+1})^{\overline b_{m}\overline b_{m+1}} 
\hskip-2pt\overset{\rm(c),(b),(c)}{=}\hskip-2pt ( \mathfrak{t}_{m} \mathfrak{t}_{m+2} \overline{\mathfrak{t}}_{m})^{\overline b_{m+1}} 
\hskip-2pt\overset{\rm(e),(a),(e)}{=} \hskip-2pt\mathfrak{t}_{m} \mathfrak{t}_{m+1}\overline{\mathfrak{t}}_{m}\hskip-2pt = \mathfrak{t}_m^{\overline \sigma_{m}}.$
\end{enumerate}

Now the result follows by induction.
\end{proof}

\medskip

We write $\Stab(\Artin\gen{A_{n}};[t_{n+1}])$ to denote the $\Artin\gen{A_{n}}$-stabi\-lizer of
the conjugacy class $[t_{n+1}]$ under the  $\Artin\gen{A_{n}}$-action on $\Sigma_{0,1,n+1}$. 
The Reidemeis\-ter-Schreier rewriting technique automatically 
gives a useful presentation of $\Stab(\Artin\gen{A_{n}};[t_{n+1}])$, 
but applying the technique can be rather tedious.  
Once the presentation has been found, we can verify it directly using the
van der Waerden trick, as in the following proof.

\begin{theorem}[Magnus~\cite{Magnus34}]\label{th:magnus1} If $n \ge 1$, then there exists a homomorphism

\vskip -0.8cm

$$\phi_n\colon\hskip -.5pt \Artin\gen{A_{n-1}}\ltimes \Sigma_{0,1,n} \to \Artin\gen{A_{n}} \text{ determined by }
 \begin{tabular}{ 
>{$}r<{$} 
@{\hskip0cm}    
>{$}l<{$}   
@{\hskip0.5cm}    
>{$}r<{$}   
@{\hskip0cm}  
>{$}l<{$}}
\setlength\extrarowheight{3pt}
\\[-.6cm] & \hskip -.6cm \underline{\scriptstyle  i\in[1{\uparrow}n-1]} & &
\\[.15cm]
(&a_i&t_{n}\,\,&)^{\phi_n}\\ 
=(&a_i &\overline a_{n}^{\,\,2}&).
\end{tabular}$$

\vskip -0.5cm

\noindent
Moreover, the following hold.
\begin{enumerate}[\normalfont (i).]
\vskip-0.7cm \null
\item\label{it:1} $\phi_n$ is injective. 
\vskip-0.7cm \null
\item\label{it:3} For each  $i \in [1{\uparrow}n]$,
 $t_i^{\phi_n} = \overline a_i^{\,\,2\Pi  a_{[i+1{\uparrow}n]}}$
in $\Artin\gen{A_n}$.
\vskip-0.7cm \null
\item\label{it:2} The image of $\phi_n$ is $\Stab(\Artin\gen{A_{n}};[t_{n+1}])$.
\end{enumerate}
\end{theorem}

\begin{proof} 
Let us write $G =  \Artin\gen{A_{n}}$ and 
$H = \Artin\gen{A_{n-1}}\ltimes \Sigma_{0,1,n}$.  

In $G$, $$(a_{n-1}a_{n}^2a_{n-1})^{a_{n}} \hskip-3.5pt
= \hskip-3.5pt(\overline a_{n} a_{n-1}a_{n})(a_{n}a_{n-1}a_{n}) \hskip-3.5pt = \hskip-3.5pt
(a_{n-1}a_{n} \overline a_{n-1})(a_{n-1}a_{n}a_{n-1}) \hskip-3.5pt = \hskip-3.5pt a_{n-1}a_{n}^2a_{n-1},$$
 and, hence,  $a_{n-1} a_{n}^2 a_{n-1} a_{n}^2 = a_{n}^2a_{n-1}a_{n}^2a_{n-1}$.  
By Lem\-ma~\ref{lem:man}, $H  \simeq \Artin\gen{B_n}$, and we see that
there exist a homomorphism 
$\phi_n\colon H \to G$ determined by 
$ \begin{tabular}{ 
>{$}r<{$} 
@{\hskip0cm}    
>{$}l<{$}   
@{\hskip0.5cm}    
>{$}r<{$}   
@{\hskip0cm}  
>{$}l<{$}}
\setlength\extrarowheight{3pt}
\\[-.6cm] & \hskip -.6cm \underline{\scriptstyle  i\in[1{\uparrow}n-1]} & &
\\[.15cm]
(&a_i&\overline t_{n}\,\,&)^{\phi_n}\\ 
=(&a_i & a_{n}^{\,\,2}&)
\end{tabular}$.

Let $v_{([1{\uparrow}n+2])} =([1{\uparrow}n+1])$, thought of as a generic $(n+1)$-tuple, and consider the 
free left $H$-set $H \times v_{[1{\uparrow}n+1]}$, with left $H$-transversal $v_{[1{\uparrow}n+1]}$. 

We construct a right $G$-action on $H \times v_{[1{\uparrow}n+1]}$ such that 
$H \times v_{[1{\uparrow}n+1]}$ becomes an $(H,G)$-bi-set.
For each $i \in [1{\uparrow}n]$, we define the right action 
of the generator $a_i \in G$  on the left $H$-set $H \times v_{[1{\uparrow}n+1]}$, 
by specifying the action on  the given left
$H$-transversal as follows.

\smallskip

\centerline{
\begin{tabular}
{
>{$}r<{$}  
@{} 
>{$}r<{$} 
@{\hskip .9cm} 
>{$}r<{$} 
@{\hskip.9cm} 
>{$}r<{$} 
@{\hskip .9cm}
>{$}r<{$} 
@{\hskip0cm} 
>{$}l<{$}
}
\setlength\extrarowheight{3pt}
&\hskip-3cm\underline{\scriptstyle k \in [1{\uparrow} i-1]}&&&&\hskip-1.2cm\underline{\scriptstyle k\in[i+2{\uparrow} n+1]}
\\ [.15cm]
(&v_k&  v_{i\phantom{+1}}& v_{i+1}&  v_k&)a_i\\
=(& a_{i-1}v_k & v_{i+1}& \overline t_i v_{i\phantom{+1}}&  a_iv_k&).
\end{tabular}}

\bigskip

\noindent We now verify that the relations of $G$ are respected.

(a). Suppose that $1 \le i < i+2\le j \le n$.  We have the following.
\medskip

\centerline{
\begin{tabular}
{
>{$}r<{$}  
@{} 
>{$}r<{$} 
@{\hskip .5cm} 
>{$}r<{$} 
@{\hskip .5cm} 
>{$}r<{$} 
@{\hskip .5cm} 
>{$}r<{$} 
@{\hskip .5cm} 
>{$}r<{$} 
@{\hskip.5cm} 
>{$}r<{$} 
@{\hskip 0cm}
>{$}r<{$} 
@{\hskip0cm} 
>{$}l<{$}
}
\setlength\extrarowheight{3pt}
&\underline{\scriptstyle k \in [1{\uparrow} i-1]}&&&\underline{\scriptstyle  k\in[i+2{\uparrow} j-1]} & 
&&\underline{\scriptstyle  k\in[j+2{\uparrow} n+1]}
\\[.15cm]
(&v_k  & v_i& v_{i+1}&  v_k &v_{j}&v_{j+1}& v_k &)a_ia_j
\\=(& a_{i-1}v_k& v_{i+1} & \overline t_i v_i &   a_iv_k 
& a_iv_{j}& a_i v_{j+1}&  a_iv_k &)a_j
\\=(& a_{i-1} a_{j-1}v_k&  a_{j-1}v_{i+1} 
& \overline t_i  a_{j-1}v_i &   a_i a_{j-1}v_k 
& a_i v_{j+1}&  a_i \overline t_j v_{j}&  a_i a_j v_k &)
\\=(& a_{j-1} a_{i-1}v_k&  a_{j-1}v_{i+1} 
&   a_{j-1}\overline t_iv_i &   a_{j-1} a_iv_k 
& a_i v_{j+1}&  \overline t_j  a_i v_{j}&  a_j a_i v_k &)
\\ =(& a_{j-1}v_k&  a_{j-1}v_{i} &  a_{j-1}v_{i+1} &   a_{j-1}v_k 
&v_{j+1}&\overline t_j v_{j}&  a_jv_k &)a_i
\\=
(&v_k  & v_i& v_{i+1}&  v_k &v_{j}&v_{j+1}& v_k &)a_ja_i.
\end{tabular}}

\bigskip

(b). Suppose that $1 \le i \le n-1$.  We have the following.

\medskip

\centerline{
\begin{tabular}
{
>{$}r<{$}  
@{} 
>{$}r<{$} 
@{\hskip .5cm} 
>{$}r<{$} 
@{\hskip .5cm} 
>{$}r<{$} 
@{\hskip .5cm} 
>{$}r<{$} 
@{\hskip.5cm} 
>{$}r<{$} 
@{\hskip0cm} 
>{$}l<{$}
}
\setlength\extrarowheight{3pt}
&\underline{\scriptstyle k \in [1{\uparrow} i-1]}&&&&\underline{\scriptstyle  k\in[i+3{\uparrow} n+1]}
\\[.15cm]
(&v_k  & v_i& v_{i+1}&  v_{i+2} &v_k &)a_ia_{i+1}a_i\\
=(& a_{i-1}v_k  & v_{i+1}& \overline t_iv_{i}&   a_iv_{i+2} & a_iv_k &)a_{i+1}a_i\\
=(& a_{i-1} a_iv_k  & v_{i+2}& \overline t_i a_i v_{i}& 
  a_i \overline t_{i+1}v_{i+1} & a_i a_{i+1}v_k &)a_i\\
=(& a_{i-1} a_i a_{i-1} v_k  &  a_iv_{i+2}& \overline t_i a_i  v_{i+1}& 
  a_i \overline t_{i+1}\overline t_i v_{i} & a_i a_{i+1} a_iv_k &)
\\
=(& a_{i} a_{i-1} a_{i} v_k  &  a_iv_{i+2}&  a_i \overline t_{i+1} v_{i+1}& 
\overline t_{i+1}\overline t_i   a_i v_{i} & a_{i+1} a_{i} a_{i+1}v_k &)
\\ 
=(& a_i a_{i-1}v_k  &  a_i v_{i+1}&    a_i v_{i+2}& 
 \overline t_{i+1}\overline t_i v_{i} & a_{i+1} a_iv_k &)a_{i+1}
\\ 
=(& a_{i}v_k  &  a_i v_{i}& v_{i+2}&  \overline t_{i+1}v_{i+1} & a_{i+1}v_k &)a_ia_{i+1}
\\
=(&v_k  & v_i& v_{i+1}&  v_{i+2} &v_k &)a_{i+1}a_ia_{i+1}.
\end{tabular}}

\bigskip

Now (a) and (b) prove that the relations of $G$ are respected.  Hence,
 we have a right $G$-action on $H \times v_{[1{\uparrow}n+1]}$.

Notice that $v_{n+1} \overline t_n^{\,\, \phi_n} = v_{n+1}a_{n}^2 = \overline t_n v_n a_{n} = \overline t_n v_{n+1}$.
Also, for each $i \in [1{\uparrow}n-1]$, $v_{n+1}a_i^{\overline \phi_n} = v_{n+1}a_{i} = a_{i}v_{n+1}$.
It follows that, for each $h \in H$, $v_{n+1}h^{\phi_n} = hv_{n+1}$. Hence, $\phi_n$ is injective.  
This proves (i).

Recall that $G = \Artin\gen{A_{n}}$.  

Let $i \in [1{\uparrow}n]$.

We shall show by decreasing induction on $i$ that
\begin{equation}\label{eq:eq}
a_{n}^{\Pi \overline a_{[n-1{\downarrow}i]}} = a_i^{\Pi a_{[i+1{\uparrow}n]}}.
\end{equation}
If $i=n$, then~\eqref{eq:eq} holds.  Now suppose that $i \ge 2$, and that~\eqref{eq:eq} holds.
Conjugating~\eqref{eq:eq} by $\overline a_{i-1}$ yields
\begin{align*}
a_{n}^{\Pi \overline a_{[n-1{\downarrow}i-1]}} = a_{i}^{\Pi a_{[i+1{\uparrow}n]}\overline a_{i-1}}
=  a_{i}^{\overline a_{i-1} \Pi a_{[i+1{\uparrow}n]}} = a_{i-1}^{ a_{i} \Pi a_{[i+1{\uparrow}n]}}
= a_{i-1}^{\Pi a_{[i{\uparrow}n]}}.
\end{align*}
By induction,~\eqref{eq:eq} holds.

Now $\overline t_i^{\phi_n} = (\overline t_n^{\Pi \overline a_{[n-1{\downarrow}i]}})^{\phi_n}
=   a_n^{2\Pi \overline a_{[n-1{\downarrow}i]}}  
\overset{\eqref{eq:eq}}{=}  a_i^{2\Pi  a_{[i+1{\uparrow}n]}}$.  This proves (ii).
Also, 
 $ \overline t_i^{\phi_n} \Pi \overline a_{[n{\downarrow}i]} = \Pi \overline a_{[n{\downarrow}i+1]}a_i$.

If  $k \in [1{\uparrow}i-1]$, then
 $$a_i^{\Pi a_{[k{\uparrow}n]}} 
= a_i^{\Pi a_{[k{\uparrow}i-2]}\Pi a_{[i-1{\uparrow}i]} \Pi a_{[i+1{\uparrow}n]}} 
= a_i^{\Pi a_{[i-1{\uparrow}i]} \Pi a_{[i+1{\uparrow}n]}}
= a_{i-1}^{\Pi a_{[i+1{\uparrow}n]}}= a_{i-1}.$$
Hence, $a_{i-1}\Pi \overline a_{[n{\downarrow}k]} = \Pi \overline a_{[n{\downarrow}k]}a_i$.

Let $\psi_n$ denote the map of sets   $$\psi_n \colon H \times v_{[1{\uparrow}n+1]} \to G, 
\quad hv_k \mapsto h^{\phi_n} \Pi \overline a_{[n{\downarrow}k]} \text{ for all }
hv_k = (h,v_k) \in H \times v_{[1{\uparrow}n+1]}.$$

Hence, for each $h \in H$, we have the following, in $G$.

\medskip

\centerline{
\begin{tabular}
{
>{$}r<{$}  
@{} 
>{$}r<{$}  
@{} 
>{$}r<{$}  
@{}
>{$}l<{$} 
@{\hskip.9cm} 
>{$}r<{$} 
@{} 
>{$}l<{$}  
@{\hskip.9cm}
>{$}r<{$}  
@{}
>{$}l<{$}  
@{\hskip.9cm}
>{$}r<{$}  
@{}
>{$}l<{$}  
@{} 
>{$}l<{$}
}
\setlength\extrarowheight{3pt}
&&&\hskip -.5cm \underline{\scriptstyle k \in\scriptstyle [1{\uparrow} i-1]}& && &
&&&\hskip-1.9cm\underline{\scriptstyle k\in\scriptstyle[i+2{\uparrow} n+1]} 
\\[.15cm]
(&h\phantom{\scriptstyle \phi_n}(&&v_k&&v_{i}&&v_{i+1}&&v_k&))^{\psi_n}a_i\\
=(&h^{\phi_n}(&&\Pi \overline a_{[n{\downarrow}k]}&&\Pi \overline a_{[n{\downarrow}i]}&&
 \Pi \overline a_{[n{\downarrow}i+1]}&&  
 \Pi \overline a_{[n{\downarrow}k]}&))a_i\\
=(&h^{\phi_n}(&a_{i-1}&\Pi \overline a_{[n{\downarrow}k]} &&  \Pi \overline a_{[n{\downarrow}i+1]}& 
\overline t_i^{\phi_n}  &\Pi \overline a_{[n{\downarrow}i]} &a_i&\Pi \overline a_{[n{\downarrow}k]}&))\\
=(&h\phantom{\scriptstyle \phi_n}(&a_{i-1}&v_k &&v_{i+1}& \overline t_i&v_{i}&  a_i&v_k &))^{\psi_n}\\
= (&h\phantom{\scriptstyle \phi_n}(&&v_k&& v_{i}&&v_{i+1}&&  v_k&)a_i)^{\psi_n}.
\end{tabular}}

\bigskip

\noindent This proves that $\psi_n$ is a map of right $G$-sets, and, hence, $\psi_n$ must be surjective.
Thus, $G = \bigcup\limits_{k\in[1{\uparrow}n+1]} H^{\phi_n} v_k^{\psi_n}$, and, hence,
the index of $H^{\phi_n}$ in $G$ is at most $n+1$.

Consider the action of $G$ on the set of conjugacy classes $\{[t_k]\}_{k \in [1{\uparrow}n+1]}$
in $\Sigma_{0,1,n+1}$.  For any  $i \in [1{\uparrow}n]$,    $a_i$ acts as the transposition $([t_i],[t_{i+1}])$.
In particular, the index of $\Stab(G;[t_{n+1}])$ in $G$ is $n+1$.
Also,  the elements of $a_{[1{\uparrow}n-1]} \cup \{a_n^2\}$ fix $[t_{n+1}]$, and, hence,
$H^{\phi_n} \le \Stab(G;[t_{n+1}])$.  By comparing indices, we see that
$H^{\phi_n} = \Stab(G;[t_{n+1}])$.  This proves (iii).
\end{proof}

\begin{theorem}[Artin] \label{th:artpres}
$\B_n = \Artin\gen{\sigma_1 \ray \sigma_2 \ray \cdots \ray \sigma_{n-1}}.$
\end{theorem}

\begin{proof} This is trivial for $n\le 1$.  Hence,  we may assume that $n \ge 1$ and that
the homomorphism
$\gamma_{n}\colon\Artin\gen{A_{n-1}} \to \B_{n}$, of Proposition~\ref{prop:artin2},
determined by\begin{tabular}
{
>{$}r<{$}  
@{} 
>{$}c<{$} 
@{} 
>{$}l<{$}
}
\setlength\extrarowheight{3pt}
&&\hskip-.8cm\underline{\scriptstyle i \in [1{\uparrow}n-1]}
\\[.15cm]
(&a_i&)^{\gamma_{n}}\\  
=(&\sigma_i&)
\end{tabular}
is an isomorphism; and it  remains to show that the surjective homomorphism
$\gamma_{n+1} \colon\Artin\gen{A_{n}} \to \B_{n+1}$
is  injective. 

Consider an element $w$ of the kernel of $\gamma_{n+1}$.  In particular, 
$w$ fixes $t_{n+1}$ in the $\Artin\gen{A_{n}}$-action on $\Sigma_{0,1,n+1}$.
By Theorem~\ref{th:magnus1}(iii), $w$ lies in the image of the homomorphism
$\phi_n\colon\hskip -.5pt \Artin\gen{A_{n-1}}\ltimes \Sigma_{0,1,n} \to \Artin\gen{A_{n}}$  determined by  
 \begin{tabular}{ 
>{$}r<{$} 
@{\hskip0cm}    
>{$}l<{$}   
@{\hskip0.5cm}    
>{$}r<{$}   
@{\hskip0cm}  
>{$}l<{$}}
\setlength\extrarowheight{3pt}
\\[-.6cm] & \hskip -.6cm \underline{\scriptstyle  i\in[1{\uparrow}n-1]} & &
\\[.15cm]
(&a_i&t_{n}\,\,&)^{\phi_n}\\ 
=(&a_i &\overline a_{n}^{\,\,2}&)
\end{tabular}\hskip -.25cm .
Thus, we may express $w$ as a product of two words $$w = w_1(a_{([1{\uparrow} n-1])})w_2(t^{\phi_n}_{([1{\uparrow} n])}).$$
Now,
\begin{equation}\label{eq:power}
\text{in $\Artin\gen{A_{n}}\ltimes \Sigma_{0,1,n+1}$, }
t_{n+1} = t_{n+1}^w = t_{n+1}^{w_1(a_{([1{\uparrow} n-1])})w_2(t^{\phi_n}_{([1{\uparrow} n])})}
= t_{n+1}^{w_2(t^{\phi_n}_{([1{\uparrow} n])})}.
\end{equation}

Consider  the homomorphism
$\phi_{n+1}\colon\hskip -.5pt \Artin\gen{A_{n}}\ltimes \Sigma_{0,1,n+1} \to \Artin\gen{A_{n+1}}$
  determined by 
 \begin{tabular}{ 
>{$}r<{$} 
@{\hskip0cm}    
>{$}l<{$}   
@{\hskip0.5cm}    
>{$}r<{$}   
@{\hskip0cm}  
>{$}l<{$}}
\setlength\extrarowheight{3pt}
\\[-.6cm] & \hskip -.6cm \underline{\scriptstyle  i\in[1{\uparrow}n]} & &
\\[.15cm]
(&a_i&t_{n+1}\,\,&)^{\phi_{n+1}}\\ 
=(&a_i &\overline a_{n+1}^{\,\,2}&)
\end{tabular}.
Let $i \in [1{\uparrow}n]$.  By Theorem~\ref{th:magnus1}(ii),
\begin{align*} (t^{\phi_n}_{i})^{\phi_{n+1}a_{n+1}} &=
(\overline a_i^{\,\,2\Pi  a_{[i+1{\uparrow}n]}})^{\phi_{n+1}a_{n+1}}
= (\overline a_i^{\,\,2\Pi  a_{[i+1{\uparrow}n]}})^{a_{n+1}}\\
&= (\overline a_i^{\,\,2\Pi  a_{[i+1{\uparrow}n+1]}}) = (t_i)^{\phi_{n+1}},\\
(t_{n+1})^{\phi_{n+1}a_{n+1}} &=
(\overline a_{n+1}^2)^{a_{n+1}} = \overline a_{n+1} ^2 = (t_{n+1})^{\phi_{n+1}}.
\end{align*}
Thus the two $(n+1)$-tuples
 $(t^{\phi_n}_{([1{\uparrow}n])}, t_{n+1})$ and $t_{([1{\uparrow}n+1])}$ 
for $\Artin\gen{A_{n}}\ltimes \Sigma_{0,1,n+1}$ become conjugate in $\Artin\gen{A_{n+1}}$
under $\phi_{n+1}$.
By Theorem~\ref{th:magnus1}(i), $\phi_{n+1}$ is injective. 
Since $t_{([1{\uparrow}n+1])}$ freely generates a free subgroup of $\Artin\gen{A_{n}}\ltimes \Sigma_{0,1,n+1}$,
we see that  $(t^{\phi_n}_{([1{\uparrow}n])}, t_{n+1})$ also freely generates  
a free subgroup of $\Artin\gen{A_{n}}\ltimes \Sigma_{0,1,n+1}$.  From~\eqref{eq:power},
 we see that $w_2$ must be the trivial word.

  Hence, $w = w_1(a_{([1{\uparrow}n-1])})$ in $\Artin\gen{A_{n}}$.
By the induction hypothesis, $w_1(a_{([1{\uparrow}n-1])}) = 1$ in $\Artin\gen{A_{n-1}}$.
Hence $w = 1$ in $\Artin\gen{A_{n}}$.  

Now the result holds by induction.
\end{proof}

Combining Lemma~\ref{lem:man}, Theorem~\ref{th:magnus1} and Theorem~\ref{th:artpres},  we have the following.

\begin{corollary}[Artin-Magnus-Manfredini]\label{cor:Manfred} If $n \ge 1$, then
\begin{align*} \B_n&=  \Artin\langle \sigma_1 
  \ray  \sigma_2   \ray  
\hskip-3pt \cdots \ray \sigma_{n-2}  \ray \sigma_{n-1}\rangle \simeq \Artin\gen{A_{n-1}},\\
\Stab(\B_n; [t_n])  &= \Artin\langle \sigma_1 
  \ray  \sigma_2   \ray  
\hskip-3.7pt \cdots \ray \sigma_{n-2}  \doubleray \sigma_{n-1}^2\rangle \simeq \Artin\gen{B_{n-1}},\\
\B_{n-1} \ltimes \Sigma_{0,1,n-1}&= \Artin\langle \sigma_1 
 \ray  \sigma_2   \ray  
\hskip-3pt \cdots \ray \sigma_{n-2}\doubleray \overline t_{n-1}\rangle  \simeq \Artin\gen{B_{n-1}}.\\ &\hskip 10.6cm \hfill\qed
\end{align*}
\end{corollary}

\begin{history}
In 1925, Artin~\cite{Artin25} found the above presentation of $\B_n$ by an intuitive
topological argument but, later, in~\cite{Artin47}, he indicated that there 
were difficulties that could be corrected.
In 1934, Magnus~\cite{Magnus34} gave an algebraic proof that the relations suffice.
In 1945, Mark\-ov~\cite{Markov} gave a similar algebraic proof.
In 1947, Bohnenblust~\cite{Bohnenblust}  gave a similar algebraic proof; 
in 1948, Chow~\cite{Chow} simplified the latter proof.
All these algebraic proofs of the sufficiency of the relations involve
the Reidemeister-Schreier rewriting process for the subgroup of index $n$.

Larue~\cite{Larue94}  
gave a new algebraic proof of the sufficiency of the relations, by using
the Dehornoy-Larue trichotomy~\cite{Dehornoy3} for braid groups. 
We shall proceed in the opposite direction.
Proofs of the trichotomy for $\Artin\gen{A_{n-1}}$ tend to be more 
difficult than proofs that $\Out^+_{0,1,n} = \Artin\gen{A_{n-1}}$, and
we shall now see that Artin's generation argument  easily gives the 
trichotomy for $\Out^+_{0,1,n}$.  
\hfill\qed
\end{history}

\section{The Dehornoy-Larue trichotomy}\label{sec:trich}

\begin{definitions}\label{defs:preorder}  Let $\phi \in \B_n$.

We say that $\phi$ is {\it $\sigma_1$-neutral}  if $\phi$ lies in the
subgroup of $\B_n$ generated by  $\sigma_{[2{\uparrow}n-1]}$.

We say that $\phi$ is {\it $\sigma_1$-positive} if $n \ge 2$ and $\phi$ can be expressed as the
product of a finite sequence of elements of
 $\sigma_{[1{\uparrow}n-1]} \cup \overline \sigma_{[2{\uparrow}n-1]}$
such that at least one term of the sequence is~$\sigma_1$. 
We say that $\phi$ is {\it $\sigma$-positive} if $n \ge 2$ and, for some $i \in[1{\uparrow}n-1]$, 
$\phi$ can be expressed as the
product of a finite sequence of elements of
 $\sigma_{[i{\uparrow}n-1]} \cup \overline \sigma_{[i+1{\uparrow}n-1]}$
such that at least one term of the sequence is~$\sigma_i$.  

We say that $\phi$ is {\it $\sigma_1$-negative} if $\overline \phi$ is $\sigma_1$-positive, 
that is, $n\ge 2$ and $\phi$ can be expressed as the
product of a finite sequence of elements of 
$\sigma_{[2{\uparrow}n-1]} \cup \overline \sigma_{[1{\uparrow}n-1]}$
such that at least one term of the sequence is~$\overline  \sigma_1$.  

If $\phi$ satisfies exactly one of the properties of being  $\sigma_1$-neutral, $\sigma_1$-positive 
$\sigma_1$-negative, we say that $\phi$ {\it satisfies the $\sigma_1$-trichotomy}.
\hfill\qed
\end{definitions}

\begin{history}
View $\Artin\gen{A_{n}}$ as a subgroup of $\Artin\gen{A_{n+1}}$ in a natural way,
and let $\Artin\gen{A_{\infty}}$ denote the union of the resulting chain; thus  
$\Artin\gen{A_{\infty}} = \gen{a_{[1{\uparrow}\infty[}}$.
Dehornoy~\cite[Theorem~6]{Dehornoy3} 
gave a one-sided ordering of $\Artin\gen{A_{\infty}}$;
the positive semigroup for this ordering
 is the set of `$a$-positive' elements of $\Artin\gen{A_{\infty}}$. 

Let $\phi \in \B_n$.
By replacing $\phi$ with $\overline \phi$ if necessary, we can apply
Dehornoy's result to deduce that there exists some $n' \ge n$ such that 
$\phi$ is $\sigma$-negative in $\B_{n'}$, or $\phi = 1$.  
Larue~\cite{LarueThesis94} showed that this implies that
$t_1^\phi \in ({t_1}{\star})$,  and that this in turn implies that
$\phi$ can be expressed as the
product of a finite sequence, of length at most $ \abs{\phi} + \frac{1}{4}n^2 3^{\abs{\phi}}$,
 of elements of
$\sigma_{[2{\uparrow}n-1]} \cup \overline \sigma_{[1{\uparrow}n-1]}$.   
Thus, every element of $\B_n$ satisfies 
the $\sigma_1$-trichotomy.  Larue's work is surveyed in~\cite[Chapter~5]{DDRW02}. 
Topological versions of these results can be found in~\cite{Fenn} and \cite[Chapter 6]{DDRW02}. 

We shall give elementary direct proofs of the foregoing results and replace 
Larue's bound $ \abs{\phi} + \frac{1}{4}n^2 3^{\abs{\phi}}$ with the much smaller bound $n 2^{\abs{\phi}}-n$.
Larue's proof contains interesting information that we shall rework in the Appendix.
\hfill\qed
\end{history}

Part (iii) of the following seems to be new.

\begin{lemma}\label{lem:new}
 Let $n \ge 1$ and let $\phi $ be an element of $ \B_n$ such that $t_1^\phi \in (t_1{\star})$.  
Let $\pi = \pi(\phi)$ and, 
for each $i \in[1{\uparrow}n]$, let  $u_i = u_i(\phi)$.
\begin{enumerate}[\normalfont (i).] 
\vskip-0.7cm \null
\item  Suppose that there exists some $i \in [1{\uparrow}n-1]$ such that
 $u_i \in ({\star} \overline t_{(i+1)^\pi})$.  Then $\norm{{\sigma_i\phi}} \le \norm{\phi} -2$ and
$t_1^{\sigma_i\phi} \in (t_1{\star})$; moreover, if $t_1^\phi = t_1$, then $i \in[2{\uparrow}n-1]$.
\vskip-0.7cm \null
\item  Suppose that there exists some $i \in [2{\uparrow}n-1]$ such that
$ u_i \in ( \overline t_{i^\pi}{\star})$.  Then   $\norm{\overline\sigma_i\phi}
 \le \norm{\phi} -2$ and $t_1^{\overline\sigma_i\phi} \in (t_1{\star})$. 
\vskip-0.7cm \null
\item  Suppose that, for each $i \in [1{\uparrow}n-1]$, $u_i \not\in ({\star} \overline t_{(i+1)^\pi})$ and,
for each $i \in [2{\uparrow}n-1]$, $u_i \not\in ( \overline t_{i^\pi}{\star})$.  Then $\phi = 1$.
\end{enumerate}
\end{lemma}

\begin{proof}  For each $i \in [0{\uparrow}n+1]$, let $w_i = w_i(\phi)$.

\medskip

(i).  The first part follows from Artin's Lemma~\ref{lem:art2}(i).
Notice that, if $t_1^\phi = t_1$, then $w_1 = 1$ and $u_1
= \overline w_2 \not \in  ({\star} \overline t_{2^\pi})$.

\medskip 

(ii) follows from Lemma~\ref{lem:art2}(ii).

\medskip

(iii). 
Recall that $u_0\mathop{\textstyle\prod}\limits_{i\in[1{\uparrow}n]}(t_{i^{\pi}} u_i)
 = \mathop{\textstyle\prod}\limits_{i\in[1{\uparrow}n]} (t_{i^\pi}^{w_{i}})
= (\mathop{\textstyle\prod}\limits_{i\in[1{\uparrow}n]}t_i)^\phi = \mathop{\textstyle\prod}\limits_{i\in[1{\uparrow}n]}t_i.$ 
Hence, \linebreak $u_0t_{1^\pi}u_1\mathop{\textstyle\prod}\limits_{i\in[2{\uparrow}n]}(t_{i^{\pi}} u_i)
= t_1\mathop{\textstyle\prod}\limits_{i\in[2{\uparrow}n]}t_i$, and, hence,
$u_1\mathop{\textstyle\prod}\limits_{i\in[2{\uparrow}n]}(t_{i^{\pi}} u_i)
= \overline t_{1^\pi} \overline u_0 t_1 \mathop{\textstyle\prod}\limits_{i\in[2{\uparrow}n]}t_i.$  
Since $u_n = w_n \not \in (\overline t_{n^\pi}{\star} )$, the hypotheses imply that there is no cancellation anywhere in
the expression
$u_1 \mathop{\textstyle\prod}\limits_{i\in[2{\uparrow}n]}(t_{i^{\pi}} u_i)$.  Hence, 
\begin{equation}\label{eq:2}
\textstyle\sum\limits_{i\in[1{\uparrow}n]} \abs{u_{i}} +  n-1 
=
\abs{u_1 \mathop{\textstyle\prod}\limits_{i\in[2{\uparrow}n]}(t_{i^{\pi}} u_i)} 
= \abs{\overline t_{1^\pi} \overline u_0 t_1 \mathop{\textstyle\prod}\limits_{i\in[2{\uparrow}n]}t_i} \le
\abs{\overline t_{1^\pi} \overline u_0 t_1} + n-1.
\end{equation}

Since $t_{1^\pi}^{\overline u_0} = t_{1^\pi}^{w_1} = t_1^\phi \in (t_1 {\star})$, we see that 
$u_0t_{1^\pi}  \in  ({{t_1}{\star}})$, and
\begin{equation}\label{eq:1}
\abs{\overline t_{1}u_0 t_{1^\pi}} =  -1 + \abs{u_0 t_{1^\pi}} \le -1 + \abs{u_0} + 1 = \abs{u_0}.
\end{equation} 

Since $\mathop{\textstyle\prod}\limits_{i\in[0{\uparrow}n]} u_i = w_0\overline w_{n+1} = 1$, we see that 
\begin{equation}\label{eq:3}
\mathop{\textstyle\prod}\limits_{i\in[1{\uparrow}n]} u_i = \overline u_0 = w_1 \not\in ({\overline t_{1^\pi}}{\star}).
\end{equation}
Now,
$\textstyle\sum\limits_{i\in[1{\uparrow}n]} \abs{u_{i}}
\overset{\eqref{eq:2}}{\le} 
\abs{\overline t_{1^\pi}\overline u_0 t_1} 
\overset{\eqref{eq:1}}{\le}  
\abs{\overline u_0} 
 \overset{\eqref{eq:3}}{=} 
\abs{\mathop{\textstyle\prod}\limits_{i\in[1{\uparrow}n]} u_i}.
$
 Therefore, there is no cancellation in $\mathop{\textstyle\prod}\limits_{i\in[1{\uparrow}n]} u_i$, and, by~\eqref{eq:3},
 $u_1 \not\in (\overline t_{1^\pi}{\star})$.  By Lemma~\ref{lem:art2}(iii), $\phi~=~1$.
\end{proof}

As in Remarks~\ref{rems:length}, we deduce the following 
from Lemma~\ref{lem:new} by induction on $\norm{\phi}$. 

\begin{corollary}[Larue~\cite{LarueThesis94}]\label{cor:larue} Let $n \ge 1$ and let $\phi \in \B_n$. 
\begin{enumerate}[\normalfont (i).]
\vskip-0.8cm \null
\item
Suppose that $t_1^\phi \in (t_1{\star})$.  Then
$\phi$ is $\sigma_1$-negative or $\sigma_1$-neutral.  In more detail, 
$\phi$ can be expressed as the
product of a sequence, of 
length at most   $n 2^{\abs{\phi}}-n$,  of elements of 
$\sigma_{[2{\uparrow}n-1]} \cup \overline \sigma_{[1{\uparrow}n-1]}$. 
\vskip-0.8cm \null
\item Moreover, $\phi$ is $\sigma_1$-neutral if and only if $t_1^\phi = t_1$. \hfill\qed
\end{enumerate}
\end{corollary}

\begin{notation}\label{not:alpha}  For each $i \in [1{\uparrow}n-1]$, let $\sigma_i'$
 and $\sigma_i''$ be the automorphisms of $\Sigma_{0,1,n}$ determined by

\medskip

\centerline{
\begin{tabular}
{
>{$}r<{$}  
@{} 
>{$}l<{$} 
@{\hskip.5cm} 
>{$}c<{$} 
@{\hskip .7cm}
>{$}r<{$} 
@{\hskip0cm} 
>{$}l<{$}
}
\setlength\extrarowheight{3pt}
&\hskip-.5cm\underline{\scriptstyle k \in [1{\uparrow}i]}  &   
   & & \hskip-1cm\underline{\scriptstyle k \in [i+2{\uparrow}n]} 
\\[0.15cm]
(&t_k&  t_{i+1}&  t_k&)^{\sigma_i'}\\  
=(&t_{k} & t_{i+1}^{t_{i}}&  t_k &),
\end{tabular}
\qquad\qquad\hskip-8pt
\begin{tabular}
{
>{$}r<{$}  
@{} 
>{$}l<{$} 
@{\hskip.5cm} 
>{$}c<{$} 
@{\hskip.9cm} 
>{$}c<{$} 
@{\hskip .5cm}
>{$}r<{$} 
@{\hskip0cm} 
>{$}l<{$}
}
\setlength\extrarowheight{3pt}
&\hskip-.5cm\underline{\scriptstyle k \in [1{\uparrow}{i-1}]}&&   &&\hskip-.6cm
\underline{\scriptstyle k \in [i+1{\uparrow}n]}
\\[.15cm]
(&t_k&t_i & t_{i+1}&  t_k&)^{\sigma_i''}\\  
=(&t_k &t_{i+1}& t_{i}&  t_k &).
\end{tabular}}

\bigskip

\noindent
Then $\sigma_i = \sigma_i' \sigma_i''$.
 The normal form in $t_{[1{\uparrow}n]}$ factorizes 
into an alternating product with factors which are normal forms of non-trivial elements of $\gen{t_{[i{\uparrow}i+1]}}$ 
 alternating with factors which are normal forms of non-trivial elements of 
$\gen{t_{[1{\uparrow}i-1]\cup[i+2{\uparrow}n]}}$.
On $\gen{t_{[i{\uparrow}i+1]}}$, $\sigma_i'$ acts as conjugation by $t_i$, while $\sigma_i''$ 
interchanges the two free generators.
On $\gen{t_{[1{\uparrow}i-1]\cup[i+2{\uparrow}n]}}$, $\sigma_i'$  and $\sigma_i''$  act as the identity map. 
\hfill\qed
\end{notation}

The next result gives three trichotomies, called (a), (b) and (c), which hold for elements of $\B_n$.
Attribution is not sharply defined, but it is reasonable to
 attribute (b) to Dehornoy~\cite{Dehornoy3},
and (a) and (c) to Larue~\cite{LarueThesis94}.

\begin{theorem}[Dehornoy-Larue~\cite{Dehornoy3},~\cite{LarueThesis94}]\label{th:D-L} Let $n \ge 1$, let
 $\phi \in \B_n$ and consider the following nine conditions.

\medskip

\centerline{
\begin{tabular}{r  @{}  l r   @{}  l  r  @{} l}
\setlength\extrarowheight{3pt}
 & \normalfont(a1). $t_1^{\phi} = t_1$.
 & & \normalfont(a2). $t_1^{\phi} \in (t_1{\star}) -\{t_1\}$.
&&\normalfont(a3). $t_1^{\phi} \not\in (t_1{\star})$.
\\  \\
& {\normalfont(b1).} $\phi$ is $\sigma_1$-neutral.
 & &{\normalfont(b2).} $\phi$ is $\sigma_1$-negative.
&& {\normalfont(b3).} $\phi$ is $\sigma_1$-positive.
\\   \\
&\normalfont(c1). $(t_1 {\star})^{\phi} = (t_1{\star})$
 & & \normalfont(c2).  $(t_1 {\star})^{\phi} \subset (t_1{\star})$.
&& \normalfont(c3).  $(t_1 {\star})^{\phi} \supset (t_1{\star})$.
\end{tabular}}

\medskip

 Then: \normalfont(a1) $ \Leftrightarrow $ (b1)  $\Leftrightarrow$  (c1);
(a2) $ \Leftrightarrow $  (b2)  $\Leftrightarrow$ (c2); (a3) $ \Leftrightarrow $  (b3)  $\Leftrightarrow$  (c3).

 \it Exactly one of \normalfont(b1),  (b2),  (b3), \it holds; that is,
$\phi$ satisfies the $\sigma_1$-trichotomy in $\B_n$.
\end{theorem}

\begin{proof}  (a1) $\Leftrightarrow$ (b1) by Corollary~\ref{cor:larue}(ii). 
We shall use (a1) and (b1) interchangeably in the remainder of the proof.  

(b1) $\Rightarrow$ (c1).  If $\phi$ is $\sigma_1$-neutral, then so is $\overline \phi$.
It follows that $(t_1{\star})^\phi \subseteq (t_1{\star})$
and $(t_1{\star})^{\overline\phi} \subseteq (t_1{\star})$.  Thus,  $(t_1{\star})^\phi = (t_1{\star})$.

(a2)  $\Rightarrow$ (b2).  If (a2) holds, then
Corollary~\ref{cor:larue}(i) shows that (b1) or (b2) holds.  Since
(a1) fails, (b1) fails.  Thus (b2) holds. 

(b2)  $\Rightarrow$ (c2).  
Using  Notation~\ref{not:alpha}, we see that
\begin{equation*}
(t_1{\star})^{\overline \sigma_1} = (t_1{\star})^{\overline\sigma_1''\overline \sigma_1'} 
= (t_2{\star})^{\overline \sigma_1'} \subseteq  (t_1t_2{\star}) \subset (t_1{\star}). 
\end{equation*}
Since the composition of injective self-maps of $(t_1{\star})$ 
can be bijective only if all the factors are bijective, we see that (b2)  $\Rightarrow$ (c2). 

(a3) $\Rightarrow$ (b3).
We translate into algebra the crucial reflection
argument of~\cite[Corollary 5.2.4]{DDRW02}.

Suppose that (a3) holds.

With Notation~\ref{not:basic}, let $w_1 = w_1(\phi)$ and $\pi = \pi(\phi)$.
Then $\overline w_1 t_{1^\pi} w_1 =  t_1^\phi  \not \in (t_1{\star})$.  
It follows that $\overline w_1 t_{1^\pi} \not \in (t_1{\star})$.
Hence, $\overline w_1 \, \overline t_{1^\pi} \not \in (t_1 {\star})$.
Hence,\newline $\overline t_1^\phi =  \overline w_1 \, \overline t_{1^\pi} w_1 \not \in (t_1{\star}) \cup \{1\}$.
  On conjugating by $t_1$, we see that
$\overline t_1^{  \phi  t_1} \in (\overline t_1{\star})$.

Let $\zeta$  be the automorphism of $\Sigma_{0,1,n}$ determined by
\begin{tabular}
{
>{$}r<{$}  
@{} 
>{$}c<{$} 
@{\hskip0cm} 
>{$}l<{$}
}
\setlength\extrarowheight{3pt}
&\underline{\scriptstyle k \in [1{\uparrow}n]}&
\\[.15cm]
(&t_k&)^{\zeta}\\  
=(&\overline t_k^{\,\,\Pi \overline t_{[k-1{\downarrow}1]}} &)
\end{tabular}.
For each $k \in [1{\uparrow}n]$, $(\Pi t_{[1{\uparrow}k]})^\zeta =  \Pi \overline t_{[k{\downarrow}1]}.$
It follows that $\zeta^2 = 1$.  Notice that $\zeta$ belongs to $\Out^{-}_{0,1,n}:= \Out_{0,1,n}-\Out^{+}_{0,1,n}$.  Also,
\begin{tabular}
{
>{$}r<{$}  
@{} 
>{$}l<{$} 
@{\hskip .5cm}
>{$}c<{$} 
@{\hskip0cm} 
>{$}l<{$}
}
\setlength\extrarowheight{3pt}
\\[-.5cm]&&&\hskip-1.8cm\underline{\scriptstyle k \in [2{\uparrow}n]}
\\[.15cm]
(&t_1&t_k\phantom{\scriptstyle{\,\,\Pi\overline t_{[k-1{\downarrow}2]}}}&)^{\overline t_1\zeta}\\  
=(&\overline t_1 &\overline t_k^{\,\,\Pi\overline t_{[k-1{\downarrow}2]}} &)
\end{tabular}.
Hence,
\begin{align*}
t_1^{\phi^\zeta} = t_1^{\zeta\phi\zeta} = \overline t_1^{\phi t_1 \overline t_1\zeta}   \in 
&(\overline t_1{\star})^{\overline t_1\zeta} \subseteq (t_1{\star}).
\end{align*}
By Corollary~\ref{cor:larue}(i), $\phi^\zeta$ can be expressed as the product of a finite sequence of elements 
of $\sigma_{[2{\uparrow}n-1]} \cup \overline \sigma_{[1{\uparrow}n-1]}$.  
It is not difficult to check that, for each $i \in [1{\uparrow}n-1]$, 
$\sigma_i^\zeta = \overline \sigma_i$ in $\Out_{0,1,n}$.
Hence $\phi^{\zeta^2} (= \phi)$ can be expressed as the product of a finite sequence of elements of
$\sigma^{\zeta}_{[2{\uparrow}n-1]} \cup \overline \sigma^{\zeta}_{[1{\uparrow}n-1]}
(= \overline\sigma_{[2{\uparrow}n-1]} \cup  \sigma_{[1{\uparrow}n-1]})$.  
Hence, (b3) or (b1) holds.  Since (a3) holds, (a1) fails, and (b1) fails.  Thus (b3) holds.

(b3) $\Rightarrow$ (c3).  If $\phi$ is $\sigma_1$-positive, then
 $\overline \phi$ is $\sigma_1$-negative, 
and, by  (b2) $\Rightarrow$ (c2), $(t_1{\star})^{\overline\phi} \subset (t_1{\star})$ and, hence, 
$(t_1{\star}) \subset  (t_1{\star})^\phi $.

(c1) $\Rightarrow$ (a1).  Suppose that (a1) fails.  Then (a2) or (a3) holds.  Hence (c2) or (c3) holds.
Hence (c1) fails.

(c2) $\Rightarrow$ (a2) and  (c3) $\Rightarrow$ (a3) are proved similarly. 

Thus the desired equivalences hold.

Since exactly one of (a1), (a2), (a3) holds, exactly one of (b1), (b2), (b3) holds.
\end{proof}

The following gives {\it  the Dehornoy right-ordering of $\B_n$}; recall the
definition of $\sigma$-positive from Definitions~\ref{defs:preorder}.

\begin{theorem}\label{th:order} For each $\phi \in \B_n$ exactly one of the following holds: $\phi = 1$;
 $\phi$ is $\sigma$-positive; $\overline \phi$ is $\sigma$-positive.  The set of $\sigma$-positive elements of
$\B_n$ is the positive cone of a right-ordering of $\B_n$.  
\end{theorem}

\begin{proof} Suppose that $\phi \ne 1$.

Let $i$ be the largest element of $[1{\uparrow}n-1]$ such that $\phi \in \gen{\sigma_{[i,n-1]}}$.
The natural subscript-shifting isomorphism from $\gen{t_{[i{\uparrow}n]}}$
to $\Sigma_{0,1,n-i+1}$ induces an isomorphism from $\gen{\sigma_{[i{\uparrow}n-1]}}$ to
$B_{n-i+1}$. Notice that $\phi$  is mapped to an 
element of $B_{n-i+1}$ which is not $\sigma_1$-neutral;
by Theorem~\ref{th:D-L}, this image is $\sigma_1$-positive or $\sigma_1$-negative but not both.
Hence exactly one of $\phi$, $\overline \phi$ is $\sigma$-positive.

It is easy to see that the product of two $\sigma$-positive elements of $\B_n$ is
$\sigma$-positive.  

Hence the set of $\sigma$-positive elements of $\B_n$ is the
positive cone for a right-ordering of $\B_n$, the Dehornoy right-ordering. 
\end{proof}

\section{Ends, right-orderings and squarefreeness}\label{sec:ends}

\begin{review}  A (reduced) {\it end} of $\Sigma_{0,1,n}$ is a function
$$[1{\uparrow}\infty[ \,\,\,\,\to \,\,\,\, t_{[1{\uparrow}n]} \cup \overline t_{[1{\uparrow}n]}, \quad i \mapsto a_i,$$
such that, for each $i \in [1{\uparrow}\infty[ $, $a_{i+1} \ne \overline a_i$.  The function is then
represented as a right-infinite reduced product, $a_1a_2\cdots$ or $\Pi a_{[1{\uparrow}\infty[}$.

We denote the set of ends of $\Sigma_{0,1,n}$ by $\mathfrak{E}(\Sigma_{0,1,n})$,
or simply by $\mathfrak{E}$ if there is no risk of confusion.  

An element of $\Sigma_{0,1,n} \cup \mathfrak{E}(\Sigma_{0,1,n})$ is said to be {\it squarefree} 
if, in its reduced expression, no two consecutive terms are equal; for example:
 $(t_1t_2)^\infty$ is a squarefree end;  $t_1t_2t_2t_3$ is a non-squarefree word.  

For each $w \in\Sigma_{0,1,n}$,
we define the {\it shadow} of $w$ in $\mathfrak E$ to be
$$(w{\blacktriangleleft}) := \{\Pi a_{[1{\uparrow}\infty[} \in \mathfrak E \mid
\text{ $\Pi a_{[1{\uparrow}\abs{w}]} = w$} \}.$$
Thus, for example, $(1{\blacktriangleleft}) = \mathfrak{E}$. 

We shall now give  $\mathfrak{E}$ an ordering, $<$.  
The first step is, for each $w \in \Sigma_{0,1,n}$, to assign 
an ordering, $<$, to a partition of $(w{\blacktriangleleft})$ into $2n$ or $2n-1$ subsets, depending 
as $w=1$ or $w\ne 1$,  as follows.  We set
\begin{align*}
&(t_1{\blacktriangleleft}) < (\overline t_1{\blacktriangleleft}) < 
(t_2{\blacktriangleleft}) < (\overline t_2{\blacktriangleleft}) < \cdots 
< (t_n{\blacktriangleleft}) < (\overline t_n{\blacktriangleleft}). 
\intertext{If $i \in [1{\uparrow}n]$ and 
$w \in  ({\star}\overline t_i)$, then we set}
(w \overline t_i{\blacktriangleleft}) &<  (w t_{i+1}{\blacktriangleleft})
 < (w \overline t_{i+1}{\blacktriangleleft}) < 
(w t_{i+2}{\blacktriangleleft}) < (w \overline t_{i+2}{\blacktriangleleft}) < \cdots\\
&\cdots <(w t_{n}{\blacktriangleleft}) < (w \overline t_{n}{\blacktriangleleft}) < 
(w t_{1}{\blacktriangleleft}) < (w \overline t_{1}{\blacktriangleleft}) <
(w t_2{\blacktriangleleft}) < \cdots \\
&\cdots < (w t_{i-1}{\blacktriangleleft}) < (w \overline t_{i-1}{\blacktriangleleft}) . 
\intertext{If $i \in [1{\uparrow}n]$ and 
$w \in ({\star} t_i)$, then we set}
(w   t_{i+1}{\blacktriangleleft}) &< (w   \overline t_{i+1}{\blacktriangleleft}) < 
(w  t_{i+2}{\blacktriangleleft}) < (w   \overline t_{i+2}{\blacktriangleleft}) < \cdots\\
&\cdots <(w  t_{n}{\blacktriangleleft}) < (w \overline t_{n}{\blacktriangleleft}) < 
(w   t_{1}{\blacktriangleleft}) < (w   \overline t_{1}{\blacktriangleleft}) <
(w  t_2{\blacktriangleleft}) < \cdots \\
&\cdots < (w  t_{i-1}{\blacktriangleleft}) < (w \overline t_{i-1}{\blacktriangleleft})
 < (w  t_i{\blacktriangleleft}). 
\end{align*}
Hence, for each $w \in \Sigma_{0,1,n}$, we have
an ordering $<$ of a partition of $(w{\blacktriangleleft})$ into $2n$ or $2n-1$ subsets.

If $\Pi a_{[1{\uparrow}\infty[}$ and $\Pi b_{[1{\uparrow}\infty[}$ 
are two different (reduced) ends, then there exists $i \in \naturals$ such that
$\Pi a_{[1{\uparrow}i]} = \Pi b_{[1{\uparrow}i]}$ in $\Sigma_{0,1,n}$, and
$a_{i+1} \ne b_{i+1}$ in $t_{[1{\uparrow}n]} \cup \overline t_{[1{\uparrow}n]}$.
Let $w = \Pi a_{[1{\uparrow}i]} = \Pi b_{[1{\uparrow}i]}$. Then 
$\Pi a_{[1{\uparrow}\infty[}$ and $\Pi b_{[1{\uparrow}\infty[}$ lie in $(w{\blacktriangleleft})$,
but lie in different elements of the partition of
$(w{\blacktriangleleft})$ into $2n$ or $2n-1$ subsets.  
We then order $\Pi a_{[1{\uparrow}\infty[}$ and $\Pi b_{[1{\uparrow}\infty[}$ according to the
order of the elements of the partition of
$(w{\blacktriangleleft})$ that they belong to. 
This completes the definition of the ordering $<$ of $\mathfrak{E}$.

We remark that the smallest element of $\mathfrak E$ is $\overline z_1^\infty = 
(\Pi t_{[1{\uparrow}n]})^\infty$ and the largest element of $\mathfrak E$ is
$z_1^\infty = (\Pi \overline t_{[n{\downarrow}1]})^\infty$.
\hfill\qed
\end{review}

\begin{review}  Following Nielsen-Thurston~\cite{Cooper},~\cite{ShortWiest},
we now define the action of $\B_n$ on $\mathfrak{E}(\Sigma_{0,1,n})$ and show that it respects the ordering;
our treatment will be quite elementary compared to the usual approaches.

We assume that $n \ge 2$, and we first define the action of $\sigma_1$ on $\mathfrak{E}$.

Consider any reduced end $\mathfrak{e} \in \mathfrak{E}$.  There is then a unique factorization
$\mathfrak{e} = \Pi  w_{[1{\uparrow}i]}$ or $\mathfrak{e} = \Pi  w_{[1{\uparrow}\infty[}$,
where, in the former case,  $w_{([1{\uparrow}i-1])}$  is a finite sequence of 
non-trivial words, and $w_i$ is a reduced end,  and, in the latter case,
$w_{([1{\uparrow}\infty[)}$ is  an infinite sequence of non-trivial words, 
and in both cases, the $w_j$ alternate 
between elements of  $\gen{t_{[1{\uparrow}2]}} \cup \mathfrak{E}(\gen{t_{[1{\uparrow}2]}})$, 
and elements  of $\gen{t_{[3{\uparrow}n]}}\cup \mathfrak{E}(\gen{t_{[3{\uparrow}n]}})$. 
We shall express this factorization as $\mathfrak{e} = [w_1][w_2] \cdots$.

Recall, from Notation~\ref{not:alpha}, that we have the factorization $\sigma_1 = \sigma_1' \sigma_1''$.
On $\gen{t_{[1{\uparrow}2]}} \cup \mathfrak{E}(\gen{t_{[1{\uparrow}2]}})$, 
$\sigma_1'$ acts as conjugation by $t_1$, while $\sigma_1''$ interchanges the two free generators.
On $\gen{t_{[3{\uparrow}n]}}$, $\sigma_1'$ and $\sigma_1''$ act as the identity map. 
This completes the description of the action of $\sigma_1'$, $\sigma_1''$ and $\sigma_1$ on $\mathfrak{E}$. 

It is not difficult to show that, for any reduced ends $\Pi a_{[1{\uparrow}\infty[}$ and $\Pi b_{[1{\uparrow}\infty[}$,
if  $(\Pi a_{[1{\uparrow}\infty[})^{\sigma_1} = \Pi b_{[1{\uparrow}\infty[}$, then
for all $i,j \in \naturals$, if $j \ge 2i$, then  
$(\Pi a_{[1{\uparrow}j]})^{\sigma_1} \in (\Pi b_{[1{\uparrow}i]}{\star})$.
Thus,  $(\Pi a_{[1{\uparrow}\infty[})^\sigma_1$ is the limit of $(\Pi a_{[1{\uparrow}j]})^{\sigma_1}$ as $j$ tends to
$\infty$.

It is clear that $\sigma_1'$, $\sigma_1''$ and, hence, $\sigma_1$ act bijectively on $\mathfrak{E}$.
Hence we have the action of $\overline \sigma_1$ on $\mathfrak{E}$.  It is then not difficult to verify that
we have an action of $\B_n$ on $\mathfrak{E}$.  

We next show that $\sigma_1$ respects the ordering of $\mathfrak{E}$.  We do this by considering all the
ways that two reduced ends can be compared, and the resulting effect of $\sigma_1'$ and $\sigma_1$.
We represent the information in tables.  
In all of the following, we understand that $t_1a$, $\overline t_1b$, $t_2c$, and $\overline t_2 d$
are reduced expressions for elements of 
$\gen{t_{[1{\uparrow}2]}} \cup \mathfrak{E}(\gen{t_{[1{\uparrow}2]}})$, and 
 $b \ne 1$.  Since  $a$ does not begin with~$\overline t_1$,  
 $a^{\sigma_1''} t_2$ begins with $t_1$ or $\overline t_1$ or $t_2$. 
We make the convention that $\Sigma_{0,1,n}$ acts trivially on the right on~$\mathfrak{E}$.

\medskip

\centerline{
\begin{tabular}{>{$}r<{$} @{}  >{$}l<{$}||>{$}r<{$} @{} >{$}l<{$} ||>{$}r<{$} @{}  >{$}l<{$}}
\setlength\extrarowheight{3pt}
 & \hskip -1cm (\cdots ][w t_1{\blacktriangleleft})
 & & \hskip -1cm  (\cdots ][w t_1{\blacktriangleleft})^{\sigma_1'}
&& \hskip -1cm  (\cdots ][w t_1{\blacktriangleleft})^{\sigma_1}
\\ 
 \hhline{------}   &&& & &
\\ 
[-.45cm] \hhline{------} &&& & &
\\  
[-.3cm] \cdots ][w t_1\phantom{]}& \phantom{[}t_2c][\cdots 
& \cdots ][(\overline t_1 w) t_1 \phantom{]}&\phantom{[}t_2(ct_1)][\cdots 
&\cdots ][(\overline t_2 w^{\sigma_1''}) t_2 \phantom{]}&\phantom{[}t_1(c^{\sigma_1''} t_2)][\cdots
\\
\cdots][w t_1\phantom{]}&\phantom{[}\overline t_2d][\cdots 
& \cdots ][(\overline t_1 w) t_1 \phantom{]}&\phantom{[}\overline t_2(dt_1)][\cdots 
& \cdots ][(\overline t_2 w^{\sigma_1''}) t_2 \phantom{]}&\phantom{[}\overline t_1(d^{\sigma_1''} t_2)][\cdots
\\
\cdots][w t_1]&[t_3{\uparrow}\overline t_n \cdots  
& \cdots ][(\overline t_1 w) t_1 \phantom{]}&\phantom{[} t_1][t_3{\uparrow}\overline t_n \cdots 
& \cdots ][(\overline t_2 w^{\sigma_1''}) t_2 \phantom{]}&\phantom{[} t_2][t_3{\uparrow}\overline t_n \cdots
\\
 \cdots ][w t_1\phantom{]}& \phantom{[}t_1a][\cdots 
&\cdots ][(\overline t_1 w) t_1 \phantom{]}&\phantom{[}t_1(at_1)][\cdots 
&\cdots ][(\overline t_2 w^{\sigma_1''}) t_2\phantom{]}&\phantom{[}t_2(a^{\sigma_1''} t_2)][\cdots
\end{tabular}}
\medskip
\noindent Here, the case $w=1$ does not present any problems.  

\medskip

\centerline{
\begin{tabular}{>{$}r<{$} @{}  >{$}l<{$}||>{$}r<{$} @{} >{$}l<{$} ||>{$}r<{$} @{}  >{$}l<{$}}
\setlength\extrarowheight{3pt}
 & \hskip -1cm (\cdots ][w \overline t_1{\blacktriangleleft})
 & & \hskip -1cm  (\cdots ][w \overline t_1{\blacktriangleleft})^{\sigma_1'}
&& \hskip -1cm  (\cdots ][w \overline t_1{\blacktriangleleft})^{\sigma_1}
\\ 
 \hhline{------}   &&& & &
\\ 
[-.45cm] \hhline{------} &&& & &
\\  
[-.3cm] 
 \cdots ][w \overline  t_1\phantom{]}& \phantom{[}\overline  t_1b][\cdots 
 &\cdots ][(\overline t_1 w)\phantom{]} &\phantom{[}\overline  t_1\overline t_1(bt_1)][\cdots 
&\cdots ][(\overline t_2 w^{\sigma_1''})\phantom{]}  & \phantom{[}\overline  t_2\overline  t_2(b^{\sigma_1''} t_2)][\cdots
\\
 \cdots ][w \overline  t_1\phantom{]}& \phantom{[}\overline  t_1][t_3{\uparrow}\overline t_n\cdots 
 &\cdots ][(\overline t_1 w)\phantom{]} &\phantom{[}\overline  t_1][t_3{\uparrow}\overline t_n\cdots 
&\cdots ][(\overline t_2 w^{\sigma_1''})\phantom{]}  &\phantom{[} \overline  t_2][t_3{\uparrow}\overline t_n\cdots
\\
 \cdots ][w \overline t_1\phantom{]}& \phantom{[}t_2c][\cdots 
 &\cdots ][(\overline  t_1 w)\phantom{]}  &\phantom{[}\overline t_1 t_2(ct_1)][\cdots 
&\cdots ][(\overline t_2 w^{\sigma_1''})\phantom{]} &\phantom{[}\overline  t_2 t_1(c^{\sigma_1''} t_2)][\cdots
\\
\cdots][w \overline t_1\phantom{]}&\phantom{[}\overline t_2d][\cdots 
& \cdots ][(\overline t_1 w)\phantom{]} &\phantom{[}\overline  t_1 \overline t_2(dt_1)][\cdots 
& \cdots ][(\overline t_2 w^{\sigma_1''})\phantom{]}  &\phantom{[}\overline  t_2 \overline t_1(d^{\sigma_1''} t_2)][\cdots
\\
\cdots][w \overline  t_1]&[t_3{\uparrow}\overline t_n \cdots  
&  \cdots ][(\overline t_1 w)]&[t_3{\uparrow}\overline t_n \cdots 
& \cdots ][(\overline t_2 w^{\sigma_1''})]&[t_3{\uparrow}\overline t_n \cdots
\end{tabular}}
\medskip
\noindent
Here,  $w$ does not end with $t_1$, and, hence,
 $(\overline t_2 w^{\sigma_1''})$ ends with $t_1$, $\overline t_1$ or $\overline t_2$.
\medskip

\centerline{
\begin{tabular}{>{$}r<{$} @{}  >{$}l<{$}||>{$}r<{$} @{} >{$}l<{$} ||>{$}r<{$} @{}  >{$}l<{$}}
\setlength\extrarowheight{3pt}
 & \hskip -1cm (\cdots ][w t_2{\blacktriangleleft})
 & & \hskip -1cm  (\cdots ][w t_2{\blacktriangleleft})^{\sigma_1'}
&& \hskip -1cm  (\cdots ][w t_2{\blacktriangleleft})^{\sigma_1}
\\ 
 \hhline{------}   &&& & &
\\ 
[-.45cm] \hhline{------} &&& & &
\\  
[-.3cm] 
\cdots][w t_2]&[t_3{\uparrow}\overline t_n \cdots  
&  \cdots ][(\overline t_1 w)t_2\phantom{]}&\phantom{[}t_1][t_3{\uparrow}\overline t_n \cdots 
& \cdots ][(\overline t_2 w^{\sigma_1''})t_1\phantom{]}&\phantom{[}t_2][t_3{\uparrow}\overline t_n \cdots
\\
 \cdots ][w t_2\phantom{]}&\phantom{[}   t_1a][\cdots 
 &\cdots ][(\overline t_1 w)t_2\phantom{]}&\phantom{[}t_1(at_1)][\cdots 
&\cdots ][(\overline t_2 w^{\sigma_1''})t_1\phantom{]}&\phantom{[}  t_2(a^{\sigma_1''} t_2)][\cdots
\\
 \cdots ][w t_2\phantom{]}&\phantom{[} \overline  t_1b][\cdots 
 &\cdots ][(\overline t_1 w)t_2\phantom{]}&\phantom{[}\overline t_1(bt_1)][\cdots 
&\cdots ][(\overline t_2 w^{\sigma_1''})t_1\phantom{]}&\phantom{[} \overline  t_2(b^{\sigma_1''} t_2)][\cdots
\\
 \cdots ][w t_2\phantom{]}&\phantom{[} \overline  t_1][t_3{\uparrow}\overline t_n\cdots 
 &\cdots ][(\overline t_1 w)t_2] &[t_3{\uparrow}\overline t_n\cdots 
&\cdots ][(\overline t_2 w^{\sigma_1''})t_1]  &[t_3{\uparrow}\overline t_n\cdots
\\
 \cdots ][w t_2\phantom{]}&\phantom{[} t_2c][\cdots 
 &\cdots ][(\overline  t_1 w)t_2\phantom{]}&\phantom{[} t_2(ct_1)][\cdots 
&\cdots ][(\overline t_2 w^{\sigma_1''})t_1\phantom{]}&\phantom{[} t_1(c^{\sigma_1''} t_2)][\cdots
\end{tabular}}

\bigskip

\centerline{
\begin{tabular}{>{$}r<{$} @{}  >{$}l<{$}||>{$}r<{$} @{} >{$}l<{$} ||>{$}r<{$} @{}  >{$}l<{$}}
\setlength\extrarowheight{3pt}
 & \hskip -1cm (\cdots ][w \overline  t_2{\blacktriangleleft})
 & & \hskip -1cm  (\cdots ][w \overline t_2{\blacktriangleleft})^{\sigma_1'}
&& \hskip -1cm  (\cdots ][w \overline  t_2{\blacktriangleleft})^{\sigma_1}
\\ 
 \hhline{------}   &&& & &
\\ 
[-.45cm] \hhline{------} &&& & &
\\  
[-.3cm] 
 \cdots ][w \overline t_2\phantom{]}& \phantom{[}\overline t_2d][\cdots 
 &\cdots ][(\overline  t_1 w)\overline t_2\phantom{]}  & \phantom{]}\overline t_2(dt_1)][\cdots 
&\cdots ][(\overline t_2 w^{\sigma_1''})\overline t_1\phantom{]} & \phantom{]}\overline t_1(d^{\sigma_1''} t_2)][\cdots
\\
\cdots][w \overline t_2]&[t_3{\uparrow}\overline t_n \cdots  
&  \cdots ][(\overline t_1 w)\overline t_2\phantom{]}&\phantom{]}t_1][t_3{\uparrow}\overline t_n \cdots 
& \cdots ][(\overline t_2 w^{\sigma_1''})\overline t_1\phantom{]}&\phantom{]}t_2][t_3{\uparrow}\overline t_n \cdots
\\
 \cdots ][w \overline t_2\phantom{]}& \phantom{]}  t_1a][\cdots 
 &\cdots ][(\overline t_1 w)\overline t_2\phantom{]} &\phantom{]}t_1(at_1)][\cdots 
&\cdots ][(\overline t_2 w^{\sigma_1''})\overline t_1 \phantom{]} &  \phantom{]} t_2(a^{\sigma_1''} t_2)][\cdots
\\
 \cdots ][w \overline t_2\phantom{]}& \phantom{]}\overline  t_1b][\cdots 
& \cdots ][(\overline t_1 w)\overline t_2\phantom{]} &\phantom{]}\overline t_1(bt_1)][\cdots 
&\cdots ][(\overline t_2 w^{\sigma_1''}) \overline t_1\phantom{]}  & \phantom{]}\overline  t_2(b^{\sigma_1''} t_2)][\cdots
\\
 \cdots ][w \overline t_2\phantom{]}&\phantom{]} \overline  t_1][t_3{\uparrow}\overline t_n\cdots 
& \cdots ][(\overline t_1 w)\overline t_2] &[t_3{\uparrow}\overline t_n\cdots 
&\cdots ][(\overline t_2 w^{\sigma_1''}) \overline t_1 ]  &[t_3{\uparrow}\overline t_n\cdots
\end{tabular}}

\bigskip

\centerline{
\begin{tabular}{>{$}r<{$} @{}  >{$}l<{$}||>{$}r<{$} @{} >{$}l<{$} ||>{$}r<{$} @{}  >{$}l<{$}}
\setlength\extrarowheight{3pt}
 & \hskip -.5cm (\cdots  t_3{\blacktriangleleft})
 & &  (\cdots   t_3{\blacktriangleleft})^{\sigma_1'}
&&  (\cdots   t_3{\blacktriangleleft})^{\sigma_1}
\\ 
 \hhline{------}   &&& && 
\\ 
[-.45cm] \hhline{------} &&& & &
\\  
[-.3cm] 
 \cdots t_3\phantom{]} & \phantom{[}t_4{\uparrow}\overline t_n\cdots 
&  \cdots t_3\phantom{]} &  \phantom{[}t_4{\uparrow}\overline t_n \cdots 
& \cdots  t_3\phantom{]} &   \phantom{[}t_4{\uparrow}\overline t_n\cdots
\\
 \cdots t_3]& [t_1a][\cdots 
 & \cdots t_3]& [(at_1)][\cdots 
& \cdots t_3]& [(a^{\sigma_1''} t_2)][\cdots
\\
 \cdots t_3]& [\overline t_1b][\cdots 
 & \cdots t_3]& [\overline t_1 \overline t_1 (b t_1)][\cdots 
& \cdots t_3]& [\overline t_2 \overline t_2 (b^{\sigma_1''} t_2)][\cdots
\\
 \cdots t_3]& [\overline t_1][t_3{\uparrow}\overline t_n\cdots 
 & \cdots t_3]& [\overline t_1 ][t_3{\uparrow}\overline t_n\cdots 
& \cdots t_3]& [\overline t_2 ][t_3{\uparrow}\overline t_n\cdots
\\
 \cdots t_3]& [t_2c][\cdots 
 & \cdots t_3]& [\overline t_1t_2(c t_1)][\cdots 
& \cdots t_3]& [\overline t_2t_1(c^{\sigma_1''} t_2)][\cdots
\\
 \cdots t_3]& [\overline t_2d][\cdots 
 & \cdots t_3]& [\overline t_1 \overline t_2 (d t_1)][\cdots 
& \cdots t_3]& [\overline t_2 \overline t_1 (d^{\sigma_1''} t_2)][\cdots
\\
 \cdots t_3\phantom{]}&  \phantom{[}t_3\cdots 
 & \cdots t_3\phantom{]}&  \phantom{[}t_3\cdots 
& \cdots t_3\phantom{]}&  \phantom{[}t_3 \cdots
\end{tabular}}
\medskip

The remaining tables are clearly of the same form as the last one.
Thus we have proved that the action of $\sigma_1$ respects the ordering of $\mathfrak{E}$.
It follows that the action of $\overline \sigma_1$ respects the ordering of $\mathfrak{E}$.
Similarly, the actions of  $\sigma_{[2{\uparrow}n-1]} \cup \overline \sigma_{[2{\uparrow}n-1]}$
respect the ordering of $\mathfrak{E}$.  Hence $\B_n$ acts on $(\mathfrak{E},\le)$.
\hfill\qed
\end{review}

\begin{remarks}[Thurston~\cite{ShortWiest}]
 The (right) action of $\B_n$ on $(\mathfrak{E},\le)$ gives rise to many right orderings of 
$\B_n$.  

Let us use the left-to-right lexicographic ordering on  $(\mathfrak{E}^{n},\le)$, and consider the
$\B_n$-orbit of $t^\infty_{([1{\uparrow}n])}:=(t_i^\infty)_{i\in [1{\uparrow}n]}$.  It is not difficult to show
 that the $\B_n$-stabilizer of 
 $t^\infty_{([1{\uparrow}n])}$ is trivial. 
 Thus we have an injective map 
$$\B_n \to \mathfrak{E}^n, \qquad \phi \mapsto t^{\infty\phi}_{([1{\uparrow}n])}:=((t_i^\infty)^\phi)_{i\in [1{\uparrow}n]}.$$ 
Let $\le$ denote the ordering of $\B_n$  induced by pullback from $\mathfrak{E}^n$.  Clearly $\le$ is a right-ordering of~$\B_n$.

If $n\ge2$ and $\phi \in \B_n$ is $\sigma_1$-negative, then, as in the proof of Theorem~\ref{th:D-L}(b2)$\Rightarrow$(c2),
we have $(t_1{\blacktriangleleft})^\phi \subset (t_1{\blacktriangleleft})$. Since 
$\max(t_1{\blacktriangleleft}) = t_1^\infty $,
we see that $(t_1^\infty)^\phi < t_1^\infty$. Hence $\phi < 1$ and $1 < \overline\phi$. 
Similar arguments with $(t_i{\blacktriangleleft})$, $i \in [2{\uparrow}n]$, show
that, if $\phi \in \B_n$ is $\sigma$-positive (resp. $\sigma$-negative), then $1 < \phi$ (resp. $1 > \phi$).
Hence the right-ordering we have obtained from $(\mathfrak{E}^{n},\le)$ coincides with the Dehornoy right-ordering.
However, the study of ends does not seem to readily yield the $\sigma_1$-trichotomy.
\hfill\qed
\end{remarks}

The following will be useful in the study of squarefreeness.

\begin{lemma} \label{lem:int} Let $n \ge 1$, let $i \in [1{\uparrow}n]$, and let 
$w \in \Sigma_{0,1,n} -  ({\star}{t_i}) - ({\star}\overline t_i)$.  Then, in $\mathfrak{E}(\Sigma_{0,1,n})$,
the following hold:
\begin{enumerate}[\normalfont (i).]
\vskip -0.7cm \null
\item\label{it:int1} $w t_i  \overline w((\Pi t_{[1{\uparrow}n]})^\infty) 
\le    w t_i ((\Pi t_{[i{\uparrow}n]} \Pi t_{[1{\uparrow}i-1]})^\infty) = \min(w t_i t_i{\blacktriangleleft});$
\vskip -0.7cm \null
\item \label{it:int22}$
\min(w t_i t_i{\blacktriangleleft}) 
< \max(w \overline t_i \overline t_i{\blacktriangleleft});$
\vskip -0.7cm \null
\item\label{it:int3} $\max(w \overline t_i \overline t_i{\blacktriangleleft})
= w \overline t_i ((\Pi \overline t_{[i{\downarrow}1]}\Pi \overline t_{[n{\downarrow}i+1]})^\infty)
\le  w \overline t_i \overline w((\Pi \overline t_{[n{\downarrow}1]})^\infty);$
\vskip -0.7cm \null
\item\label{it:int4} $(w t_i t_i{\blacktriangleleft}) \cup (w \overline t_i \overline t_i{\blacktriangleleft})\subseteq 
[w t_i  \overline w((\Pi t_{[1{\uparrow}n]})^\infty) , \, w \overline t_i 
\overline w((\Pi \overline t_{[n{\downarrow}1]})^\infty)].$
\vskip -0.7cm \null
\item\label{it:int5}
If   $n\ge 3$, then  one of the following holds:
\begin{enumerate}[\normalfont (a).]
\vskip -0.7cm \null
\item $t_1((\Pi \overline t_{[n{\downarrow}1]})^\infty) \,\,<\,\,   w t_i  \overline w((\Pi t_{[1{\uparrow}n]})^\infty)$;
\vskip -0.7cm \null
\item $t_1((\Pi \overline t_{[n{\downarrow}1]})^\infty)\,\, >\,\,   w \overline t_i 
\overline w((\Pi \overline t_{[n{\downarrow}1]})^\infty)$;
\end{enumerate} 
and, hence, $t_1((\Pi \overline t_{[n{\downarrow}1]})^\infty) \not 
\in [w t_i  \overline w((\Pi t_{[1{\uparrow}n]})^\infty) , \, w \overline t_i 
\overline w((\Pi \overline t_{[n{\downarrow}1]})^\infty)]$, that is,
$t_1(z_1^\infty) \not 
\in [w t_i  \overline w(\overline z_1^\infty) , \, w \overline t_i 
\overline w(z_1^\infty)]$
\end{enumerate}
\end{lemma}

\begin{proof} Recall that:

\centerline{ $(t_1{\blacktriangleleft}) 
< (\overline t_1{\blacktriangleleft}) <  (t_2{\blacktriangleleft}) 
< \cdots < (t_n{\blacktriangleleft}) 
< (\overline t_n{\blacktriangleleft}),$}

\centerline{
$(t_it_{i+1}{\blacktriangleleft}) 
< (t_i\overline t_{i+1}{\blacktriangleleft}) 
< \cdots 
< (t_i\overline t_n{\blacktriangleleft}) < (t_it_1{\blacktriangleleft}) 
< \cdots 
< (t_i\overline t_{i-1}{\blacktriangleleft}) <  (t_it_{i}{\blacktriangleleft}),$}

\centerline{
$(\overline t_i \overline t_{i}{\blacktriangleleft}) 
< (\overline t_i\overline t_{i+1}{\blacktriangleleft}) 
< \cdots 
< (\overline t_i\overline t_n{\blacktriangleleft}) < (\overline t_it_1{\blacktriangleleft}) 
< \cdots 
< (\overline t_i t_{i-1}{\blacktriangleleft}) <  (\overline t_i\overline t_{i-1}{\blacktriangleleft}).$}

\medskip

\eqref{it:int1}. It is straightforward to see that $
 w t_i ((\Pi t_{[i{\uparrow}n]} \Pi t_{[1{\uparrow}i-1]})^\infty) = \min(w t_i t_i{\blacktriangleleft})$.

Let $x$ denote the element of $t_{[1{\uparrow}n]} \cup \overline t_{[1{\uparrow}n]}$ such that 
 $\overline w((\Pi t_{[1{\uparrow}n]})^\infty) \in (x {\blacktriangleleft})$; notice that
$x \ne \overline t_i$.

\noindent  If $x \ne t_i$, then $ (w t_i x {\blacktriangleleft}) 
 \,\,<  \,\, (w t_i t_i{\blacktriangleleft})$, and we have
$$w t_i  \overline w((\Pi t_{[1{\uparrow}n]})^\infty) \,\,
\in(w t_i x {\blacktriangleleft})  \,\,<  \,\, (w t_i t_i{\blacktriangleleft}) 
\,\,\ni\,\, \min(w t_i t_i{\blacktriangleleft}).$$

\noindent  If $x = t_i$, then $\overline w$ is completely cancelled in 
$ \overline w((\Pi t_{[1{\uparrow}n]})^\infty)$, and, moreover, 
$$w t_i  \overline w((\Pi t_{[1{\uparrow}n]})^\infty) = w t_i ((\Pi t_{[i{\uparrow}n]}\Pi t_{[1{\uparrow}i-1]})^\infty)
= \min(w t_i t_i{\blacktriangleleft}).$$

Thus, \eqref{it:int1} holds.

\eqref{it:int22} is clear.

\eqref{it:int3}.  It is straightforward to see that 
$w \overline t_i ((\Pi \overline t_{[i{\downarrow}1]}\Pi \overline t_{[n{\downarrow}i+1]})^\infty)
= \max(w \overline t_i \overline t_i{\blacktriangleleft})$.

Let $x$ denote the element of $t_{[1{\uparrow}n]} \cup \overline t_{[1{\uparrow}n]}$ such that 
 $ \overline w((\Pi \overline t_{[n{\downarrow}1]})^\infty) 
\in (x {\blacktriangleleft})$;  notice that $x \ne  t_i$.

 \noindent If $x \ne \overline t_i$, then
 $ (w\overline t_i \overline t_i  {\blacktriangleleft})  
\,\,<  \,\, (w \overline t_i x{\blacktriangleleft})$, and we have
$$\max(w \overline t_i \overline t_i{\blacktriangleleft}) \,\,
\in\,\,(w\overline t_i \overline t_i  {\blacktriangleleft})  
\,\,<  \,\, (w \overline t_i x{\blacktriangleleft})
\,\,\ni\,\, w \overline t_i \overline w((\Pi \overline t_{[n{\downarrow}1]})^\infty).$$

\noindent  If $x = \overline t_i$, then $\overline w$ is completely cancelled in 
$ \overline w(\Pi \overline t_{[n{\downarrow}1]})^\infty$, and, moreover, 
$$w \overline t_i  \overline w((\Pi \overline t_{[n{\downarrow}1]})^\infty) = w \overline
t_i ((\Pi \overline t_{[i{\downarrow}1]} \Pi \overline t_{[n{\downarrow}i+1]})^\infty)
= \max(w \overline t_i \overline t_i{\blacktriangleleft}).$$

Thus, \eqref{it:int3} holds.

\eqref{it:int4}  follows from \eqref{it:int1}-\eqref{it:int3}.

\eqref{it:int5}. It is not difficult to see
 that

\centerline{$\text{ $ w t_i  \overline w((\Pi t_{[1{\uparrow}n]})^\infty) \in (w t_i {\blacktriangleleft})$ \quad and \quad
$ w \overline t_i \overline w((\Pi \overline t_{[n{\downarrow}1]})^\infty) \in
   (w \overline t_i {\blacktriangleleft})$}.$}

\noindent \textbf{Case 1.} $w  \not \in (t_1 {\star})$.

 If $w = 1$, then 
$$t_1((\Pi \overline t_{[n{\downarrow}1]})^\infty)\,\, \in\,\, (t_1\overline t_n{\blacktriangleleft}) 
\,\, < \,\, (t_it_1{\blacktriangleleft}) \,\,\ni\,\, t_i  ((\Pi t_{[1{\uparrow}n]})^\infty) =
w t_i  \overline w((\Pi t_{[1{\uparrow}n]})^\infty).$$

 If $w \ne 1$, then $t_1((\Pi \overline t_{[n{\downarrow}1]})^\infty)\,\, \in \,\,(t_1{\blacktriangleleft}) 
\,\, < \,\, (w{\blacktriangleleft}) \,\, \ni \,\, w t_i  \overline w((\Pi t_{[1{\uparrow}n]})^\infty)$.

In both subcases,  (a)  holds.

 \noindent   \textbf{Case 2.} $w  \in (t_1 {\star})$.

Here, $w \overline t_i \overline  w((\Pi \overline t_{[n{\downarrow}1]})^\infty) \,\,\in
\,\,   (w {\blacktriangleleft}) 
\,\, \subseteq \,\, (t_1  {\blacktriangleleft})$.
Hence, 
$$w \overline t_i \overline  w((\Pi \overline t_{[n{\downarrow}1]})^\infty)\,\, \le \,\,
\max(t_1{\blacktriangleleft}) \,\,=\,\,t_1((\Pi \overline t_{[n{\downarrow}1]})^\infty).$$
To prove that (b)  holds, it remains to show that 
$$w \overline t_i \overline  w((\Pi \overline t_{[n{\downarrow}1]})^\infty)\,\, \ne \,\,
t_1((\Pi \overline t_{[n{\downarrow}1]})^\infty),$$
that is,
$\overline t_1 w \overline t_i \overline  w((\Pi \overline t_{[n{\downarrow}1]})^\infty)\,\, \ne \,\,
(\Pi \overline t_{[n{\downarrow}1]})^\infty ,$ that is,
$\overline t_1 w \overline t_i \overline  w \not \in \gen{\Pi \overline t_{[n{\downarrow}1]}}$.
We can write $w = t_1 u$ where $u \not\in (\overline t_1 {\star})$.
Then $\overline t_1  w \overline t_i \overline w  =   u \overline t_i \overline u   \overline t_1$, in
normal form.  Thus it suffices to show that 
$u \overline t_i \overline  u \overline t_1 \not \in \gen{\Pi \overline t_{[n{\downarrow}1]}}$.

If $u = 1$, then  $ u \overline t_i \overline u \overline t_1  = \overline t_i \overline t_1
\not \in \gen{\Pi \overline t_{[n{\downarrow}1]}}$, since $n \ge 3$.

If $u \ne 1$, then  $ u \overline t_i \overline u \overline t_1 
\not \in \gen{\Pi \overline t_{[n{\downarrow}1]}}$, since  
$ u \overline t_i \overline u \overline t_1 $ does not lie in the submonoid
of $\Sigma_{0,1,n}$ generated by $t_{[1{\uparrow}n]}$, nor in the 
submonoid generated by  $\overline t_{[1{\uparrow}n]}$.

In both subcases, (b) holds.

\medskip

In both cases, \eqref{it:int5} holds.
\end{proof}

The following appeared as~\cite[Lema~2.2.17]{Bacardit}.

\begin{theorem}\label{th:square} If $n\ge 1$ then, for each $\phi\in\B_n$, 
$t_1^\phi((\Pi \overline t_{[n{\downarrow}1]})^\infty)$ is a squarefree end.
\end{theorem}

\begin{proof} This is clear if $n = 1$.  

For $n = 2$, 
 $\B_2 = \gen{\sigma_1}$, and 
$$t_1^{\B_2} =  \{t_1^{\sigma_1^{2m}}, t_1^{\sigma_1^{1+2m}} \mid m \in \integers\}
=  \{t_1^{(t_1t_2)^m}, t_2^{(t_1t_2)^m} \mid m \in \integers\}.$$
Thus, every word in $t_1^{\B_2}$ is squarefree and does not end in $\overline t_2$.
Hence, every end in $t_1^{\B_2}((\Pi \overline t_{[n{\downarrow}1]})^\infty)$ is squarefree.

Thus, we may assume that $n \ge 3$.

Recall that $z_1 = \Pi \overline t_{[n{\uparrow}1]}$, and, hence, $\overline z_1 = \Pi t_{[1{\uparrow}n]}$.
Let $\cup[t]_{[1{\uparrow}n]}$ denote $\bigcup\limits_{i\in[1{\uparrow}n]} [t_i]$.
By Lemma~\ref{lem:int}\eqref{it:int5},  $t_1( z_1^\infty)$ does not lie in 
$$\textstyle\bigcup\limits_{x \in \cup[t]_{[1{\uparrow}n]}}[x(\overline z_1^\infty) , \,  
\overline x (z_1^\infty)]
\,\,\,\,\,\,(=\,\,\,\,\,\,
\textstyle \bigcup\limits_{i=1}^n \,\,\,
\bigcup\limits_{w \in \Sigma_{0,1,n} - ({\star} t_i) - ({\star} \overline t_i)}
[wt_i\overline w(\overline z_1^\infty), \,  
w\overline t_i\overline w (z_1^\infty) ]).$$
  Notice that $\phi$ permutes the elements of each of the following sets: $\cup[t]_{[1{\uparrow}n]}$;
$\{\overline z_1^\infty \}$;
$\{z_1^\infty\}$; and, $\bigcup\limits_{x \in \cup[t]_{[1{\uparrow}n]}}[x(\overline z_1^\infty), \,  
\overline x (z_1^\infty)]$.  Hence $ (t_1(z_1^\infty))^\phi$ does not lie in $\bigcup\limits_{x \in \cup[t]_{[1{\uparrow}n]}}[x(\overline z_1^\infty), \,  
\overline x (z_1^\infty)]$.
By Lemma~\ref{lem:int}\eqref{it:int4}, 
$$\textstyle\bigcup\limits_{x \in \cup[t]_{[1{\uparrow}n]}}[x(\overline z_1^\infty) , \,  
\overline x (z_1^\infty)] \,\,\,\,\,\,\supseteq\,\,\, \,\,\,\bigcup\limits_{i=1}^n \,\,\,
\textstyle\bigcup\limits_{w \in \Sigma_{0,1,n} - ({\star} t_i) - ({\star} \overline t_i)}
((w t_i t_i{\blacktriangleleft}) \cup (w \overline t_i \overline t_i{\blacktriangleleft})).$$
Hence, $ (t_1(z_1^\infty))^\phi$ does not lie in the latter set either, and, hence,
$ (t_1(z_1^\infty))^\phi$ is a squarefree end.  Since
$(t_1(z_1^\infty))^\phi  =  t_1^\phi(z_1^\infty)$, the desired result holds.
\end{proof}

We now obtain new information about the $\B_n$-orbit of $t_1$ in $\Sigma_{0,1,n}$.

\begin{corollary}\label{cor:square} Let $n \ge 1$, let $\phi\in\B_n$, and let $k \in [1{\uparrow}n]$.
\begin{enumerate}[\normalfont (i).]
\vskip-0.6cm \null
\item  $t_1^\phi$ is a squarefree word in $\Sigma_{0,1,n}$.
\vskip-0.6cm \null
\item  $t_1^\phi \not \in ( \Pi \overline t_{[n{\downarrow}k+1]} t_k {\star} ) -\{t_k^{\Pi t_{[k+1{\uparrow}n]}}\}$.
\vskip-0.6cm \null
\item  $t_1^\phi \not \in ( \Pi  t_{[1{\uparrow}k-1]} \overline t_{k} {\star} )$.
\end{enumerate}
\end{corollary}

\begin{proof}  Recall from Notation~\ref{not:basic} that we write
$t_1^{\phi} = t_{1^{\pi(\phi)}}^{w_1(\phi)}$. Let $\pi = \pi(\phi)$ and $w_1 = w_1(\phi)$.

It is not difficult to see that  
$$ t_1^\phi(z_1^\infty) 
 = \overline w_1 t_{1^{\pi}} w_1 ((\Pi \overline t_{[n{\downarrow}1]})^\infty)
\,\,\, \in \,\,\, (\overline w_1{\blacktriangleleft}).$$
By Theorem~\ref{th:square}, $t_1^\phi(z_1^\infty) $ is a squarefree end. 
Hence, $\overline w_1$ is a squarefree word,  and $w_1 \not \in ({\star} \overline t_k \Pi t_{[k+1{\downarrow}n]} )$.

Since $\overline w_1$ is a squarefree word, $t_1^{\phi}$ is also a squarefree word.  Hence (i) holds.

Also,  $w_1 \not \in ({\star} \overline t_k \Pi t_{[k+1{\uparrow}n]} )$  implies that
  $\overline  w_1 \not \in ( \Pi \overline t_{[n{\downarrow}k+1]} t_k {\star} )$ 
and, hence, $t_1^{\phi} \not \in ( \Pi \overline t_{[n{\downarrow}k+1]} t_k {\star} ) -\{t_k^{\Pi t_{[k+1{\uparrow}n]}}\}$
and, also,  $\overline t_1^{\phi} \not \in ( \Pi \overline t_{[n{\downarrow}k+1]} t_k {\star} )$.
In particular, (ii) holds.

Let $\xi$  be the automorphism of $\Sigma_{0,1,n}$ determined by \begin{tabular}
{
>{$}r<{$}  
@{} 
>{$}c<{$} 
@{\hskip0cm} 
>{$}l<{$}
}
\setlength\extrarowheight{3pt}
&\underline{\scriptstyle j \in [1{\uparrow}n]}&
\\[.15cm]
&(t_j)^{\xi}&\\  
=(&\overline t_{n+1-j} &)
\end{tabular}.
Then $\xi^2 = 1$ and $\xi \in \Out^{-}_{0,1,n}:= \Out_{0,1,n}-\Out^{+}_{0,1,n}$.  Also,
$$
t_n^{\phi^\xi} = t_n^{\xi\phi\xi} = \overline t_1^{\phi \xi} \not  \in 
( \Pi \overline t_{[n{\downarrow}k+1]} t_k {\star} )^{\xi} = ( \Pi  t_{[1{\uparrow}n-k]} \overline t_{n+1-k} {\star} ).
$$  
It follows that $t_n^{\B_n^\xi} \cap ( \Pi  t_{[1{\uparrow}n-k]} \overline t_{n+1-k} {\star} ) = \emptyset$.
Since $\B_n^\xi = \B_n$ and $t_n^{\B_n} = t_1^{\B_n}$,  we see that
$t_1^{\phi} \not \in ( \Pi  t_{[1{\uparrow}n-k]} \overline t_{n+1-k} {\star} )$.  
Now replacing $k$ with $n+1-k$ gives (iii).
\end{proof}
In Remark~\ref{rem:square2}, we shall give a second proof of Corollary~\ref{cor:square} using Larue-Whitehead
diagrams.

\section{Actions on free products of cyclic groups}\label{sec:freeprods}

\begin{notation}\label{not:powers}  Throughout this section, we assume that $n \ge 1$ and we 
fix a positive integer $N$.

Let $p_{([1{\uparrow}N])}$ be a partition of $n$.  Thus,
$p_{([1{\uparrow}N])}$  is an $N$-tuple  for $[1{\uparrow}\infty[$
such that $p_1 + \cdots +p_N = n$.

Let $m_{([1{\uparrow}N])}$ be an $N$-tuple for $\naturals  - \{1\}$. 

We let $\Sigma_{0,1,p_1^{(m_1)}\perp p_2^{(m_2)} \perp \cdots \perp p_N^{(m_n)}}$ denote the
group with presentation $$\gen{z,\tau_{[1{\uparrow}n]} \mid 
z\Pi  \tau_{[1{\uparrow}n]}, \{\tau_{j+\sum p_{[1{\uparrow}i-1]}}^{m_i}\}_{i \in [1{\uparrow}N], j  \in 
[1{\uparrow}p_i]}}.$$
Thus, $\Sigma_{0,1,p_1^{(m_1)}\perp p_2^{(m_2)} \perp \cdots \perp p_N^{(m_N)}}$  is isomorphic to
a free product of cyclic groups,
 $C_{m_1}^{\ast p_1} \ast C_{m_2}^{\ast p_2} \ast \cdots C_{m_N}^{\ast p_N}$, where $C_0$ is interpreted as
$C_\infty$, and $p_i^{(0)}$ is also written $p_i$ with no exponent.

We let $\Out_{0,1,p_1^{(m_1)}\perp p_2^{(m_2)} \perp \cdots \perp p_N^{(m_N)}}$
denote the group of all automorphisms of 
 $\Sigma_{0,1,p_1^{(m_1)}\perp p_2^{(m_2)} \perp \cdots \perp p_N^{(m_n)}}$
which map $\{z,\overline z\}$ and $$
\{\{\{[\tau_i],[\overline \tau_i]\} \mid i \in [p_1 + ... + p_{j-1} + 1{\uparrow}p_1 + ... + p_{j}]\}\mid
 j \in [1{\uparrow}N]\}$$
to themselves. 

We let $\Out^+_{0,1,p_1^{(m_1)}\perp p_2^{(m_2)} \perp \cdots \perp p_N^{(m_n)}}$
denote the group of all automorphisms of 
 $\Sigma_{0,1,p_1^{(m_1)}\perp p_2^{(m_2)} \perp \cdots \perp p_N^{(m_n)}}$
which map $\{z\}$ and $$
\{\{[\tau_i] \mid i \in [p_1 + ... + p_{j-1} + 1{\uparrow}p_1 + ... + p_{j}]\}\mid j \in [1{\uparrow}N]\}$$
to themselves.

In the case where all the $m_i$ are $0$, we get  groups denoted
$\Out_{0,1,p_1 \perp p_2  \perp \cdots \perp p_N }$ and
$\Out^+_{0,1,p_1 \perp p_2  \perp \cdots \perp p_N }$.  Notice that 
$\Out_{0,1,p_1 \perp p_2  \perp \cdots \perp p_N }$
is the subgroup of $\Out_{0,1,n}$ consisting of those
 elements such that the permutation in $\Sym_n$,
arising from the permutation  
of $\{\{[t_1],[\overline t_1]\}, \ldots, \{[t_n],[\overline t_n]\}\}$,
lies in the natural image of $\Sym_{p_1} \times \Sym_{p_2} \times \cdots \times \Sym_{p_N}$
in $\Sym_n$.

There are natural maps
\begin{align}\label{eq:maps0}
\Out_{0,1,p_1\perp p_2 \perp \cdots \perp p_N} 
&\to 
\Out_{0,1,p_1^{(m_1)}\perp p_2^{(m_2)} \perp \cdots \perp p_N^{(m_n)}},
\\
\Out^+_{0,1,p_1\perp p_2 \perp \cdots \perp p_N} 
&\to 
\Out^+_{0,1,p_1^{(m_1)}\perp p_0^{(m_2)} \perp \cdots \perp p_N^{(m_n)}}.  \label{eq:maps1} 
\end{align}
Since~\eqref{eq:maps1} is of index two in~\eqref{eq:maps0}, we see that~\eqref{eq:maps0} is injective, 
resp. surjective, resp. bijective, if and only if~\eqref{eq:maps1} is.
\hfill\qed   
\end{notation}

For topological reasons, we suspect that~\eqref{eq:maps0} and~\eqref{eq:maps1} are isomorphisms. 
In this section, we shall prove that this holds in the case where all the $m_i$ are equal or $N=1$.
We begin by proving that~\eqref{eq:maps0} and~\eqref{eq:maps1} are injective, which seems to be new.

\begin{theorem}\label{th:BHplus} With {\normalfont Notation~\ref{not:powers}},  the maps 
\setcounter{theorem}{1}
\begin{align} 
\Out_{0,1,p_1\perp p_2 \perp \cdots \perp p_N} 
&\to 
\Out_{0,1,p_1^{(m_1)}\perp p_2^{(m_2)} \perp \cdots \perp p_N^{(m_n)}},
 \\
\Out^+_{0,1,p_1\perp p_2 \perp \cdots \perp p_N} 
&\to 
\Out^+_{0,1,p_1^{(m_1)}\perp p_2^{(m_2)} \perp \cdots \perp p_N^{(m_n)}}  
\end{align}
are injective. 
\setcounter{theorem}{2}
\end{theorem}

\begin{proof} Suppose that $\phi$ is an element of the kernel of 
 \eqref{eq:maps0} or  \eqref{eq:maps1}.   Clearly, $\phi \in \Out^+_{0,1,n}$,
and $t^\phi_{([1{\uparrow}n])}$, $t_{([1{\uparrow}n])}$
have the same image  in $\Sigma_{0,1,p_1^{(m_1)}\perp p_2^{(m_2)} \perp \cdots \perp p_N^{(m_N)}}$.
By Theorem~\ref{th:square}, $(t_{([1{\uparrow}n])})^\phi$ is an $n$-tuple of squarefree words
in $\Sigma_{0,1,n}$, and, hence, has the same normal form in $\Sigma_{0,1,n}$ and in
$\Sigma_{0,1,p_1^{(m_1)}\perp p_2^{(m_2)} \perp \cdots \perp p_N^{(m_N)}}$.  Hence  
 $t^\phi_{([1{\uparrow}n])} = t_{([1{\uparrow}n])}$ as $n$-tuples for $\Sigma_{0,1,n}$.  Thus $\phi = 1$, and
the result is proved.
\end{proof}

\begin{history}\label{rems:BH}  
Let us now restrict to the classic case where $N=1$.
Here, for an integer $m\ge 2$,
 we are considering the action of $\Out_{0,1,n}$ on $C_m^{\ast n}$, and it induces maps
\begin{align}
\Out_{0,1,n} &\to \Out_{0,1,n^{(m)}}, \label{eq:BH} \\
\Out^+_{0,1,n} &\to \Out^+_{0,1,n^{(m)}}.\label{eq:BH2}
\end{align}

Theorem~\ref{th:BHplus} shows that these maps are injective.
Birman-Hilden~\cite[Theorem 7]{BirmanHilden}  gave a topological proof that~\eqref{eq:BH2}
 is injective,
thus answering a question of Magnus.
Crisp-Paris~\cite{CrispParis2} gave an elegant algebraic
proof of the injectivity of~\eqref{eq:BH2} 
using the trichotomy argument of Larue~\cite{Larue94}
and Shpilrain~\cite{Shpilrain}.
The Crisp-Paris argument can be summarized as follows.

For each $i \in [1{\uparrow}n]$, 
let  $(\gen{\tau_{i}}{\star})$ denote the set of
elements of $\Sigma_{0,1,n^{(m)}}$ whose free-product normal form
begins with an element of $\gen{\tau_i}-\{1\}$. 

Suppose that  $\phi$ is a non-trivial element of $\B_n = \Out^+_{0,1,n}$.
We will show that $\phi$ acts non-trivially on $\Sigma_{0,1,n^{(m)}}$.

We may assume that $n\ge 3$.  By Theorem~\ref{th:order},
by replacing $\phi$ with $\overline \phi$ if necessary, we may assume that $\phi$
is $\sigma$-neg\-ative.  Thus there exists some\linebreak
$i \in [1{\uparrow}n-1]$ such that $\phi$ 
is the product of a finite sequence of elements of 
$\sigma_{[i+1{\uparrow}n-1]} \cup \overline \sigma_{[i{\uparrow}n-1]}$,
and $\overline \sigma_i$ appears at least once in the sequence.

With Notation~\ref{not:alpha},
 $$(\gen{\tau_{i}}{\star})^{\overline \sigma_i} = (\gen{\tau_{i}}{\star})^{\overline \sigma_i'' \overline \sigma_i'} 
= (\gen{\tau_{i+1}}{\star})^{\overline \sigma_i'} 
\subseteq  (\tau_i(\gen{\tau_{i+1}}{\star})) 
\subset (\gen{\tau_{i}}{\star}),$$
since $n \ge 3$. 
Because the elements of $\sigma_{[i+1{\uparrow}n-1]} \cup \overline \sigma_{[i{\uparrow}n-1]}$ 
act as injective self-maps on $(\gen{\tau_{i}}{\star})$, it follows that 
$(\gen{\tau_{i}}{\star})^{\phi} \subset (\gen{\tau_{i}}{\star})$, 
and, hence, $\phi$ acts non-trivially on $\Sigma_{0,1,n^{(m)}}$, as desired.
 \hfill\qed
\end{history}

Let us now verify the surjectivity  of the maps 
\eqref{eq:BH} and \eqref{eq:BH2}.
The case where  $m = 2$  was verified by Stephen Humphries~\cite[Lemma~2.1.7]{Sakuma}.

\begin{notation} Let $m,\,n \in \naturals$ with $n \ge 1$ and  $m \ge 2$.
Let $\lfloor\frac{m}{2}\rfloor$ denote the greatest integer not exceeding $\frac{m}{2}$.
Then $[0{\uparrow}\lfloor\frac{m}{2}\rfloor ] \cup [-1{\downarrow}(-\lfloor\frac{m-1}{2}\rfloor)]$ 
is a set of representatives for the integers modulo~$m$.    
For $\tau^k \in \gen{\tau \mid \tau^m=1}$, we define $\abs{\tau^k}$ by
\begin{tabular}
{
>{$}r<{$} 
@{} 
>{$}l<{$}
@{\hskip .5cm}
>{$}r<{$} 
@{\hskip 0cm} 
>{$}l<{$}
}
\setlength\extrarowheight{3pt}
 &\hskip-.4cm\underline{\scriptstyle k \in [0{\uparrow}\lfloor\frac{m}{2}\rfloor] } 
&&\hskip-1.4cm\underline{\scriptstyle k \in [-1{\downarrow}-\lfloor\frac{m-1}{2}\rfloor]}
\\[.25cm]
(&\,\,\abs{\tau^k}&  \abs{\tau^k}\phantom{1}&)\\  
=(&\,\,2k & -2k-1 &)
\end{tabular};
we extend $\abs{-}$ to all of $\Sigma_{0,1,n^{(m)}}$ by using normal forms for the free product $C_m^\ast$. 

Let $\phi \in \Out^+_{0,1,n^{(m)}}$.
 There exists
a unique permutation $\pi \in \Sym_n$, and a unique $(n+2)$-tuple $(w_{([0{\uparrow}n+1])})$ for $\Sigma_{0,1,n^{(m)}}$
 such that
$w_0 = 1$ and $w_{n+1} = 1$, and,
for each $i \in [1{\uparrow}n]$,  $w_i \not\in (t_{i^\pi}{\star})\cup (\overline t_{i^\pi}{\star})$ 
and  $t_i^\phi = t_{i^\pi}^{w_i}$.
For each $i \in [0{\uparrow}n]$, let $u_i = w_i \overline w_{i+1}$. 
We define    $ \pi(\phi) := \pi$, $w_i(\phi) :=  w_i$, $i \in [0{\uparrow}n+1]$, and 
$u_i(\phi):=u_i$, $i \in [0{\uparrow}n]$.  
We write $\norm{\phi}:=  n + 2 \sum\limits_{i \in [1{\uparrow}n]} \abs{w_i(\phi)}$.
\hfill\qed
\end{notation}

The following is similar to Artin's Lemma~\ref{lem:art2}.

\begin{lemma}\label{lem:humph}
Let $n\ge 1$, $m \ge 2$ and let $\phi \in \Out_{0,1,n^{(m)}}$.  Let $\pi = \pi(\phi)$.
For each  $i \in[0{\uparrow}n]$,  let
$u_i = u_i(\phi)$.  
For each
 $i \in[1{\uparrow}n]$,  let $a_i$, $b_i$ denote the elements of $[0,m-1]$ determined by the following:
there exists some $u_i' \in \Sigma_{0,1,n^{(m)}} - (\star \gen{\tau_{i^\pi}})$ such that
$u_{i-1} = u_i'\tau_{i^\pi}^{a_i}$;
there exists some $u_i''\in \Sigma_{0,1,n^{(m)}} - (\gen{\tau_{i^\pi}}\star )$ such that
$u_{i} = \tau_{i^\pi}^{b_i}u_i''$.   In particular, $a_1 = b_{n} = 0$.
\begin{enumerate}[\normalfont (i).] 
\vskip-0.7cm \null
\item  Suppose that there exists some $i \in [2{\uparrow}n]$
such that  $a_i \in[\lfloor\frac{m}{2}\rfloor {\uparrow} m-1]$.  Then
$\norm{\sigma_{i-1}\phi} < \norm{\phi}$. 
\vskip-0.7cm \null
\item   Suppose that there exists some $i \in [1{\uparrow}n-1]$
such that   $b_i \in [\lfloor\frac{m+1}{2}\rfloor {\uparrow}m-1]$.  Then
$\norm{\overline\sigma_{i}\phi} < \norm{\phi}$. 
\vskip-0.7cm \null
\item If $\phi \ne 1$, there exists some 
$\sigma_i^\epsilon \in \sigma_{[1{\uparrow}n-1]} \cup
\overline \sigma_{[1{\uparrow}n-1]}$ such that $\norm{\sigma_i^\epsilon \phi} < \norm{\phi}$.
\end{enumerate}
\end{lemma}

\begin{proof} (i). Let $a = a_i$. There exists some  
$v \in \Sigma_{0,1,n^{(m)}} - (\star \gen{\tau_{i^\pi}})$ such that $u_{i-1} = v\tau_{i^\pi}^{a}$.  
Since $w_{i-1}(\phi) = u_{i-1}w_{i}(\phi)$, we have 
\begin{equation}\label{eq:ws2}
w_{i-1}(\phi)  = v\tau_{i^\pi}^aw_{i}(\phi);
\end{equation}
since $w_{i}(\phi) \not\in (\gen{\tau_{i^\pi}}\star)$  
and   $v \not\in (\star \gen{\tau_{i^\pi}})$,
 $v \tau_{i^\pi}^aw_{i}(\phi)$ is a free-product normal form for $w_{i-1}(\phi)$.
 
\medskip

\noindent \textbf{Claim.} {\it 
$\abs{\tau_{i^\pi}^{a+1}} < \abs{\tau_{i^\pi}^a}$.}

\begin{proof}   If $a \in[\lfloor\frac{m}{2}\rfloor +1 {\uparrow} m-1]$, then $a-m \in 
 [-\lfloor\frac{m-1}{2}\rfloor {\uparrow} -1]$, and, hence,
$$\abs{\tau_{i^\pi}^{a}} = \abs{\tau_{i^\pi}^{a-m}} = -2(a-m)-1 = 2m-2a-1.$$

Therefore, if $a\in [\lfloor\frac{m}{2}\rfloor {\uparrow} m-2]$, 
$\abs{\tau_{i^\pi}^{a+1}} = 2m -2(a+1) -1 = 2m -2a -3$.

Thus, $\abs{\tau_{i^\pi}^{a+1}} < \abs{\tau_{i^\pi}^a}$ if $a \in[\lfloor\frac{m}{2}\rfloor +1 {\uparrow} m-2]$.

For $a = \lfloor\frac{m}{2}\rfloor$, $a \ge \frac{m-1}{2}$, and
$\abs{\tau_{i^\pi}^a} =  2a > 2m-2a-3 = \abs{\tau_{i^\pi}^{a+1}}$.

For $a = m-1$, $\abs{\tau_{i^\pi}^a}= 1$ and $\abs{\tau_{i^\pi}^{a+1}} = 0$.
\end{proof} 

 Thus, $$\abs{w_{i-1}(\phi)} = \abs{v} + \abs{\tau_{i^\pi}^a} +  \abs{w_{i}(\phi)} 
>  \abs{v} + \abs{\tau_{i^\pi}^{a+1}} +  \abs{w_{i}(\phi)}.$$

By~\eqref{eq:ws2}, $w_{i-1}(\phi) \overline w_{i}(\phi)\tau_{i^{\pi}} =  v\tau_{i^\pi}^{a+1}$; hence
$$\tau_{i}^{\sigma_{i-1}\phi} = (\tau_{i-1}^{\tau_{i}})^{\phi} 
= (\tau_{(i-1)^\pi}^{w_{i-1}(\phi)})^{(\tau_{i^\pi}^{w_{i}(\phi)})}
 = \tau_{(i-1)^\pi}^{ v \tau_{i^\pi}^{a+1}w_{i}(\phi)}.$$
Hence, $\abs{w_{i}(\sigma_{i-1}\phi)} = \abs{ v \tau_{i^\pi}^{a+1}w_{i}(\phi)} \le  \abs{v} +  \abs{\tau_{i^\pi}^{a+1}} + \abs{w_{i}(\phi)}
< \abs{w_{i-1}(\phi)}$.

 For each $j \in [1{\uparrow}i-2] \cup [i+1{\uparrow}n]$, $\tau_j^{\sigma_{i-1}\phi} = \tau_j^{\phi}$,
 and, hence,
$\abs{w_j(\sigma_{i-1}\phi)}= \abs{w_j(\phi)}$.

Also, $\tau_{i-1}^{\sigma_{i-1}\phi}= \tau_{i}^{\phi}$; in particular, 
$\abs{w_{i-1}(\sigma_{i-1}\phi)}= \abs{w_{i}(\phi)}$.

It now follows that $\norm{\sigma_{i-1}\phi} < \norm{\phi}$.

\medskip

(ii).  Let $b = b_i$.  
There exists some $v \in \Sigma_{0,1,n^{(m)}} -(\gen{\tau_{i^\pi}}\star)$ such that $ u_i = \tau_{i^\pi}^b v$. 
Since $w_{i+1}(\phi) = \overline u_i w_{i}(\phi)$, we have 
\begin{equation}\label{eq:wsss}
w_{i+1}(\phi)  = \overline v \,\,\overline \tau_{i^\pi}^b w_{i}(\phi).
\end{equation}
Since $w_{i}(\phi) \not\in (\gen{\tau_{i^\pi}}\star)$  
and   $\overline v \not\in (\star \gen{\tau_{i^\pi}})$,
 $\overline v \,\,\overline \tau_{i^\pi}^b w_{i}(\phi)$ is a free-product normal form
for $w_{i+1}(\phi)$.  Hence, 
$\abs{w_{i+1}(\phi)} =  \abs{\overline v} + \abs{\overline \tau_{i^\pi}^b} +  \abs{w_{i}(\phi)}.$

\medskip

\noindent \textbf{Claim.} {\it 
$\abs{\overline \tau_{i^\pi}^{\,\,b+1}} < \abs{\overline \tau_{i^\pi}^b}$.}

\begin{proof}   For any $b \in[\lfloor\frac{m+1}{2}\rfloor{\uparrow} m]$, then $m-b \in 
 [\lfloor\frac{m}{2}\rfloor {\downarrow} 0]$, and, hence,
$$\abs{\overline \tau_{i^\pi}^{b}} = \abs{\tau_{i^\pi}^{m-b}} = 2(m-b) = 2m -2b.$$
Therefore, since $b \in [\lfloor\frac{m+1}{2}\rfloor {\uparrow} m-1]$, 
$$\abs{\overline \tau_{i^\pi}^{b+1}} = 2m -2(b+1) = 2m -2b -2< \abs{\overline \tau_{i^\pi}^b}, $$
as claimed.
\end{proof} 

Hence 
$\abs{w_{i+1}(\phi)} >  \abs{\overline v} + \abs{\overline \tau_{i^\pi}^{b+1}} +  \abs{w_{i}(\phi)}.$

For all $j \in [1{\uparrow}i-1]\cup[i+2{\uparrow}n]$, $\tau_j^{\overline\sigma_i\phi} = \tau_j^\phi$; hence, 
$\abs{w_j(\overline\sigma_i\phi)}  =  \abs{w_{j}(\phi)}$.

Since $\tau_{i+1}^{\overline\sigma_i\phi} = \tau_{i}^\phi$, 
we see  that  $\abs{w_{i+1}(\overline \sigma_i\phi)}  =  \abs{w_{i}(\phi)}$.

By~\eqref{eq:wsss}, $w_{i+1}(\phi) \overline w_{i}(\phi)\overline \tau_{i^{\pi}} 
=  \overline v\,\,\overline \tau_{i^{\pi}}^{b+1}$; hence
$$\tau_{i}^{\overline\sigma_i\phi} = (\tau_{i+1}^{\overline \tau_{i}})^{\phi} 
= (\tau_{(i+1)^\pi}^{w_{i+1}(\phi)})^{(\overline \tau_{i^\pi}^{w_{i}(\phi)})}
 = \tau_{i^\pi}^{ \overline v \,\,\overline \tau_{i^{\pi}}^{b+1} w_{i}(\phi)}.$$
Hence, $\abs{w_{i}(\overline \sigma_i\phi)}  = \abs{ \overline v \,\,\overline \tau_{i^{\pi}}^{b+1} w_{i}(\phi)}
\le \abs{\overline v} + \abs{\overline \tau_{i^{\pi}}^{b+1}}  +\abs{w_{i}(\phi)} < \abs{w_{i+1}(\phi)}$.

It now follows that $\norm{\overline\sigma_i\phi} < \norm{\phi}$, and (ii) is proved.

\medskip

(iii). If $\phi \ne 1$, we choose a distinguished element of $[1{\uparrow}n]$ as follows.

If, for some $i \in [1{\uparrow}n]$, $\tau_{i^\pi}^{a_i+1+b_i} =  1$,
we take any such $i$ to be our
distinguished element of $[1{\uparrow}n]$.

Consider then the case where, for all $i \in [1{\uparrow}n]$,
$\tau_{i^\pi}^{a_i+1+b_i} \ne  1$.  Thus, there is no further cancellation in
$\Pi  \tau^\phi_{[1{\uparrow}n]}$.
Since $\phi$ fixes $\Pi  \tau_{[1{\uparrow}n]}$, it is not difficult to
see that, for all $i \in [1{\uparrow}n]$, $\tau_{i^\pi}^{a_i+1+b_i} = \tau_i$.
Since $\phi \ne 1$,  it is then not difficult to show that there exists some
$i \in [1{\uparrow}n]$ such that $(a_i,b_i) \ne (0,0)$.  We take any such $i$ to be our
distinguished element of $[1{\uparrow}n]$.

Let $i$ denote our distinguished element of $[1{\uparrow}n]$.  

Notice that  $(a_i,b_i) \ne (0,0)$ and that $\tau_{i^\pi}^{a_i+1+b_i} \in \{1,\tau_{i^\pi}\}$.
Hence, $a_i+1+b_i \in \{m, m+1\}$, and, hence, $b_i \in \{m-a_i-1,m-a_i\}$.

\medskip

\noindent\textbf{Case 1.} $a_i \in [\lfloor\frac{m}{2}\rfloor {\uparrow} m-1]$. 

Here, $i \in [2{\uparrow}n]$ and, by (i),  $\norm{\sigma_{i-1}\phi} < \norm{\phi}$.

\medskip

\noindent\textbf{Case 2.} $a_i \in[0{\uparrow}\lfloor\frac{m-2}{2}\rfloor]$

Here, $m-a_i-1 \in [m-1{\downarrow}\lfloor\frac{m+1}{2}\rfloor]$,
and, hence,  $b_i \in [\lfloor\frac{m+1}{2}\rfloor {\uparrow}m-1]$.
Here, $i \in [1{\uparrow}n-1]$ and, by (ii),  $\norm{\overline \sigma_i\phi} < \norm{\phi}$.
\end{proof}

\begin{theorem} Let $n\ge 1$, $m \ge 2$.  
The natural map $\Out^+_{0,1,n} \to \Out^+_{0,1,n^{(m)}}$ is an isomorphism, and, 
hence, the natural map $\Out_{0,1,n} \to \Out_{0,1,n^{(m)}}$ is an isomorphism.

With {\normalfont Notation~\ref{not:powers}},  the maps 
$\Out_{0,1,p_1\perp p_2 \perp \cdots \perp p_N} 
\to 
\Out_{0,1,p_1^{(m)}\perp p_2^{(m)} \perp \cdots \perp p_N^{(m)}},$
and $
\Out^+_{0,1,p_1\perp p_2 \perp \cdots \perp p_N} 
\to 
\Out^+_{0,1,p_1^{(m)}\perp p_2^{(m)} \perp \cdots \perp p_N^{(m)}}$
are isomorphisms.  \hfill\qed
\end{theorem}

The following is essentially 
an algebraic translation of a part of a topological argument in~\cite[Section~3]{PerronVannier}.

\begin{proposition}\label{prop:faith}
With {\normalfont Notation~\ref{not:powers}},  let $H$ be a subgroup of
$$\Sigma_{0,1,p_1^{(m_1)}\perp p_2^{(m_2)} \perp \cdots \perp p_N^{(m_n)}}$$
of finite index, and let $A$ be the subgroup of 
$$\Out_{0,1,p_1^{(m_1)}\perp p_2^{(m_2)} \perp \cdots \perp p_N^{(m_n)}}$$ consisting of
elements which map $H$ to itself.  Then,
either the induced map
$A \to \Aut(H)$
is injective or $(n,N,m_1) = (2,1,2)$.
\end{proposition}

\begin{proof} 
Suppose that $\phi \in \Out_{0,1,p_1^{(m_1)}\perp p_2^{(m_2)} \perp \cdots \perp p_N^{(m_n)}}$,
and that $\phi$ acts as the identity on~$H$.  We shall show that $\phi = 1$ or  $(n,N,m_1) = (2,1,2)$. 

Let $G =  \Sigma_{0,1,p_1^{(m_1)}\perp p_2^{(m_2)} \perp \cdots \perp p_N^{(m_n)}}$.

  For any $g \in G$, right multiplication by $g$ permutes the
elements of the finite set $H\backslash G$, so there exists some positive integer $k$ such that
 $g^k$ acts trivially on $H\backslash G$.  In particular, 
$Hg^k = H$ and, hence, $g^{k} \in H$.

Hence, there exists some positive integer $k$ such that $(\Pi \tau_{[1{\uparrow}n]})^k \in H$.
Now $(\Pi \tau_{[1{\uparrow}n]})^\phi = (\Pi \tau_{[1{\uparrow}n]})^{\epsilon}$ 
for some $\epsilon \in \{1,-1\}$, and, hence,
$$(\Pi \tau_{[1{\uparrow}n]})^k = (\Pi \tau_{[1{\uparrow}n]})^{k\phi} = 
(\Pi \tau_{[1{\uparrow}n]})^{\phi k} = (\Pi \tau_{[1{\uparrow}n]})^{\epsilon k}= (\Pi \tau_{[1{\uparrow}n]})^{ k\epsilon}.  $$
Since $\Pi \tau_{[1{\uparrow}n]}$ has infinite order in $G$, we see that $\epsilon = 1$.  Thus
$\phi$ fixes $\Pi \tau_{[1{\uparrow}n]}$.

Consider any $i \in [1{\uparrow}n]$.  Since $(\Pi \tau_{[1{\uparrow}n]})^{\tau_i} \in G$,  there exists some
positive integer $k$ such that  $(\Pi \tau_{[1{\uparrow}n]})^{\tau_ik} \in H$.  
Hence, 
$$(\Pi \tau_{[1{\uparrow}n]})^{k\tau_i} =(\Pi \tau_{[1{\uparrow}n]})^{\tau_ik} =  (\Pi \tau_{[1{\uparrow}n]})^{\tau_ik\phi} =  
(\Pi \tau_{[1{\uparrow}n]})^{k \phi \tau_i^\phi}
 =(\Pi \tau_{[1{\uparrow}n]})^{k\tau_i^\phi}.$$ 
Hence $\tau_i^\phi\overline \tau_i$ commutes with $(\Pi \tau_{[1{\uparrow}n]})^k$. 
 A straightforward normal-form argument
shows that $\tau_i^\phi\overline \tau_i \in \gen{\Pi \tau_{[1{\uparrow}n]}}$.

Hence there exists an integer $j$ such that $\tau_i^\phi = (\Pi \tau_{[1{\uparrow}n]})^j\tau_i$.
Since $\tau_i^\phi$ is a conjugate of $\tau_{i^{\pi(\phi)}}$, the cyclically-reduced form of 
$(\tau_{[1,n]})^j\tau_i$ is $\tau_{i^{\pi(\phi)}}$.  Either $j=0$, or
there must be cyclic cancellation, and a straightforward analysis then shows that
$(n,N,m_1) = (2,1,2)$.  Since $i$ was arbitrary, this completes the proof.
\end{proof}

\section{The $\B_{n+1}$-group $\Phi_{n}$}\label{sec:finite index}

\begin{notation}\label{not:F} 
Recall that
$\Sigma_{0,1,(n+1)^{(2)}} = C_2^{\ast (n+1)} = \gen{\tau_{[1{\uparrow}n+1]} 
\mid \tau_{[1{\uparrow}n+1]}^2 =  1}.$
We define
$\Phi_{n}$ to be the $\B_{n+1}$-group consisting of 
the set of elements of  $\Sigma_{0,1,(n+1)^{(2)}}$
which have even exponent sum in the $\tau_i$.
It is not difficult to see that $\Phi_{n}$ is  a free group of rank $n$,
and that there is induced a map from $\Out_{0,1,n+1} = \Out_{0,1,(n+1)^{(2)}}$ to $\Aut \Phi_{n}$.
Since $\B_{n+1} = \Out^+_{0,1,n+1} = \Out^+_{0,1,(n+1)^{(2)}}$,
$\Phi_{n}$ has a $\B_{n+1}$-action; we say that $\Phi_{n}$ is a $\B_{n+1}$-group, 
and that $\Phi_{n}$ is a $\B_{n+1}$-subgroup of $\Sigma_{0,1,(n+1)^{(2)}}$.

Proposition~\ref{prop:faith} shows that, if $n \ne 1$, then the map from
$\Out_{0,1,n+1} = \Out_{0,1,(n+1)^{(2)}}$ to $\Aut \Phi_{n}$ is injective, 
and we say that the $\B_{n+1}$-action is {\it faithful}, 
and that $\Phi_{n}$ is a {\it faithful} $\B_{n+1}$-group. 
\hfill\qed
\end{notation}

Over the course of this section, we shall choose various free generating sets of $\Phi_{n}$ to
obtain interesting actions.  In the next two examples, we identify 
$\Sigma_{g,1,0}$ with $\Phi_{2g}$ and  $\Sigma_{g,2,0}$ with $\Phi_{2g+1}$.

\begin{example}\label{ex:g1}  Now that algebraic proofs of the requisite theorems are known to us, 
let us review~\cite[Example~15.6]{DF05} which was an algebraic approximation of results in~\cite[Section~3]{PerronVannier}.

 Let $g \in \naturals$.  Let 
$$\Sigma_{g,1,0} := \gen{x_1,y_1,\ldots, x_g,y_g, z_1\mid [x_1,y_1] \cdots [x_g,y_g]z_1 = 1},$$
where the commutator $[x,y]$ of group elements $x$, $y$ is $\overline x\,\, \overline y x y$.
Let $\Out^+_{g,1,0}$ denote the group of all automorphisms of 
$\Sigma_{g,1,0}$ which fix $z_1$.
Then $\Sigma_{g,1,0}$ is free of rank $2g$ with ordered free generating set $(x_1,y_1, \ldots, x_g,y_g)$, 
and $\Out^+_{g,1,0}$ is the group of all automorphisms of $\Sigma_{g,1,0}$ which fix 
$[x_1,y_1]\cdots[x_g,y_g]$.

We now recall some Dehn-twist elements of $\Out^+_{g,1,0}$ 
from Definitions~3.10 and Remarks~5.1 of~\cite{DF05}.

 For each $i \in [1{\uparrow}g]$, 
we define  $\alpha_i$, $\beta_i \in \Out_{g,1,0}^+$ by
\medskip

\centerline{
\begin{tabular}
{
>{$}r<{$} 
@{} 
>{$}l<{$} 
@{\hskip .5cm}  
>{$}c<{$} 
@{\hskip .5cm}  
>{$}c<{$}
@{\hskip .5cm}  
>{$}r<{$} 
@{\hskip 0cm} 
>{$}l<{$}
}
\setlength\extrarowheight{3pt}
&\hskip-.1cm\underline{\scriptstyle k\in[1{\uparrow}i-1]}&&&&
\hskip-1.4cm\underline{\scriptstyle k \in [i+1{\uparrow}g]}
\\[.15cm]
(&x_k\hskip .5cm y_k& x_i&y_i&x_k\hskip .5cm y_k&)^{\alpha_i}\\  
=(&x_k\hskip .5cm y_k&\overline y_ix_i&y_i&x_k\hskip .5cm y_k &),
\end{tabular}
\quad \hskip-5pt \text{and} \hskip -5pt \quad
\begin{tabular}
{
>{$}r<{$} 
@{} 
>{$}l<{$} 
@{\hskip .5cm}  
>{$}c<{$} 
@{\hskip .5cm}  
>{$}c<{$}
@{\hskip .5cm}  
>{$}r<{$} 
@{\hskip 0cm} 
>{$}l<{$}
}
\setlength\extrarowheight{3pt}
&\hskip-.1cm\underline{\scriptstyle k\in[1{\uparrow}i-1]}&&&&
\hskip-1.4cm\underline{\scriptstyle k \in [i+1{\uparrow}g]}
\\[.15cm]
(&x_k\hskip .5cm y_k& x_i&y_i&x_k\hskip .5cm y_k&)^{\beta_i}\\  
=(&x_k\hskip .5cm y_k&x_i&x_iy_i&x_k\hskip .5cm y_k &).
\end{tabular}}

\bigskip

 For each $i \in [1{\uparrow}g-1]$, 
we define  $\gamma_i \in \Out_{g,1,0}^+$ by
\medskip

\centerline{
\begin{tabular}
{
>{$}r<{$} 
@{} 
>{$}l<{$} 
@{\hskip .5cm}  
>{$}c<{$} 
@{\hskip .5cm}  
>{$}c<{$} 
@{\hskip .5cm}  
>{$}c<{$} 
@{\hskip .5cm}  
>{$}c<{$}
@{\hskip .5cm}  
>{$}r<{$} 
@{\hskip 0cm} 
>{$}l<{$}
}
\setlength\extrarowheight{3pt}
&\hskip-.1cm\underline{\scriptstyle k\in[1{\uparrow}i-1]}&&&&&&
\hskip-1.4cm\underline{\scriptstyle k \in [i+2{\uparrow}g]}
\\[.15cm]
(&x_k\hskip .5cm y_k& x_i&y_i&x_{i+1}&y_{i+1}&x_k\hskip .5cm y_k&)^{\gamma_i}\\  
=(&x_k\hskip .5cm y_k&y_{i+1}^{x_{i+1}}\overline y_{i}x_{i}
&y_{i}^{\overline y_{i+1}^{x_{i+1}}}
&x_{i+1}y_{i}\overline y_{i+1}^{x_{i+1}}
&y_{i+1}&x_k\hskip .5cm y_k &).
\end{tabular}}

\bigskip

Let us identify $\Sigma_{g,1,0}$ with $\Phi_{2g}$ via

\medskip

\centerline{
\begin{tabular}
{
>{$}r<{$} 
@{} 
>{$}l<{$} 
@{\hskip .5cm}  
>{$}c<{$} 
@{\hskip .7cm}   
>{$}r<{$} 
@{\hskip 0cm} 
>{$}l<{$}
}
\setlength\extrarowheight{3pt}
&&\hskip-2.9cm \underline{\hskip1.8cm \scriptstyle k\in[1{\uparrow}g]\hskip1.8cm}&&
\\[.15cm]
(&x_k&y_k& z_1&)^{\Sigma_{g,1,0} \overset{\sim}{\to} \Phi_{2g}}\\  
=(&\Pi\tau_{[2k+1{\downarrow}2k]}&\tau_{2k+1}\Pi\tau_{[1{\uparrow}2k+1]}&z_1^2 &).
\end{tabular}}

\bigskip

\noindent Notice that $[x_k,y_k] = \overline x_k \overline y_k x_ky_k $ is then identified with 
$$\Pi\tau_{[2k{\uparrow}2k+1]} \Pi  \tau_{[2k+1{\downarrow}1]}\tau_{2k+1} \Pi\tau_{[2k+1{\downarrow}2k]}\tau_{2k+1}
\Pi\tau_{[1{\uparrow}2k+1]}$$
which equals $\Pi \tau_{[2k-1{\downarrow}1]} \Pi\tau_{[2k{\uparrow}2k+1]} \Pi\tau_{[1{\uparrow}2k+1]}$.
Hence $\mathop{\Pi}\limits_{k\in[1{\uparrow}g]}[x_k,y_k]$ is identified with $(\Pi\tau_{[1{\uparrow}2g+1]})^2$.

This corresponds to the surface of genus $g$ with one boundary component arising as a 
two-sheeted branched cover of a sphere with one boundary component and $2g+1$ double points.
Then $\B_{2g+1} = \Out^+_{0,1,2g+1} = \Out^+_{0,1,(2g+1)^{(2)}}$ becomes embedded in $\Out^+_{g,1,0}$ via
the homomorphism represented as 
$$\begin{pmatrix}
\sigma_1&\sigma_2&\sigma_3&\sigma_4&\sigma_5&\cdots&\sigma_{2g-2}&\sigma_{2g-1}&\sigma_{2g}\\
\alpha_1&\beta_1&\gamma_1&\beta_2&\gamma_2&\cdots&\beta_{g-1}&\gamma_{g-1}&\beta_g
\end{pmatrix}.$$
\vskip -.8cm \hfill\qed \vskip .5cm
\end{example}

Clearly, in the preceding example, the subgroup $\B_{2g}$ of $\B_{2g+1}$ is also embedded in 
$\Out_{g,1,0}$, but it is more natural to remove from the surface  a handle containing the boundary component
(a sphere with three boundary components or a `pair of pants'), and embed 
$\B_{2g}$ in $\Out_{g-1,2,0}$, as follows.

\begin{example}\label{ex:g2} Now that algebraic proofs of the requisite theorems are known to us, 
let us review~\cite[Example~15.7]{DF05} which was an algebraic approximation of 
results in~\cite[Section~3]{PerronVannier}.

 Let $g \in \naturals$.  Let 
$$\Sigma_{g,2,0} := \gen{x_{[1{\uparrow}g]}, y_{[1{\uparrow}g]}, z_{[1{\uparrow}2]}
\mid (\textstyle \mathop{\Pi}\limits_{i\in[1{\uparrow}g]}[x_i,y_i]) \mathop{\Pi} z_{[1{\uparrow}2]} = 1}.$$
Recall that   $[x,y]:= \overline x\,\, \overline y x y$.
Then $\Sigma_{g,2,0}$ is free of rank $2g+1$ with free generating set
 $(x_{[1{\uparrow}g]},y_{[1{\uparrow}g]}, z_1)$ and distinguished element $z_2$ such that
$\overline z_2 = (\mathop{\Pi}\limits_{i\in[1{\uparrow}g]}[x_i,y_i]) z_1$.
Let $\Out^+_{g,1\perp 1,0}$ denote the group of all automorphisms of 
$\Sigma_{g,2,0}\ast \gen{e_1\mid\quad}$ which map $\Sigma_{g,2,0}$ to itself, and
fix $z_1^{e_1}$ and $z_2$.  It can be shown
 that $\Out^+_{g,1\perp 1,0}$ acts faithfully on the subset
$\Sigma_{g,2,0} \cup \Sigma_{g,2,0}e_1$ of $\Sigma_{g,2,0}\ast \gen{e_1\mid\quad}$.

Here, $e_1$ represents an arc from the base-point of one
boundary component, to the base-point of the other boundary component. 
Karen Vogtmann calls such an arc a  `tether joining the basepoint to the second boundary component'. 
For any surface-with-boundaries, A'Campo~\cite[Section 4, Remarque 6]{ACampo},~\cite[p.232]{PerronVannier}
 identifies basepoints of 
all the boundary components, 
which makes tethers into loops,
to obtain a topological quotient space whose fundamental group is acted on, faithfully,
 by the mapping-class group of the 
sur\-face\d1with\d1bound\-aries.

We now recall some Dehn-twist elements of $\Out^+_{g,1\perp1,0}$ 
from Definitions~3.10 and Remarks~5.1 of~\cite{DF05}.

 For each $i \in [1{\uparrow}g]$, 
we define  $\alpha_i$, $\beta_i \in \Out_{g,1\perp1,0}^+$ by

\medskip

\centerline{
\begin{tabular}
{
>{$}r<{$} 
@{} 
>{$}l<{$} 
@{\hskip .5cm}  
>{$}c<{$} 
@{\hskip .5cm}  
>{$}c<{$} 
@{\hskip .5cm} 
>{$}c<{$} 
@{\hskip .5cm} 
>{$}c<{$}
@{\hskip .5cm}  
>{$}r<{$} 
@{\hskip 0cm} 
>{$}l<{$}
}
\setlength\extrarowheight{3pt}
&\hskip-.1cm\underline{\scriptstyle k\in[1{\uparrow}i-1]}&&&
\hskip-.1cm\underline{\scriptstyle k \in [i+1{\uparrow}g]}&&&
\\[.15cm]
(&x_k\hskip.5cm y_k& x_i&y_i&x_k\hskip.5cm y_k&z_1&e_1&)^{\alpha_i}\\  
=(&x_k\hskip.5cm y_k&\overline y_ix_i&y_i&x_k\hskip.5cm y_k &z_1&e_1&),
\end{tabular}
}

\bigskip

\centerline{
\begin{tabular}
{
>{$}r<{$} 
@{} 
>{$}l<{$} 
@{\hskip .5cm}  
>{$}c<{$} 
@{\hskip .5cm}  
>{$}c<{$} 
@{\hskip .5cm} 
>{$}c<{$} 
@{\hskip .5cm} 
>{$}c<{$}
@{\hskip .5cm}  
>{$}r<{$} 
@{\hskip 0cm} 
>{$}l<{$}
}
\setlength\extrarowheight{3pt}
&\hskip-.1cm\underline{\scriptstyle k\in[1{\uparrow}i-1]}&&&
\hskip-.1cm\underline{\scriptstyle k \in [i+1{\uparrow}g]}&&&
\\[.15cm]
(&x_k\hskip.5cm y_k& x_i&y_i&x_k\hskip.5cm y_k&z_1&e_1&)^{\beta_i}\\  
=(&x_k\hskip.5cm y_k&x_i&x_iy_i&x_k\hskip.5cm y_k&z_1&e_1 &).
\end{tabular}
}

\bigskip

For each $i \in [1{\uparrow}g-1]$, 
we define  $\gamma_i \in \Out_{g,1\perp1,0}^+$ by
\medskip

\centerline{
\begin{tabular}
{
>{$}r<{$} 
@{} 
>{$}l<{$} 
@{\hskip .5cm}  
>{$}c<{$} 
@{\hskip .5cm}  
>{$}c<{$} 
@{\hskip .5cm}  
>{$}c<{$} 
@{\hskip .5cm}  
>{$}c<{$} 
@{\hskip .5cm} 
>{$}c<{$} 
@{\hskip .5cm} 
>{$}c<{$}
@{\hskip .5cm}  
>{$}r<{$} 
@{\hskip 0cm} 
>{$}l<{$}
}
\setlength\extrarowheight{3pt}
&\hskip-.1cm\underline{\scriptstyle k\in[1{\uparrow}i-1]}&&&&&
\hskip-.1cm\underline{\scriptstyle k \in [i+2{\uparrow}g]}&&&
\\[.15cm]
(&x_k\hskip.5cm y_k& x_i&y_i&x_{i+1}&y_{i+1}&x_k\hskip.5cm y_k&z_1&e_1&)^{\gamma_i}\\  
=(&x_k\hskip.5cm y_k&y_{i+1}^{x_{i+1}}\overline y_{i}x_{i}
&y_{i}^{\overline y_{i+1}^{x_{i+1}}}
&x_{i+1}y_{i}\overline y_{i+1}^{x_{i+1}}
&y_{i+1}&x_k\hskip.5cm y_k &z_1&e_1&),
\end{tabular}
}

\bigskip

\noindent and we define $\gamma_g \in \Out_{g,1\perp1,0}^+$

\medskip

\centerline{
\begin{tabular}
{
>{$}r<{$} 
@{} 
>{$}l<{$} 
@{\hskip .5cm}  
>{$}c<{$} 
@{\hskip .5cm}  
>{$}c<{$} 
@{\hskip .5cm}  
>{$}c<{$} 
@{\hskip .5cm}  
>{$}r<{$} 
@{\hskip 0cm} 
>{$}l<{$}
}
\setlength\extrarowheight{3pt}
&\hskip-.1cm\underline{\scriptstyle k\in[1{\uparrow}i-1]}&&&&&
\\[.15cm]
(&x_k\hskip.5cm y_k& x_{g}&y_{g}&z_1&e_1&)^{\gamma_g}\\  
=(&x_k\hskip.5cm y_k&\overline z_{1} 
\overline y_{g}x_{g}&y_{g}^{z_1}&z_1^{y_gz_1}&\overline z_1 \overline y_g e_1&).
\end{tabular}
}

\bigskip

Let us identify $\Sigma_{g,2,0}$ with $\Phi_{2g+1}$ and $\Sigma_{g,2,0} \cup \Sigma_{g,2,0}e_1$ 
with $\Sigma_{0,1,(2g+2)^{(2)}}$ via the map $\Sigma_{g,2,0}\ast \gen{e_1} \to \Sigma_{0,1,(2g+2)^{(2)}}$ 
determined by

\bigskip

\centerline{
\begin{tabular}
{
>{$}r<{$} 
@{} 
>{$}l<{$} 
@{\hskip .5cm}  
>{$}c<{$} 
@{\hskip .5cm}  
>{$}c<{$} 
@{\hskip .5cm}  
>{$}c<{$} 
@{\hskip .5cm}  
>{$}r<{$} 
@{\hskip 0cm} 
>{$}l<{$}
}
\setlength\extrarowheight{3pt}
&&\hskip-2.6cm \underline{\hskip1.7cm \scriptstyle k\in[1{\uparrow}g]\hskip1.7cm}&&&&
\\[.15cm]
(&\phantom{\Pi \tau 2} x_k&y_k& z_1&e_1&z_2&)^{\Sigma_{g,2,0}\ast \gen{e_1} \to \Sigma_{0,1,(2g+2)^{(2)}}}\\  
=(&\Pi\tau_{[2k+1{\downarrow}2k]}&\tau_{2k+1}\Pi\tau_{[1{\uparrow}2k+1]}&z_1^{\tau_{2g+2}}&\tau_{2g+2}&z_1 &).
\end{tabular}
}

\bigskip

\noindent This corresponds to the surface of genus $g$ with two boundary components arising as a 
two-sheeted branched cover of a sphere with one boundary component and $2g+2$ double points.
Now $\B_{2g+2} = \Out^+_{0,1,2g+2} = \Out^+_{0,1,(2g+2)^{(2)}}$ is embedded in $\Out^+_{g,1\perp1,0}$ via
a homomorphism represented as
$$\left(\hskip-.1cm \begin{array}{cccccccccc}
\sigma_1&\sigma_2&\sigma_3&\sigma_4&\sigma_5&\cdots&\sigma_{2g-2}&\sigma_{2g-1}&\sigma_{2g}&\sigma_{2g+1}\\
\alpha_1&\beta_1&\gamma_1&\beta_2&\gamma_2&\cdots&\beta_{g-1}&\gamma_{g-1}&\beta_g&\gamma_g
\end{array}\hskip-.2cm\right).$$
For $g \ge 1$, Proposition~\ref{prop:faith} shows that this is an embedding. 
In the case where $g=0$, the interpretation of the notation is as follows:
$\sigma_1$ is mapped to $\gamma_0$; $\gamma_0$ fixes $z_1$ and sends 
$e_1$ to $\overline z_1e_1$.
\hfill\qed
\end{example}

Clearly, in the preceding example, the subgroup $\B_{2g+1}$ of $\B_{2g+2}$ is also embedded in 
$\Out^+_{g,1\perp1,0}$, but it is more natural to remove from the surface a disc containg 
the two boundary components (a sphere with three boundary components or a `pair of pants'), and embed 
$\B_{2g+1}$ in $\Out^+_{g,1,0}$, as in Example~\ref{ex:g1}.

\medskip

We next discuss the  Perron-Vannier isomorphism  $\B_{n+1}\ltimes \Phi_{n} \simeq \Artin\gen{D_{n+1}}$ for $n \ge 1$.
The following was shown to us by Mladen Bestvina.

\begin{lemma}\label{lem:D} Let $n \ge 2$.  Then, 
$\Artin\langle D_n\rangle$ has a unique automorphism 
$\upsilon$ of order two which fixes $d_1,\ldots,d_{n-2}$ and interchanges $d_{n-1}$ and $d_n$.
 The semidirect product $\Artin\langle D_n\rangle\rtimes \gen{\upsilon}$ has presentation
$$\Artin\langle \,\,d_1 
 \ray  d_2   \ray  
\cdots  \ray d_{n-3} \ray d_{n-2} \ray d_{n-1}  \doubleray  \upsilon \mid \upsilon^2 = 1\,\,\rangle.
$$
\end{lemma}

\begin{proof} Notice that 
$\gen{d_{n-1}, d_n, \upsilon \mid \upsilon^2 = 1, d_{n-1}^\upsilon = d_n, d_{n-1}d_n = d_nd_{n-1}}$
is isomorphic to $\gen{d_{n-1},  \upsilon \mid \upsilon^2 = 1,  d_{n-1}d_{n-1}^\upsilon = d_{n-1}^\upsilon d_{n-1}}$, 
and the latter is $\Artin\langle \,\,  d_{n-1}  \doubleray  \upsilon \mid \upsilon^2 = 1\,\,\rangle$.  The result
now follows easily.
\end{proof}

Part of the following appears in~\cite{PerronVannier} and~\cite{CrispParis1}.

\begin{theorem}[Perron-Vannier~\cite{PerronVannier}]\label{th:PV}
Let $n \ge 2$.  The semidirect product $\B_n \ltimes \Phi_{n-1}$ has
presentation
\begin{align*} 
&\Artin\langle \sigma_1 
 \ray  \sigma_2   \ray  
\cdots \ray \sigma_{n-3}  
\begin{aligned}
[b]\sigma_{n-1}&\tau_n\tau_{n-1}\\[-2pt]\rule[-4pt]{0.4pt}{.6cm}\hskip2pt&\\[-3pt]\ray\sigma&\null_{n-2}\ray
\end{aligned} \sigma_{n-1}\rangle \simeq \Artin\gen{D_n}.
\intertext{Hence, $\B_n \ltimes \Phi_{n-1}$ has a unique automorphism $\upsilon$ of order two which fixes
$\sigma_1,\ldots,\sigma_{n-2}$ and interchanges $\sigma_{n-1}$ and $\sigma_{n-1}\tau_n\tau_{n-1}$.
The double semidirect product $(\B_n \ltimes \Phi_{n-1})\rtimes \gen{\upsilon}$ has presentation}
&
\Artin\langle \sigma_1 
 \ray  \sigma_2   \ray  
\cdots  \ray \sigma_{n-3} \ray \sigma_{n-2} \ray \sigma_{n-1}  \doubleray  \upsilon \mid \upsilon^2 = 1\rangle.
\end{align*}
\end{theorem}

\begin{proof} By Corollary~\ref{cor:Manfred}, we have a presentation
\begin{align*} 
\B_n \ltimes \Sigma_{0,1,n} \,\,\,&= \,\,\, \Artin\langle \sigma_1 
 \ray  
\cdots  \ray \sigma_{n-1}  \doubleray \,\, \overline t_n\rangle.
\end{align*}
If we impose the relation $t_n^2 = 1$, we transform $\B_n \ltimes \Sigma_{0,1,n}$ into
$\B_n \ltimes \Sigma_{0,1,n^{(2)}}$, and we have
\begin{align*} 
\B_n \ltimes \Sigma_{0,1,n^{(2)}} \,\,\,&= \,\,\, \Artin\langle \sigma_1 
 \ray  
\cdots  \ray \sigma_{n-1}  \doubleray \,\, \tau_n \mid \tau_n^2 = 1\rangle.
\end{align*}
Here, there exists a retraction to $\gen{\tau_n}$ with kernel the normal subgroup generated by
 $\sigma_{[1{\uparrow}n-1]}$.  This normal subgroup contains
$\sigma_{i}^{\tau_{i+1}} = \sigma_{i}\tau_{i+1}\tau_{i}$ for all $i \in[1{\uparrow}n-1]$.
By Lemma~\ref{lem:D}, the normal subgroup has presentation
\begin{align*} 
\B_n \ltimes \Phi_{n-1} \,\,\,&= \,\,\, \Artin\langle \,\,\sigma_1 
 \ray  
\cdots  \ray \sigma_{n-3}  \begin{aligned}
[b] \sigma&\null_{n-1}^{\tau_n}&\\[-2pt]\rule[-4pt]{0.4pt}{.6cm}\hskip2pt&\\[0pt]\ray \sigma&\null_{n-2}\ray
\end{aligned} \hskip -0.3cm \sigma_{n-1}  \,\,\rangle,
\end{align*}
and this agrees with the desired presentation.
\end{proof}

\begin{remarks}
Corollary~\ref{cor:Manfred} says that, for $n \ge 1$, we can go down by index 
$n+1$ from $\Artin\langle A_n\rangle$ by squaring the last generator,
and arrive at
$\Artin\langle B_n\rangle \simeq \Artin\langle A_{n-1}\rangle \ltimes \Sigma_{0,1,n}$.

Theorem~\ref{th:PV} says that, for $n \ge 2$, we can kill the square of the new last generator, go down by index 2,
 and arrive at
$\Artin\langle D_n \rangle \simeq \Artin\langle A_{n-1}\rangle \ltimes \Phi_{n-1}$.
\hfill\qed
\end{remarks}

\medskip

We now record some other free generating sets of $\Phi_{n}$ which appear in the literature.

\begin{examples}\label{ex:wada}  Recall Notation~\ref{not:F}. 
In particular,  the $\B_{n+1}$-action on $\Phi_{n}$ is faithful if $n \ne 1$.  

(1). For each  $k \in [1{\uparrow}n]$,  set $x_k = \tau_{k} \tau_{k+1}$ in $\Phi_{n}$.
Then $x_{[1{\uparrow}n]}$ is a free generating set for $\Phi_{n}$, and, 
for each $i \in [1{\uparrow}n]$,  the action of $\sigma_i$ 
on $\Phi_{n}$ is determined by

\bigskip

\centerline{
\begin{tabular}{ 
>{$}r<{$} 
@{\hskip0cm}    
>{$}l<{$}   
@{\hskip0.7cm}    
>{$}c<{$}   
@{\hskip0.7cm}      
>{$}c<{$} 
@{\hskip0.7cm}     
>{$}c<{$} 
@{\hskip0.7cm}
>{$}r<{$} 
@{\hskip0cm}  
>{$}l<{$}}
\setlength\extrarowheight{3pt}
\\[-.7cm]&\hskip-.5cm\underline{\scriptstyle k \in [1{\uparrow}i-2]}& & & &&
\hskip -1cm \underline{\scriptstyle  k\in[i+2{\uparrow}n]}
\\ [.15cm]
(&x_k  &x_{i-1} & x_{i} & x_{i+1} &x_k &)^{\sigma_i}\\
=(&  x_k  &  x_{i-1}x_i & x_{i} & \overline x_{i} x_{i+1}  & x_k &),
\end{tabular}}

\medskip

\noindent interpretated appropriately for $i = 1$ and $i = n$.

\bigskip

(2).  For each  $k \in [1{\uparrow}n]$,  set $x_k = \tau_{n+1}\tau_{k}$ in $\Phi_{n}$,
Then $x_{[1{\uparrow}n]}$ is a free generating set for $\Phi_{n}$, and,  
for each $i \in [1{\uparrow}n-1]$,  $\sigma_i$   acts on $x_{[1{\uparrow}n]}$ as follows.

\bigskip

\centerline{
\begin{tabular}
{
>{$}r<{$} 
@{} 
>{$}l<{$} 
@{\hskip .4cm} 
>{$}c<{$} 
@{\hskip .8cm}  
>{$}c<{$}
@{\hskip .8cm}
>{$}r<{$} 
@{} 
>{$}l<{$}
}
\setlength\extrarowheight{3pt}
\\[-.8cm] &\hskip -.4cm\underline{\scriptstyle k \in [1{\uparrow}i-1]}  &   & &
 &\hskip -1cm\underline{\scriptstyle k \in [i+2{\uparrow}n]}
\\[.15cm]
(&x_k   & x_{i} & x_{i+1} &x_k &)^{\sigma_i}\\
=(&  x_k  & x_{i+1} & x_{i+1}\overline x_{i}x_{i+1} & x_k &).
\end{tabular}
\qquad
\begin{tabular}{ 
>{$}r<{$} 
@{\hskip0cm}    
>{$}l<{$}   
@{\hskip0.5cm}    
>{$}r<{$}   
@{\hskip0cm}  
>{$}l<{$}}
\setlength\extrarowheight{3pt}
\\[-.8cm] & \hskip -.2cm \underline{\scriptstyle  k\in[1{\uparrow}n-1]} & &
\\[.15cm]
(&\,\,\,\,\,x_k    &x_{n} &)^{\sigma_{n}}\\
=(&  x_{n-1}x_k   & x_{n} &).
\end{tabular}
}

\bigskip

(3).  We next consider the  free generating set used in the proof of~\cite[Proposition A.1(2)]{CrispParis2}.

For each  $k \in [1{\uparrow}n]$,  set $x_k = \tau_{n+1}^{\Pi\tau_{[1{\uparrow}k]}}\tau_{k+1}$ in $\Phi_{n}$.
Then $x_{[1{\uparrow}n]}$ is a free generating set for $\Phi_{n}$, and,
for each $i \in [1{\uparrow}n-1]$,  $\sigma_i$  acts on $x_{[1{\uparrow}n]}$ as follows,

\bigskip

\centerline{
\begin{tabular}
{
>{$}r<{$} 
@{} 
>{$}l<{$} 
@{\hskip .4cm} 
>{$}c<{$} 
@{\hskip .8cm}  
>{$}c<{$}
@{\hskip .8cm}
>{$}r<{$} 
@{} 
>{$}l<{$}
}
\setlength\extrarowheight{3pt}
 &\hskip -.4cm\underline{\scriptstyle k \in [1{\uparrow}i-1]}  &   & &
 &\hskip -1cm\underline{\scriptstyle k \in [i+2{\uparrow}n]}
\\[.15cm]
(&x_k   & x_{i} & x_{i+1} &x_k &)^{\sigma_i}\\
=(&  x_k  & x_i \Pi x_{[i{\uparrow}i+1]} & \Pi \overline x_{[i+1{\downarrow}i]} x_{i+1} & x_k &).
\end{tabular}}

\bigskip 

\noindent Let $w =  (\Pi x^2_{[1{\uparrow}n-1]} x_{n})^{-1}$; then $\sigma_{n}$ acts as follows. 

\bigskip

\centerline{
\begin{tabular}{ 
>{$}r<{$} 
@{\hskip0cm}    
>{$}l<{$}   
@{\hskip0.5cm}    
>{$}r<{$}   
@{\hskip0cm}  
>{$}l<{$}}
\setlength\extrarowheight{3pt}
\\[-.8cm] & \hskip 0cm \underline{\hskip.6cm\scriptstyle  k\in[1{\uparrow}n-1]\hskip.6cm} & &
\\[.15cm]
(&\hskip1cm x_k    & x_{n}\phantom{wwwwww}&)^{\sigma_{n}}\\
=(& w^{(-1)^{k}\Pi x_{[1{\uparrow}k-1]} }x_k \,\,\,  & 
 w^{(-1)^{n} \Pi x_{[1{\uparrow}n-1]}}  x_{n}w &).
\end{tabular}
}

\bigskip

(4).  By reflecting the previous example, we can invert the elements of $\sigma_{[1{\uparrow}n]}$.

For each  $k \in [1{\uparrow}n]$,  set 
$x_k = (\tau_{n+1}^{\,\,\Pi \tau_{[n{\downarrow}1]}} \tau_k)^{\Pi\tau_{[k{\uparrow}n+1]}}$ in $\Phi_{n}$.
Then $x_{[1{\uparrow}n]}$ is a free generating set for $\Phi_{n}$, and,  
for each $i \in [1{\uparrow}n-1]$,  $\sigma_i$  acts on $x_{[1{\uparrow}n]}$ as follows. 

\bigskip

\centerline{
\begin{tabular}
{
>{$}r<{$} 
@{} 
>{$}l<{$} 
@{\hskip .4cm} 
>{$}c<{$} 
@{\hskip .8cm}  
>{$}c<{$}
@{\hskip .8cm}
>{$}r<{$} 
@{} 
>{$}l<{$}
}
\setlength\extrarowheight{3pt}
 &\hskip -.4cm\underline{\scriptstyle k \in [1{\uparrow}i-1]}  &   & &
 &\hskip -1cm\underline{\scriptstyle k \in [i+2{\uparrow}n]}
\\[.15cm]
(&x_k   & x_{i} & x_{i+1} &x_k &)^{\sigma_i}\\
=(&  x_k  & x_i \Pi \overline x_{[i+1{\downarrow}i]} & \Pi   x_{[i{\uparrow}i+1]} x_{i+1} & x_k &).
\end{tabular}}
\bigskip 

\noindent Let $w =  \Pi x^2_{[1{\uparrow}n-1]} x_{n}$; then $\sigma_{n}$ acts as follows. 

\bigskip

\centerline{
\begin{tabular}{ 
>{$}r<{$} 
@{\hskip0cm}    
>{$}l<{$}   
@{\hskip0.5cm}    
>{$}r<{$}   
@{\hskip0cm}  
>{$}l<{$}}
\setlength\extrarowheight{3pt}
\\[-.8cm] & \hskip 0cm \underline{\hskip.6cm\scriptstyle  k\in[1{\uparrow}n-1]\hskip.6cm} & &
\\[.15cm]
(&\hskip 1cm x_k    &x_{n}\phantom{wwwwww} &)^{\sigma_{n}}\\
=(& w^{(-1)^{k}\Pi x_{[1{\uparrow}k-1]} }x_k \,\,\,  & 
 w^{(-1)^{n} \Pi x_{[1{\uparrow}n-1]}}  x_{n}w &).
\end{tabular}
} \vskip -.5cm  \hfill\qed \vskip .5cm
\end{examples}

\begin{history}  Let us view $\B_{n}$ as a subgroup of 
$\B_{n+1}$ by suppressing~$\sigma_{n}$.   Then
the $\B_{n+1}$-group $\Phi_{n}$ becomes a faithful $\B_{n}$-group,
 even if $n = 1$.  

Wada~\cite{Wada} defined various left actions of $\B_n$ on a free group 
of rank $n$.  All but four of them are obviously non-faithful, and
two of the remaining four actions are obviously equivalent up to changing the free generating set,
leaving three actions to be studied for faithfulness.   Shpilrain~\cite{Shpilrain}
 ingeniously used the $\sigma_1$-trichotomy to prove 
that these three are all faithful.
Crisp-Paris~\cite[Proposition~A.1(2)]{CrispParis2} 
showed that the second and third of these three Wada actions 
are equivalent up to changing the free generating set.
They correspond to Examples~\ref{ex:wada}(2), (4), above,
with $\sigma_n$ suppressed. 
Notice that our actions on the right
are the inverses of their actions on the left.  
In summary, the {\it second} and {\it third} Wada actions are
obtained by choosing suitable free generating sets of the Perron-Vannier $\B_{n+1}$-group $\Phi_n$.

The {\it first} Wada action is constructed by choosing a non-zero integer
$m$, and, for each $1\in [1{\uparrow}n-1]$, letting $\sigma_i$ act on
$\gen{x_{[1{\uparrow}n]}\mid \quad}$ via
\bigskip

\centerline{
\begin{tabular}
{
>{$}r<{$} 
@{} 
>{$}l<{$} 
@{\hskip .4cm} 
>{$}c<{$} 
@{\hskip .8cm}  
>{$}c<{$}
@{\hskip .8cm}
>{$}r<{$} 
@{} 
>{$}l<{$}
}
\setlength\extrarowheight{3pt}
 &\hskip -.4cm\underline{\scriptstyle k \in [1{\uparrow}i-1]}  &   & &
 &\hskip -1cm\underline{\scriptstyle k \in [i+2{\uparrow}n]}
\\[.15cm]
(&x_k   & x_{i} & x_{i+1} &x_k &)^{\sigma_i}\\
=(&  x_k & x_{i+1}&x_{i}^{x_{i+1}^m} & x_k &).
\end{tabular}}

\medskip

\noindent Edward Formanek has pointed out that 
$x_{[1,n]}^m$ is then a free generating set of a faithful 
$\B_n$-subgroup of $\gen{x_{[1,n]} \mid \quad}$, where faithfulness 
can be seen from the fact that the $\B_n$-action
is the standard Artin action with respect to this free generating set.
This gives a transparent proof that the first Wada action is faithful.
\hfill\qed 
\end{history}

\bigskip

\bigskip

\centerline{\LARGE\textbf{Appendix. Larue-Whitehead diagrams}}

\addvspace \baselineskip

In this appendix, we rework ideas from Larue's thesis~\cite[Chapter~2 and Appendix A]{LarueThesis94},
using combinatorial arguments to obtain a description of the
$\B_n$-orbit of $t_1$ when $n \ge 1$.   A topological treatment of similar ideas was given in~\cite{Fenn}, 
and it was arrived at independently of~\cite{LarueThesis94}.  See \cite[Chapters~5,~6]{DDRW02}.

\setcounter{section}{0}
\renewcommand\thesection{\Roman{section}}
\section{Self-homeomorphisms}\label{sec:self}
\setcounter{theorem}{1}

This section is purely motivational.  We shall
 briefly indicate the map\-ping\d1class viewpoint of  the braid group,
and the Jordan-curve nature of the Whitehead graphs of the elements in the 
$\B_n$-orbit of $t_1$ if $n \ge 1$.   

Let $\complexes$ denote the complex plane, and  $\widehat \complexes$  the 
Riemann sphere, or 
projective complex line, $\complexes\cup\{\infty\}$. For each $z \in \complexes$ and each non-negative
real number $r$, let $\mathbf{D}(z,r)$, resp. $\mathbf{D}^\circ(z,r)$,
 denote the closed, resp. open, disc in $\complexes$ 
with centre $z$ and radius $r$.

Let $S_{0,1,n}$ denote the surface formed by 
deleting from a sphere an open disc and $n$ points.
We shall think of the discs and points as being 
distinguished rather than deleted; for example, it is then meaningful to speak of
the self-homeomorphisms of $S_{0,1,n}$ as permuting the points.
We take as our model of $S_{0,1,n}$ the sphere $\widehat \complexes$
having $[1{\uparrow}n]$ as its set of $n$
distinguished points,  
and $\mathbf{D}^\circ(0,\frac{1}{2})$ as its distinguished open disc.
We are particularly interested in the set   $[0{\uparrow}n]$,
and, in our diagrams, we shall mark these points out by drawing discs of small radii around them.

For each distinguished point 
$k \in [0{\uparrow}n]$, we have a distinguished oriented tether, or arc,
 $\{k - r\mathbf{i} \mid -\infty \le r\le 0\}$,   joining $\infty$ to $k$. 
We label the right flank of this oriented arc $t_k$,  
and label the left flank $\overline t_k$;  we then cut $\widehat \complexes$ open  along these arcs
and obtain a $(2n +2)$-gon, with clockwise
 boundary label $\mathop{\Pi}\limits_{k\in[0{\uparrow}n]} (t_k \overline t_k)$; 
see Fig.~\ref{fig:circletwist}. 
We shall use $t_0$ and $z_1$ interchangeably in this section. 
Performing the boundary identifications then gives back $\widehat\complexes$.

The self-homeomorphism $\lambda$ of $\mathbf{D}(0,1)$ given by
$\lambda(re^{\mathbf{i}\theta}) := re^{\mathbf{i}(\theta - 2\pi r)}$ fixes the boundary of $\mathbf{D}(0,1)$ 
and interchanges $\frac{1}{2}$ and $-\frac{1}{2}$; see Fig.~\ref{fig:Dehntwist}. For each \phantom{  $i \in [1{\uparrow}n-1]$,}
\begin{figure}[H] 
\vspace{-.4cm}
\begin{center}
\setlength{\unitlength}{1.3mm}
\begin{picture}(55,25)
\put(0,0){\epsfig{file= 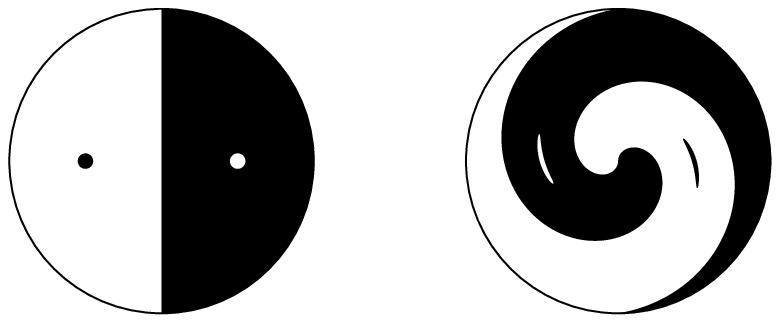,width=7.5truecm}}
\end{picture}
\caption{The map $\lambda\colon\mathbf{D}(0,1) \to \mathbf{D}(0,1) $, $re^{\mathbf{i}\theta} \mapsto re^{\mathbf{i}(\theta - 2\pi r)}$.}\label{fig:Dehntwist}
\end{center}
\vspace{-.7cm}
\end{figure}
\noindent 
   $i \in [1{\uparrow}n-1]$, let $\phi_i$ denote the self-homeomorphism of $\widehat\complexes$
which, on $\widehat\complexes - \mathbf{D}(i + \frac{1}{2},1)$, acts as the identity map,
 and, on $\mathbf{D}(i + \frac{1}{2},1)$, acts by
$z \mapsto \lambda(z- i - \frac{1}{2}) + i + \frac{1}{2}$.
Then $\phi_{[1{\uparrow}n-1]}$ generates a group $\gen{\phi_{[1{\uparrow}n-1]}}$ 
of self-homeomorphisms of $\widehat \complexes$,
which will shed light on the $\B_n$-orbit of $t_1$.
To describe the  induced action of 
$\gen{\phi_{[1{\uparrow}n-1]}}$  on the fundamental 
group of $S_{0,1,n}$, we first give $\widehat \complexes$ a CW-structure by specifying a graph 
$S^{(1)}_{0,1,n}$ embedded in $\complexes \subset \widehat\complexes$.

For each $k \in [-1{\uparrow}n]$, we have vertices $w_k := k + \frac{1}{2} -\mathbf{i}$ and 
$v_k := k + \frac{1}{2} + \mathbf{i}$, and  an oriented straight edge $f_k$ joining $w_k$ to $v_k$.
For each $k \in [0{\uparrow}n]$, we have an oriented straight edge $e_k$ joining $w_{k-1}$ to $w_k$,
and an oriented straight edge $d_k$ joining $v_{k-1}$ to $v_k$.
This completes the description of the graph $S^{(1)}_{0,1,n}$. For $n=3$, 
$S^{(1)}_{0,1,3}$ can be seen in Fig.~\ref{fig:X4}. 
Each distinguished point $k \in [0{\uparrow}n]$ is the midpoint of the  
rectangle in $\complexes$ cut out by the path $f_{k-1}d_k \overline f_k \overline e_k$. 
\begin{figure}
\vspace{0cm}
\begin{center}
\setlength{\unitlength}{.75mm}
\begin{picture}(80,45)
\put(0,0){\epsfig{file= 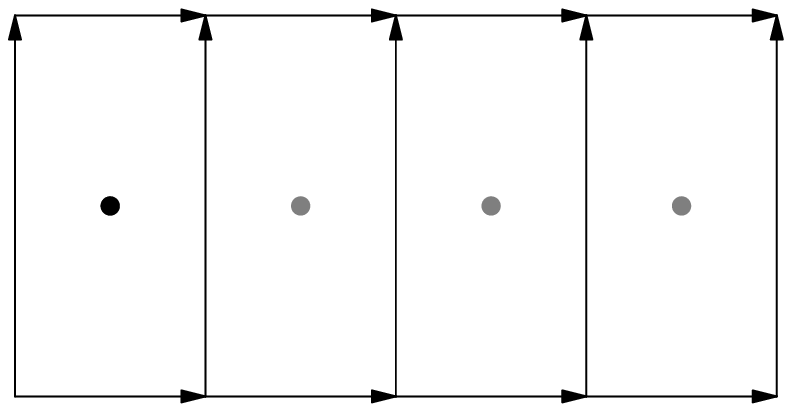,width=6truecm}}
\put(-2,-2) {\makebox(0,0)[l]{\small $w_{-1}$}}
\put(18,-2) {\makebox(0,0)[l]{\small $w_0$}}
\put(38,-2) {\makebox(0,0)[l]{\small $w_1$}}
\put(56,-2) {\makebox(0,0)[l]{\small $w_2$}}
\put(76,-2) {\makebox(0,0)[l]{\small $w_3$}}
\put(-2,42){\makebox(0,0)[l]{\small $v_{-1}$}} 
\put(18,42){\makebox(0,0)[l]{\small $v_0$}} 
\put(38,42){\makebox(0,0)[l]{\small $v_1$}} 
\put(56,42){\makebox(0,0)[l]{\small $v_2$}} 
\put(76,42){\makebox(0,0)[l]{\small $v_3$}} 
\put(10,-2) {\makebox(0,0)[l]{\small $e_0$}}
\put(30,-2) {\makebox(0,0)[l]{\small $e_1$}}
\put(48,-2) {\makebox(0,0)[l]{\small $e_2$}}
\put(68,-2) {\makebox(0,0)[l]{\small $e_3$}}
\put(10,42){\makebox(0,0)[l]{\small $d_0$}} 
\put(30,42){\makebox(0,0)[l]{\small $d_1$}} 
\put(48,42){\makebox(0,0)[l]{\small $d_2$}} 
\put(68,42){\makebox(0,0)[l]{\small $d_3$}} 
\put(-5.5,22) {\makebox(0,0)[l]{\small $f_{-1}$}}
\put(16,22) {\makebox(0,0)[l]{\small $f_0$}}
\put(35,22) {\makebox(0,0)[l]{\small $f_1$}}
\put(54,22) {\makebox(0,0)[l]{\small $f_2$}}
\put(73,22) {\makebox(0,0)[l]{\small $f_3$}}
\end{picture}
\caption{$S_{0,1,3}$.}\label{fig:X4}
\end{center}
\vspace{-.7cm}
\end{figure}

Let $\gen{S^{(1)}_{0,1,n}\mid \quad}$ denote the (free) fundamental groupoid of
   $S^{(1)}_{0,1,n}$, and let\linebreak
$\gen{S^{(1)}_{0,1,n}\mid \quad}(w_{-1},w_{-1})$ denote 
the (free) fundamental group of $S^{(1)}_{0,1,n}$ at $w_{-1}$.
The subgraph of $S^{(1)}_{0,1,n}$ spanned by 
$e_{[0{\uparrow}n]}\cup f_{[-1{\uparrow}n]}$ is a maximal subtree of $S^{(1)}_{0,1,n}$,
and   $d_{[0{\uparrow}n]}$ then determines a free 
generating set $t_{[0{\uparrow}n]}$  of $\gen{S^{(1)}_{0,1,n}\mid \quad}(w_{-1},w_{-1})$;
explicitly,  for each $k \in [0{\uparrow}n]$, 
$t_k = \Pi e_{[0{\uparrow}k-1]}f_{k-1}d_k \overline f_k \Pi \overline e_{[k{\downarrow}0]}$.

The path $f_{-1}\Pi d_{[0{\uparrow}n]}\overline f_{n}\Pi \overline e_{[n{\downarrow}0]}$
cuts out a rectangle in $\complexes$; the
complementary region in $\widehat \complexes$ together with the
graph $S^{(1)}_{0,1,n}$ is then a retract of 
$\widehat\complexes -[0{\uparrow}n]$.  Let $\sim$ denote homotopy for closed paths at $w_{-1}$ in
 $\widehat\complexes -[0{\uparrow}n]$.
We can identify the fundamental groupoid of $S_{0,1,n}$ with
$\gen{S^{(1)}_{0,1,n}\mid f_{-1}\Pi d_{[0{\uparrow}n]}\overline f_{n}\Pi \overline e_{[n{\downarrow}0]} \sim w_{-1}}.$
We then identify $\Sigma_{0,1,n}$ with the fundamental group of $S_{0,1,n}$ at $w_{-1}$, 
\begin{align*}
\Sigma_{0,1,n} &= \gen{S^{(1)}_{0,1,n}\mid f_{-1}\Pi d_{[0{\uparrow}n]}\overline f_{n}
\Pi \overline e_{[n{\downarrow}0]} \sim w_{-1}}(w_{-1},w_{-1})\\
&= \gen{ t_{[0{\uparrow}n]} \mid \Pi t_{[0{\uparrow}n]} = 1}.
\end{align*}

Consider the action of $\phi_1$ on the graph $S^{(1)}_{0,1,n}$.
For $n=3$, the result can be \phantom{seen} 
\begin{figure}[H]
\begin{center}
\setlength{\unitlength}{.75mm}
\begin{picture}(80,45)
\put(0,0){\epsfig{file= 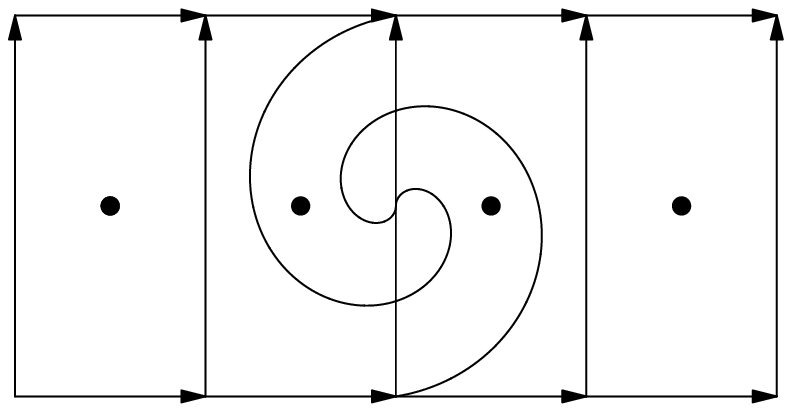,width=6truecm}}
\put(10,-2) {\makebox(0,0)[l]{\small $e_0$}}
\put(30,-2) {\makebox(0,0)[l]{\small $e_1$}}
\put(48,-2) {\makebox(0,0)[l]{\small $e_2$}}
\put(68,-2) {\makebox(0,0)[l]{\small $e_3$}}
\put(10,42){\makebox(0,0)[l]{\small $d_0$}} 
\put(30,42){\makebox(0,0)[l]{\small $d_1$}} 
\put(48,42){\makebox(0,0)[l]{\small $d_2$}} 
\put(68,42){\makebox(0,0)[l]{\small $d_3$}} 
\put(-6,22) {\makebox(0,0)[l]{\small $f_{-1}$}}
\put(30,30) {\makebox(0,0)[l]{\small $f_1^{\phi_1}$}}
\put(16,22) {\makebox(0,0)[l]{\small $f_0$}}
\put(34,6) {\makebox(0,0)[l]{\small $f_1$}}
\put(60,22) {\makebox(0,0)[l]{\small $f_2$}}
\put(73,22) {\makebox(0,0)[l]{\small $f_3$}}
\end{picture}
\caption{$S^{(1)}_{0,1,3}$ and its image under  $\phi_1$.}\label{fig:X4twist}
\end{center}
\vspace{-.7cm}
\end{figure}
\noindent
seen in Fig.~\ref{fig:X4twist}.   
The crucial point is that 
 $f_1^{\phi_1} \sim e_2 f_2 \overline d_2 \overline f_1 \overline e_1 f_0 d_1$,
and all the other elements of $S^{(1)}_{0,1,3}$ are fixed by $\phi_1$; 
this makes the action quite simple algebraically.
Then, $\overline f_1^{\,\,\phi_1} \sim \overline d_1 \overline f_0 e_1 f_1 d_2 \overline f_2 \overline e_2$, 
and, for the free generator $t_1 = e_0  f_0 d_1 \overline f_1 \overline e_{[1,0]}$, 
we have $$t_1^{\phi_1} \sim
e_0  f_0 d_1 (\overline d_1 \overline f_0 e_1 f_1 d_2 \overline f_2 \overline e_2) \overline e_{[1,0]}
\sim  e_{[0,1]} f_1 d_2 \overline f_2 \overline e_{[2,0]} = t_2.$$
Similarly,  for this element, $t_2$, 
we have
 $$t_2^{\phi_1} \sim
e_{[0,1]} (e_2 f_2 \overline d_2 \overline f_1 \overline e_1 f_0 d_1) d_2 \overline f_2 \overline e_{[2,0]} \sim
e_{[0,2]} f_2 \overline d_2 \overline f_1 \overline e_1 f_0 d_{[1,2]} \overline f_2 \overline e_{[2,0]} 
\sim \overline t_2 t_1t_2,$$
where the latter homotopy can be seen directly
by collapsing the elements of $e_{[0,2]} \cup f_{[0,2]}$, which lie in the maximal subtree.
Thus, we see that $\phi_{1}$  acts on $\Sigma_{0,1,n}$ as the automorphism $\sigma_1$.

\begin{figure}
\vspace{1.4cm}
\begin{center}
\setlength{\unitlength}{1mm}
\begin{picture}(100,75)
\put(29.6,76.5){\makebox(0,0)[l]{\small $z_1$}} 
\put(34.5,76.5){\makebox(0,0)[l]{\small $\overline z_1$}} 
\put(44.5,76.5){\makebox(0,0)[l]{\small $t_1$}} 
\put(49,76.9){\makebox(0,0)[l]{\small $\overline t_1$}} 
\put(58.5,76.5){\makebox(0,0)[l]{\small $t_2$}} 
\put(63,76.5){\makebox(0,0)[l]{\small $\overline t_2$}} 
\put(72.5,76.5){\makebox(0,0)[l]{\small $t_3$}} 
\put(77.5,76.5){\makebox(0,0)[l]{\small $\overline t_3$}} 
\put(11,38.5){\makebox(0,0)[l]{\small $z_1$}} 
\put(25,38.5){\makebox(0,0)[l]{\small $\overline z_1$}} 
\put(37.5,26.5){\makebox(0,0)[l]{\small $t_1$}} 
\put(38,14.5){\makebox(0,0)[l]{\small $\overline t_1$}}
\put(25,1){\makebox(0,0)[l]{\small $ t_2$}} 
\put(10.5,2){\makebox(0,0)[l]{\small $\overline t_2$}}  
\put(0,26.5){\makebox(0,0)[l]{\small $\overline t_3$}} 
\put(0,14.5){\makebox(0,0)[l]{\small $ t_3$}} 
\put(71,38.5){\makebox(0,0)[l]{\small $z_1$}} 
\put(85,38.5){\makebox(0,0)[l]{\small $\overline z_1$}} 
\put(97.5,26.5){\makebox(0,0)[l]{\small $t_1$}} 
\put(98,14.5){\makebox(0,0)[l]{\small $\overline t_1$}}
\put(70.5,2){\makebox(0,0)[l]{\small $\overline t_2$}}  
\put(85,1){\makebox(0,0)[l]{\small $ t_2$}} 
\put(60,26.5){\makebox(0,0)[l]{\small $\overline t_3$}} 
\put(60,14.5){\makebox(0,0)[l]{\small $ t_3$}} 
\put(25,45){\epsfig{file= 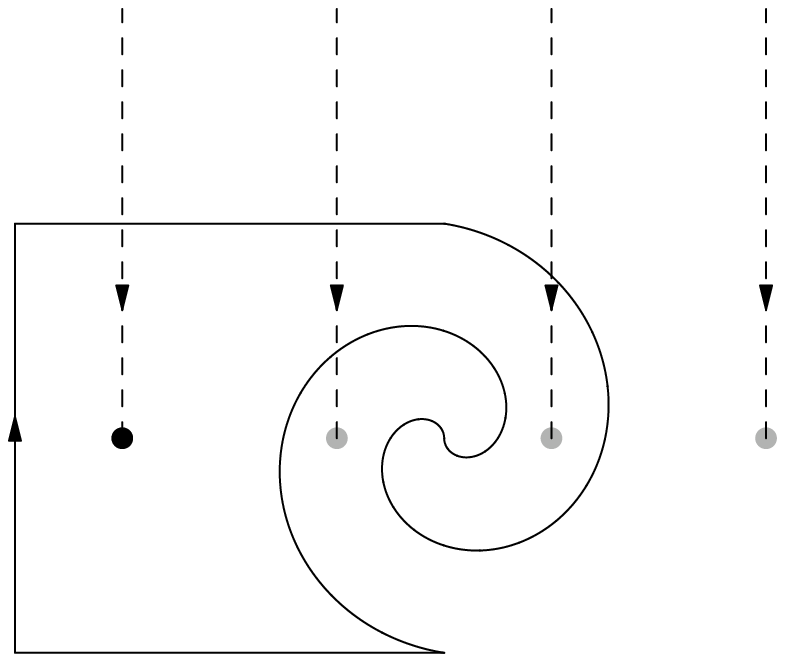,width=5.325truecm}}
\put(0,0){\epsfig{file= 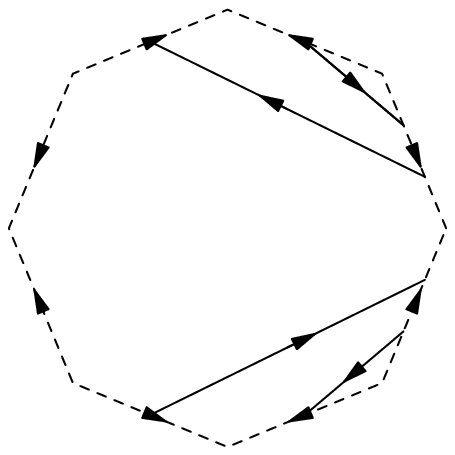,width=4truecm}}
\put(60,0){\epsfig{file= 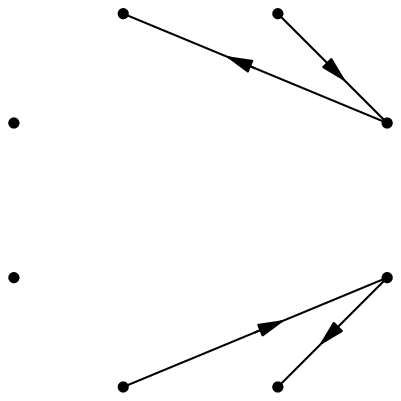,width=4truecm}}
\end{picture}
\caption{Jordan curves for $z_1t_1^{\overline \phi_1}$ and a Whitehead graph for  
$t_1^{\overline \sigma_1} = t_1t_2\overline t_1$.}\label{fig:circletwist}
\end{center}
\vspace{-.7cm}
\end{figure}

 It follows that the action of any given element of $\B_n$ on $\Sigma_{0,1,n}$ is induced by some
 self\d1homeo\-mor\-phism $\phi \in \gen{\phi_{[1{\uparrow}n-1]}}$.  The interesting feature now is that 
$\phi$ carries the oriented Jordan curve 
 $f_{-1}d_{[0{\uparrow}1]}\overline f_1 \overline e_{[1{\downarrow}0]} \,\,\,(\sim t_0t_1)$ to 
an oriented
Jordan curve $f_{-1}d_{[0{\uparrow}1]}\overline f_1^{\,\phi} \overline e_{[1{\downarrow}0]} 
\,\,\, (\sim (t_0t_1)^\phi \sim t_0t_1^\phi)$.  Recall that $\widehat \complexes$ is obtained 
by edge identification from  the $(2n+2)$-gon with clockwise, boundary label
 $\mathop{\Pi}\limits_{i\in[0{\uparrow}n]} (t_i \overline t_i)$. 
The Jordan curve
$f_{-1}d_{[0{\uparrow}1]}\overline f_1^{\,\phi} \overline e_{[1{\downarrow}0]}$
has  as its preimage,  in the $(2n+2)$-gon, the union of a family of (disjoint) 
oriented arcs.  These arcs can be used to reconstruct~$t_1^\phi$, since the 
Jordan curve cyclically reads off $t_0t_1^\phi$ from its meetings with 
the labelled oriented tethers; notice that the set of tethers is now 
dual to the set of generators~$t_{[0{\uparrow}n]}$.  
The purpose of this appendix is to define and study a combinatorial representation of 
the family of arcs, and recover Larue's  
characterization of the elements of~$t_1^{\B_n}$. 

Although it will not be used in our arguments,  let us mention the fact that, 
on collapsing the interior of each labelled edge of the $(2n+2)$-gon to a labelled vertex, 
each oriented arc in the family becomes an oriented edge, and we recover the
(directed, multi-edge, non-cyclic) Whitehead graph of  $t_1^\phi$; see Fig.~\ref{fig:circletwist}.

\section{Nested sets}\label{sec:wh0}

We now introduce some formal definitions that will allow us to associate
 a combinatorial Jordan curve to each element of $t_1^{\B_n}$.

\begin{definitions} Let $(A,\le)$ be a finite ordered set, and let $m \in \naturals$.

Let $N$ denote the number of elements of $A$.  Then $A$ is order-isomorphic to $[1{\uparrow}N]$
in a unique way, and we assign to $A$ the induced metric, denoted $d_A$.  Thus $d_A(a_1,a_2) = 1$
if and only if $a_1 \ne a_2$ and no element of $A$ lies strictly between $a_1$ and $a_2$.

Recall that the elements of $A^{m}$ are called  {\it $m$-tuples for } $A$.

Let $a_1,a_2,b_1,b_2$ be   elements of $A$.  We say that 
$\{a_1,b_1\}$ is {\it nested with} $\{a_2,b_2\}$ (for $(A,\le)$) if
$a_1,a_2,b_1,b_2$ are  distinct elements of $A$, and
either both of, or neither of, $a_2$ and $b_2$ lie between $a_1$ and $b_1$ in $(A,\le)$.
 It is not difficult to see that, in this event,
$\{a_2,b_2\}$ is   nested with  $\{a_1,b_1\}$.

Let $a_{([1{\uparrow}m])}$ and $b_{([1{\uparrow}m])}$ be $m$-tuples of $A$.

We say that $a_{([1{\uparrow}m])}$ is an
 $m$-tuple {\it without repetitions} if $a_i \ne a_j$ 
for all $i \ne j$ in $[1{\uparrow}m]$.

We say that $(a_{[1{\uparrow}m]})$  is an {\it  ascending} $m$-tuple (for $(A,\le)$) 
if $a_1 \le a_2 \le \cdots \le a_{m}$ in $(A,\le)$.

We say that
 $\{\{a_i,b_i\}\}_{i\in[1{\uparrow}m]}$ is {\it nested} (for $(A,\le)$) if, for all 
$i \ne j$ in $[1{\uparrow}m]$, $\{a_i,b_i\}$ is nested with $\{a_j,b_j\}$ for $(A,\le)$. 

We let $\Sym_m$ act on $A^m$, on the left, by 
$\null^\pi(a_{([1{\uparrow}m])}): = a_{([1{\uparrow}m])^{\overline\pi}}$.
For example, $\null^{(1,2,3)}(a_1,a_2,a_3)= (a_3, a_1, a_2)$, and, hence,
$\null^{(1,2,3)}(a,b,c)= (c,a,b)$.
The {\it ascending rearrangement} of $a_{([1{\uparrow}m])}$ is the unique
ascending $m$-tuple for $(A,\le)$ that lies in the $\Sym_m$-orbit 
of $a_{([1{\uparrow}m])}$.

Let $a_{([1{\uparrow}2m])}$ be a $2m$-tuple for $A$.  

A permutation $\pi \in \Sym_{2m}$ is said to  {\it embed 
$a_{([1{\uparrow}2m])}$ in a plane} if   $\null^\pi a_{([1{\uparrow}2m)])}$ is ascending for $(A,\le)$, and
 both $\{[2i-1{\uparrow}2i]^{\pi}\}_{i\in[1{\uparrow}m]}$ and 
 $\{[2i{\uparrow}2i+1]^{\pi}\}_{i\in[1{\uparrow}m-1]}$ are nested in
$(\naturals, \le)$.

 We say that
  $a_{([1{\uparrow}2m])}$   is  a  {\it planar} $2m$-tuple (for $(A,\le)$) 
if  there exists some $\pi \in \Sym_{2m}$ which embeds $a_{([1{\uparrow}2m])}$ in a plane.
(If no two consecutive terms of $a_{([1{\uparrow}2m])}$ are equal, 
$\pi$ is then unique, but we shall not need this fact.)
There is then an associated  diagram formed as follows. We assign, to each point
$i \in[1{\uparrow}2m] \subset \reals \subset \complexes,$ 
the label $a_{i^{\overline \pi}}$; notice that this means that  the label of $i^{\pi}$ is~$a_i$.
For each $i \in [1{\uparrow}m]$,  we join $(2i-1)^{\pi}$ (labelled $a_{2i-1}$)
to $(2i)^{ \pi}$ (labelled $a_{2i}$)
by an oriented semi-circle in the upper half-plane, and for each $i \in [1{\uparrow}m-1]$, 
 we join $(2i)^{\pi}$  (labelled $a_{2i}$) to $(2i+1)^{\pi}$
(labelled $a_{2i+1}$)
by an oriented semi-circle in the lower half-plane.
These oriented semi-circles form an oriented arc with no crossings which traces
out the $2m$-tuple $a_{([1{\uparrow}2m])}$.
\hfill\qed
\end{definitions}

\begin{example}\label{ex:diag} Suppose that 
$a_{([1{\uparrow}8])} = (\overline z_1, t_1, \overline t_1, t_2, \overline  t_2, \overline t_1, t_1,z_1)$
 is an
8-tuple for some ordered set $(A,\le)$, and that 
 the ascending rearrangement of $a_{([1{\uparrow}8])}$  
 is $( \overline z_1, t_1, t_1, \overline t_1, \overline t_1, t_2, \overline t_2, z_1).$

The permutation $\left(\begin{matrix}
1&2&3&4&5&6&7&8\\
1&2&5&6&7&4&3&8
\end{matrix}\right) = (3,5,7)(4,6)$
embeds $a_{([1{\uparrow}8])}$ in a plane since both
 $\{\{1,2\},\{5,6\},\{7,4\},\{3,8\}\}$  and   $\{\{2,5\},\{6,7\},\{4,3\}\}$ are nested in $(\naturals,\le)$, 
and $\null^{(3,5,7)(4,6)}(\overline z_1, t_1, \overline t_1, t_2, \overline  t_2, \overline t_1, t_1,z_1) 
\hskip -2pt = \hskip-2pt
( \overline z_1, t_1, t_1, \overline t_1, \overline t_1, t_2, \overline t_2, z_1).$  The associated diagram can
be seen in Fig.~\ref{fig:xy1}.
\begin{figure}[H]
\vspace{-.3cm}
\begin{center}
\setlength{\unitlength}{.1mm}
\begin{picture}(706,400) 
\put(0,0){\epsfig{file= 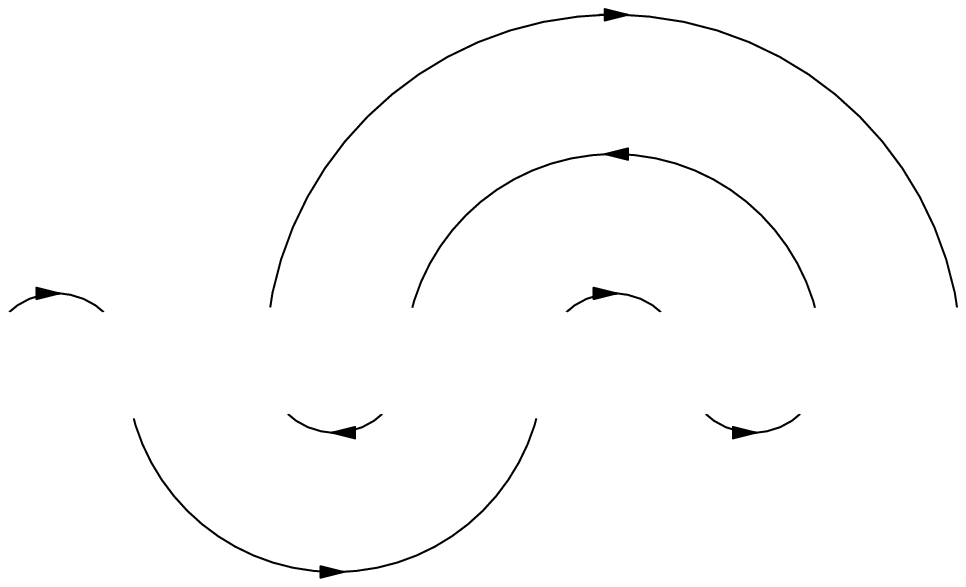,width=70.6truemm}}
\put(55,136){\makebox(0,0)[c]{$\underset{1}{\overline z_1}$}} 
\put(140,136){\makebox(0,0)[c]{$\underset{2}{t_1}$}} 
\put(228,136){\makebox(0,0)[c]{$\underset{3}{t_1}$}} 
\put(313,136){\makebox(0,0)[c]{$\underset{4}{\overline t_1}$}} 
\put(401,136){\makebox(0,0)[c]{$\underset{5}{\overline t_1}$}} 
\put(483,136){\makebox(0,0)[c]{$\underset{6}{t_2}$}} 
\put(571,136){\makebox(0,0)[c]{$\underset{7}{\overline t_2}$}} 
\put(656,136){\makebox(0,0)[c]{$\underset{8}{z_1}$}} 
\end{picture}
\caption{$(\overline z_1, t_1, \overline t_1, t_2, \overline  t_2, \overline t_1, t_1,z_1)$.}\label{fig:xy1}
\end{center}
\vspace{-.7cm}
\end{figure}
\vskip -.7cm \hfill\qed \vskip 0.4cm
\end{example}

Let us record two results which will be useful later.

\begin{lemma}\label{lem:nest3} Let $(A,\le)$ be an ordered set, and 
let $m$ be a positive integer.  Let $c_{[1{\uparrow}m]}$ and
$\overline c_{[1{\uparrow}m]}$ be $m$-tuples without repetitions
for   $(A,\le)$ such that $\{\{c_i,\overline c_i\}\}_{i\in [1{\uparrow}m]}$ is nested,
and $\max(c_{[1{\uparrow}m]}) < \min(\overline c_{[1{\uparrow}m]})$.
If $c_{([1{\uparrow}m])}$ is ascending, then $\overline c_{([m{\downarrow}1])}$ is also ascending.
\end{lemma}

\begin{proof} We argue by induction on $m$.  If $m = 1$, the conclusion is trivial.
Now, assume that $m \ge 2$ and that the implication
holds with $m-1$ in place of $m$.  We see that
$c_1 < c_2 \le \max(c_{[1{\uparrow}m]}) < \min(\overline c_{[1{\uparrow}m]}) \le \overline c_1.$
Since $\{c_1,\overline c_1\}$ is nested with $\{c_2,\overline c_2\}$, we also see that
$c_1 < \overline c_2 < \overline c_1$.
By the induction hypothesis, $\overline c_{([m{\downarrow}2])}$ is ascending, and hence 
$\overline c_{([m{\downarrow}1])}$ is ascending.  Hence,  the result is proved.
\end{proof}

\begin{lemma}\label{lem:B}
Let $(A,\le)$ be an ordered set, let $m \in \naturals$, and 
let $a_{([1{\uparrow}2m])}$ be a $2m$-tuple for $A$.

Then $a_{([1{\uparrow}2m])}$ is planar for $(A,\le)$ if and only if there exists
an ordered set $(B,\le)$, and a $2m$-tuple $b_{([1{\uparrow}2m])}$ for $B$, without repetitions,
and an ordered-set map  $B \to A$, $b \mapsto \lab(b)$,
such that $b_{[1{\uparrow}2m]} = B$,  $\lab(b_{([1{\uparrow}2m])}) = a_{([1{\uparrow}2m])}$,  and 
$\{b_{[2i{\uparrow}2i+1]}\}_{i\in[1{\uparrow}m-1]}$ and $\{b_{[2i-1{\uparrow}2i]}\}_{i\in[1{\uparrow}m]}$
are nested for $(B,\le)$.
\end{lemma}

\begin{proof} Suppose first that $a_{([1{\uparrow}2m])}$ is planar for $(A,\le)$, and 
let $\pi$ be an element of $\Sym_{2m}$ that embeds $a_{([1{\uparrow}2m])}$ in a plane. 
We take $B$ to be $[1{\uparrow}2m]$ with the usual ordering.
 For each  $i \in [1{\uparrow}2m]$, let $\lab(i) = a_{i^{\overline \pi}}$ and let $b_i = i^\pi$; thus, 
$\lab(b_i) = \lab(i^\pi) = a_i$.  All the conditions are satisfied.

Conversely, if $B$ exists, we can identify $B$ with $[1{\uparrow}2m]$ with the usual ordering, in a unique way.
Then the map $i \mapsto b_i$ is an element  $\pi$ of $\Sym_{2m}$ that embeds $a_{([1{\uparrow}2m])}$ in a plane. 
\end{proof}

\section{Planar words in $\Sigma_{0,1,n}$}\label{sec:wh}

\begin{definitions}\label{defs:rect} Let  $A$ be the monoid generating set  
$\{z_1, \overline z_1\} \cup t_{[1{\uparrow}n]} \cup \overline t_{[1{\uparrow}n]}$ of $\Sigma_{0,1,n}$.
We form the ordered set $(A,\le)$ with
$$\overline z_1 < t_1 < \overline t_1 < \cdots < t_n < \overline t_n < z_1.$$
We remark that, for $n \ne 1$, the ordering on $A$ is
reminiscent of the ordering of the ends of  $\Sigma_{0,1,n}$ in Section~\ref{sec:ends}.
We emphasize that, even if $n = 1$, $z_1 \ne \overline t_1$ in~$A$.

Let $m \in \naturals$.  
Consider an $m$-tuple $a_{([1{\uparrow}m])}$ for $t_{[1{\uparrow}n]} \cup \overline t_{[1{\uparrow}n]}$, and let
$w = \Pi a_{[1{\uparrow}m]}  \in \Sigma_{0,1,n}$; thus
  $a_{([1{\uparrow}m])}$ is an  expression
for $w$.
We define the {\it Whitehead expansion} of $a_{([1{\uparrow}m])}$ to be the  $(2m+2)$-tuple 
$$(\overline z_1, a_1, \overline a_1, a_2, \overline a_2, \ldots, a_m, \overline a_m, z_1)$$
for $A$, and we shall express it in the format
 $(\overline z_1,((a_i,\overline a_i))_{i\in[1{\uparrow}m]},z_1)$.
We say that $a_{([1{\uparrow}m])}$ is a {\it planar} expression
for $w$ if the Whitehead expansion of $a_{([1{\uparrow}m])}$
 is planar for $(A,\le)$.  
If the unique reduced expression for $w$ is a planar expression for $w$, then
we say that $w$ is a {\it planar} word in $\Sigma_{0,1,n}$.
\hfill\qed
\end{definitions}

\begin{examples} (i). The word $t_1\overline t_2\overline t_1$ is planar, 
since the Whitehead expansion of the reduced expression is
$( \overline z_1, t_1, \overline t_1, t_2, \overline  t_2, \overline t_1, t_1, z_1),$
and, by Example~\ref{ex:diag}, $( \overline z_1, t_1, \overline t_1, t_2, \overline  t_2, \overline t_1, t_1, z_1)$ is 
planar for $(A,\le)$;  in a sense, Fig.~\ref{fig:xy1}  reflects  Fig.~\ref{fig:circletwist}.
We call Fig.~\ref{fig:xy1} the Larue-Whitehead diagram of $t_1\overline t_2\overline t_1$.

(ii).  The word $t_1 \overline t_2$ is not planar; there is only one permutation to consider.

(iii).  The word $t_1^2$ is  not planar;  there are four permutations to consider.

(iv).   The word $t_3^{ t_1 \overline t_2 \overline t_1}$ is planar, while 
the word $t_3^{t_1t_2\overline t_1}$ is not planar, and these two words have the same Whitehead graph.
\hfill\qed
\end{examples}

\begin{proposition} Let $w \in \Sigma_{0,1,n}$. If there exists some planar expression for~$w$,
 then $($the reduced expression for$)$ $w$ is planar.
\end{proposition}

\begin{proof}  Suppose that $ a_{([1{\uparrow} m])}$ is a planar expression for $w$,
as in Definitions~\ref{defs:rect}.

By Lemma~\ref{lem:B}, there exists an ordered set $(B,\le)$, 
and a planar $(2m+2)$-tuple $b_{([1{\uparrow}2m+2])}$ for $(B,\le)$, 
without repetitions, 
and a labelling $B \to A$, $b \mapsto \lab(b)$,
such that the labelling respects the orderings and 
$\lab(b_{([1{\uparrow}2m+2])})$ is the Whitehead expansion of $ a_{([1{\uparrow} m])}$.
Moreover, $B = b_{[1{\uparrow}2m+2]}$.

Suppose that the given planar expression $ a_{([1{\uparrow} m])}$   is not reduced.  
We shall find a shorter planar expression for $w$.

There exists some $j \in [1{\uparrow}m-1]$ such that
$a_{j+1} = \overline a_j$ in $t_{[1{\uparrow}n]} \cup \overline t_{[1{\uparrow}n]}$, and we may suppose
that we have chosen this $j$ in such a way that $d_B(b_{2j+1},b_{2j+2})$ has the minimum possible value.
Notice that $\lab(b_{([2j{\uparrow}2j+3])}) = (a_j, \overline a_j, \overline a_j, a_j)$.

Clearly, $w = \Pi a_{[1{\uparrow}j-1]}  \Pi a_{[j+1{\uparrow}m]}$, and 
$\lab(b_{([1{\uparrow}2j-1])},b_{([2j+4{\uparrow}2m+2])})$ is
$$(\overline z_1,((a_i,\overline a_i))_{i\in[1{\uparrow}j-1]}, ((a_i,\overline a_i))_{i\in[j+2{\uparrow}m]},z_1)$$
$$(\overline z_1, a_1, \overline a_1, \ldots, a_{j-1}, \overline a_{j-1}, a_{j+2}, 
\overline a_{j+2}, \ldots, a_m, \overline a_m, z_1).$$
It suffices to show that $(b_{([1{\uparrow}2j-1])},b_{([2j+4{\uparrow}2m+2])})$ is planar for $(B, \le)$.

\medskip

\noindent \textbf{Claim.} $d_B(b_{2j},b_{2j+3}) = 1$.

\begin{proof} 
Consider any $k \in [1{\uparrow}2m-1]$ such that $b_{k}$ lies between $b_{2j}$ and $b_{2j+3}$.

Let  $\eta$ denote $(-1)^k$.

Since $\lab(b_{2j}) = \lab(b_{2j+3}) = a_j$, we see that 
$\lab(b_{k}) =a_j$.  Hence 
$\lab(b_{k+\eta}) =  \overline a_j = \lab(b_{2j+1}) = \lab(b_{2j+2})$.

Either $a_j < \overline a_j$ or $a_j > \overline a_j$ in $(A,\le)$.  Hence,
\begin{list}{}{}
\item either $\max\{b_{2j}, b_{k}, b_{2j+3}\} < \min\{b_{2j+1}, b_{k+\eta}, b_{2j+2}\}$ in $(B,\le)$, 
\item or $\min\{b_{2j}, b_{k}, b_{2j+3}\} > \max\{b_{2j+1}, b_{k+\eta}, b_{2j+2}\}$ in $(B,\le)$,
\end{list}
respectively. 

Since $\{ \{b_{2j}, b_{2j+1}\}, \{b_{2j+2}, b_{2j+3}\}, \{b_{k}, b_{k+\eta}\}\}$ is nested,
and $b_{k}$ lies between  $b_{2j}$ and $b_{2j+3}$, we see, from Lemma~\ref{lem:nest3}, that 
$b_{k+\eta}$ lies between $b_{2j+1}$ and $b_{2j+2}$.

Since  $\{b_{2j+1}, b_{2j+2}\}$  is nested with $\{b_{k+\eta}, b_{k+2\eta}\}$ and
$b_{k+\eta}$ lies between $b_{2j+1}$ and $b_{2j+2}$, we see that
$b_{k+2\eta}$ lies between $b_{2j+1}$ and $b_{2j+2}$.  
Hence, $$d_B(b_{k+2\eta}, b_{k+\eta}) \le d_B(b_{2j+1}, b_{2j+2}),$$
with equality holding only if
$\{b_{k+2\eta}, b_{k+\eta}\} = \{b_{2j+1}, b_{2j+2}\}$.
Also, $\lab(b_{k+ 2\eta}) = \overline a_j$, and, hence,
$\lab(b_{k+ 3\eta}) = a_j$.  Thus 
$$\lab(b_{k}, b_{k+\eta}, b_{k+2\eta}, b_{k+3\eta}) = ( a_j,
\overline a_j, \overline a_j, a_j ).$$  
By the minimality of $d_B(b_{2j+1}, b_{2j+2})$,  
we see that   $k = 2j$ or $k = 2j+3$.  This proves the claim.
\end{proof}

Now consider the passage from $b_{([1{\uparrow}2m+2])}$  to $b_{([1{\uparrow}2j-1])},b_{([2j+4{\uparrow}2m+2])}$.

On the odd-to-even steps, we pass from
$\{b_{[2i-1{\uparrow}2i]}\}_{i\in[1{\uparrow}m+1]}$ to 
$$\{b_{[2i-1{\uparrow}2i]}\}_{i\in[1{\uparrow}j-1]\cup[j+3{\uparrow}m+1]} \cup \{\{b_{2j-1},b_{2j+4}\}\}.$$
Thus, we remove $\{b_{2j-1}, b_{2j}\}$,  $\{b_{2j+1}, b_{2j+2}\}$,  $\{b_{2j+3}, b_{2j+4}\}$,
and we add only  $\{b_{2j-1}, b_{2j+4}\}$.  To see that, for all $k \in [1{\uparrow}j-1]\cup[j+3{\uparrow}m+1]$,
  $\{b_{2k-1}, b_{2k}\}$ is nested with
 $\{b_{2j-1}, b_{2j+4}\}$, we note the following:
\begin{align*}
(b_{2j-1} \text{ lies }&\text{between } b_{2k-1} \text { and } b_{2k}) 
\\&\Leftrightarrow
(b_{2j} \text{ lies between } b_{2k-1} \text { and } b_{2k})
\\ &\hskip 2cm  \text{since } 
\{b_{2j-1},b_{2j}\} \text{ is nested with } \{b_{2k-1}, b_{2k}\}
\\&\Leftrightarrow
(b_{2j+3} \text{ lies between } b_{2k-1} \text { and } b_{2k})
\\&\hskip 2cm \text{since } 
d_B(b_{2j},b_{2j+3}) = 1
\\&\Leftrightarrow
(b_{2j+4} \text{ lies between } b_{2k-1} \text { and } b_{2k})
\\&\hskip 2cm \text{since } 
\{b_{2j+3},b_{2j+4}\} \text{ is nested with } \{b_{2k-1}, b_{2k}\}.
\end{align*} 

On the even-to-odd steps,
we pass from
$\{b_{[2i{\uparrow}2i+1]}\}_{i\in[1{\uparrow}m]}$ to 
$$\{b_{[2i{\uparrow}2i+1]}\}_{i\in[1{\uparrow}j-1]\cup[j+2{\uparrow}m]}.$$
Thus, we remove  $\{b_{2j}, b_{2j+1}\}$ and $\{b_{2j+2}, b_{2j+3}\}$,
and we add  nothing.  Hence this remains nested.  This completes the proof.
\end{proof}

\begin{proposition}\label{prop:square4} Let $w$ be a planar word in $\Sigma_{0,1,n}$, and let $k \in [1{\uparrow}n]$.
\begin{enumerate}[\normalfont (i).]
\vskip-0.7cm \null
\item  $w$ is a squarefree word in $\Sigma_{0,1,n}$.
\vskip-0.7cm \null
\item  $w \not \in ( \Pi \overline t_{[n{\downarrow}k+1]} t_k {\star} ) -\{t_k^{\Pi t_{[k+1{\uparrow}n]}}\}$.
\vskip-0.7cm \null
\item  $w \not \in ( \Pi  t_{[1{\uparrow}k-1]} \overline t_{k} {\star} )$.
\end{enumerate}
\end{proposition}

\begin{proof}  For some $m \in \naturals$, there exists a reduced  expression $a_{([1{\uparrow}m])}$ 
for $w$.  

(i). Suppose that $w$ is  not squarefree, say $t_i, t_i$ occurs in $a_{([1{\uparrow}m])}$, 
then  $t_i, \overline t_i, t_i, \overline t_i$ occurs in  
$$(\overline z_1,((a_i,\overline a_i))_{i\in[1{\uparrow}m]},z_1).$$

Let $m_i$ be the number of occurrences of $t_i^{\pm 1}$ in $a_{([1{\uparrow}m])}$.

Suppose $c_{([1{\uparrow}m_i])}$ are labelled $t_i$ and $\overline c_{([m_i{\downarrow}1]}$ are such that the
even-to-odd pairing contains $\{\{c_k,\overline c_k\}\}_{k\in[1{\uparrow}m_i]}$.
The odd-to-even pairing contains $\{c_k, \overline c_j\}$ for some $k,j \in [1{\uparrow}m_i]$.
Let us choose $(k,j)$ so that $k+j$ is as large as possible.
Then $c_k < c_{k+1} < \overline c_j$.  Whatever $c_{k+1}$ is paired with 
in the odd-to-even pairing must lie in the interval $[c_k, \overline c_j]$ and 
cannot have label $t_i$ since the signs alternate, so $c_{j+1}$ 
is paired with $\overline c_k$ for some $k > j$.
This contradicts the maximality of $k+j$.  Hence $k = m_i$.  Similarly,
$j = m_i$.  Thus $\{c_{m_i}, \overline c_{m_i}\}$ lies in both the 
even-to-odd pairings and the odd-to-even pairings.
This gives a sinlge component, which is a contradiction.

(ii). Suppose that  $w   \in ( \Pi \overline t_{[n{\downarrow}k+1]} t_k {\star} )$.

Thus $(\overline z_1, ((a_i, \overline a_i))_{i\in[1{\uparrow}n-k+2}])$ is
$$(\overline z_1, \overline t_n, t_n, \overline t_{n-1}, t_{n-1}, \ldots, \overline t_{k+1}, t_{k+1},  
t_k, \overline t_k,
a_{n-k+2}, \overline a_{n-k+2})$$
Notice that $\{\overline t_k, a_{n-k+2}\}$ must be nested with $\{t_{k+1},t_k\}$, and, hence $a_{n-k+2}$ must lie in
$\{t_k, \overline  t_{k}, t_{k+1}\}$. By (i), $a_{n-k+2} \ne t_k$.  Since $a_{([1{\uparrow}m])}$ is a reduced expression,
$a_{n-k+2} \ne \overline t_k$.  Hence $a_{n-k+2} = t_{k+1}$. Let us denote this term $t_{k+1}'$ to distinguish it from
the preceding occurrence of $t_{k+1}$. $\{\overline t_k, t_{k+1}'\}$  is nested with $\{t_{k+1},t_k\}$.  Hence, 
Then $t_{k+1}' < t_{k+1}$.   By Lemma~\ref{lem:nest3},  $\overline t_{k+1}' > \overline t_{k+1}$.

Thus $(\overline z_1, ((a_i, \overline a_i))_{i\in[1{\uparrow}n-k+3}])$ is
$$(\overline z_1, \overline t_n, t_n, \overline t_{n-1}, t_{n-1}, \ldots, \overline t_{k+1}, t_{k+1},  
t_k, \overline t_k, t_{k+1}', \overline t_{k+1}',
a_{n-k+3}, \overline a_{n-k+3})$$
Notice that $\{\overline t_{k+1}', a_{n-k+3}\}$ must be nested with $\{t_{k+2}, \overline t_{k+1} \}$, and,
 hence, $a_{n-k+3}$ must lie in
$\{\overline  t_{ k+1 }, t_{k+2}\}$.   Since $a_{([1{\uparrow}m])}$ is a reduced rexpression,
$a_{n-k+3} \ne \overline t_{k+1}$.  Hence $a_{n-k+3} = t_{k+2}$, and we denote this by $t_{k+2}'$.  
Then  $t_{k+2}' < t_{k+2}$, and,  by Lemma~\ref{lem:nest3},  $\overline t_{k+2}' > \overline t_{k+2}$.

By repeating the argument in the last paragraph, we eventually find that $w = t_k^{\Pi t_{[k+1{\uparrow}n]}}$.

(iii).  Suppose that $w \in ( \Pi  t_{[1{\uparrow}k-1]} \overline t_{k} {\star} )$. 

Then
$(\overline z_1, ((a_i, \overline a_i))_{i\in[1{\uparrow}2k]})=
(\overline z_1, t_1, \overline t_1, t_2, \overline t_{2},   \ldots, t_{k-1}, \overline t_{k-1}, 
\overline t_{k},  t_k),$
and by an argument similar to that in (ii), we find that this is impossible.
\end{proof}

\section{$\B_n$ permutes the planar words in $\Sigma_{0,1,n}$}\label{sec:wh2}

\begin{proposition} Let $w \in \Sigma_{0,1,n}$ and let $i \in [1{\uparrow}n-1]$.  
If $w$ is a planar word in $\Sigma_{0,1,n}$, then $w^{\sigma_i}$
is a planar word in $\Sigma_{0,1,n}$. 
\end{proposition}

\begin{proof} Suppose that $r_{([1{\uparrow} m])}$ is any planar expression for $w$, 
as in Definitions~\ref{defs:rect}.

In applying $\sigma_i$ to $(\overline z_1,((r_i,\overline r_i))_{i\in[1{\uparrow}m]},z_1)$, 
we 
\begin{list}{}{}
\vskip -.8cm\null
\item replace each $t_i,\overline t_i $ with $t_{i+1},\overline t_{i+1}$,
\vskip -.8cm\null
\item replace each $\overline t_i,t_i$ with $\overline t_{i+1},t_{i+1}$, 
\vskip -.8cm\null
\item replace each $t_{i+1},\overline t_{i+1}$ with 
$\overline t_{i+1}, t_{i+1}, t_i, \overline t_i, t_{i+1}, \overline t_{i+1}$, 
\vskip -.8cm\null
\item replace each  $\overline t_{i+1}, t_{i+1}$ with 
$\overline t_{i+1},t_{i+1}, \overline t_i, t_i, t_{i+1}, \overline t_{i+1}$.
\vskip -.8cm\null
\end{list} 
We will not perform any cancellations in the resulting sequence. 

Let $\pi \in \Sym_{2m+2}$ be a permutation which
embeds $(\overline z_1,((r_i,\overline r_i))_{i\in[1{\uparrow}m]},z_1)$ in a plane.
By Lemma~\ref{lem:B}, there exists an ordered set $(B,\le)$, and
a $(2m+2)$-tuple $p_{([1{\uparrow}2m+2])}$ without repetitions, 
for $(B,\le)$, such that $\pi$ embeds $p_{([1{\uparrow}2m+2])}$ in a plane.
Moreover, there exists a labelling $B \to A$, $b \mapsto \lab(b)$,
such that the labelling respects the orderings and
$$\lab(p_{([1{\uparrow}2m+2])}) =(\overline z_1,((r_i,\overline r_i))_{i\in[1{\uparrow}m]},z_1).$$
Moreover, $B = p_{[1{\uparrow}2m+2]}$.

Let $m_i$ denote the number of elements of $B$ with label $t_i$, 
and let $m_{i+1}$ denote the number of elements of $B$ with label $t_{i+1}$.
To begin, we have to add $4m_{i+1}$ elements to $B$, and we have to specify the ordering 
on the expanded set.

Let $c_{[1{\uparrow}m_i]}$ denote the set, in ascending order, of those 
elements of $B$ which have the label $t_i$.
Let $\overline c_{[m_i{\downarrow}1]}$ denote the set, in ascending order, 
of those elements of $B$ which have the label $\overline t_i$.
Let  $d_{[1{\uparrow}m_{i+1}]}$ denote the set, in ascending order, 
of those elements of $B$ which have the label   $t_{i+1}$.
Let $\overline d_{[m_{i+1}{\downarrow}1]}$ denote the set, in ascending order, 
of those elements of $B$ which have the label $\overline t_{i+1}$.
Thus we have $$c_{1} <  \ldots < c_{m_{i}} < \overline c_{m_{i}} < \ldots  < \overline c_{1} < d_{1} < \ldots < d_{m_{i+1}} < 
\overline d_{m_{i+1}} < \ldots <  \overline d_{1}$$
and no other element of $B$ lies in the interval $[c_1{\uparrow} \overline d_1]$. 
We write $$[c_1{\uparrow} \overline d_1] =
 (c_{([1{\uparrow}m_{i}])}, \overline c_{([m_{i}{\downarrow}1])}, 
d_{([1{\uparrow}m_{i+1}])}, \overline d_{([m_{i+1}{\downarrow}1])})$$
to express this.

\begin{figure}[H]
\vspace{0cm}
\begin{center}
\setlength{\unitlength}{.1mm}
\begin{picture}(711,551) 
\put(0,0){\epsfig{file= 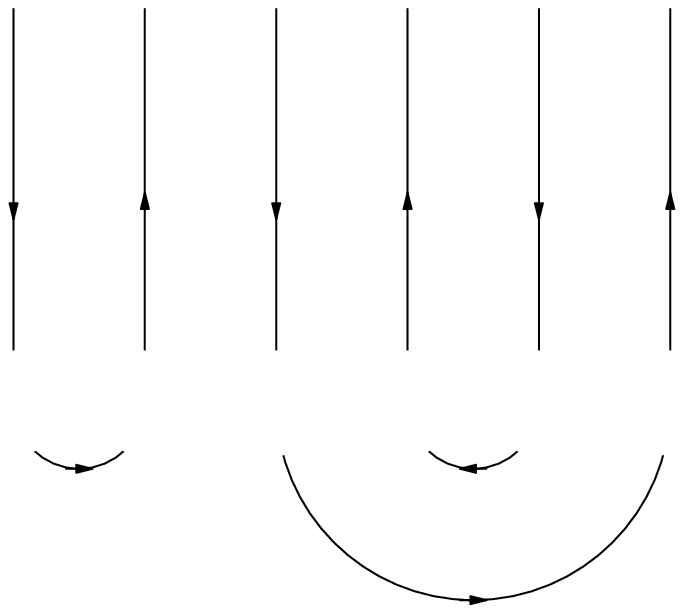,width=71.1truemm}}
\put(68,190){\makebox(0,0)[c]{$\textstyle \underset{t_{i}}{c_1}$}}
\put(187,190){\makebox(0,0)[c]{$\textstyle \underset{\overline t_{i}}{\overline c_1}$}}
\put(302,190){\makebox(0,0)[c]{$\textstyle \underset{t_{i+1}}{d_1}$}}
\put(415,190){\makebox(0,0)[c]{$\textstyle\underset{t_{i+1}}{d_2}$}}
\put(537,188){\makebox(0,0)[c]{$\textstyle\underset{\overline t_{i+1}}{\overline d_2}$}}
\put(653,188){\makebox(0,0)[c]{$\textstyle \underset{\overline t_{i+1}}{\overline d_1}$}}
\end{picture}
\caption{$(c_{([1{\uparrow}1])}, \overline c_{([1{\downarrow}1])}, 
d_{([1{\uparrow}2])}, \overline d_{([2{\downarrow}1])})$.}\label{fig:xy7}
\end{center}
\vspace{-.7cm}
\end{figure}

With the preceding notation, we create an interval of $4m_{i+1}$ new elements denoted 
$$[a_1{\uparrow}\overline b_1] = 
(a_{([1{\uparrow}m_{i+1}])}, \overline a_{([m_{i+1}{\downarrow} 1])}, b_{([1{\uparrow}m_{i+1}])},
 \overline b_{([m_{i+1}{\downarrow}1])}).$$ 
We expand $B$ by inserting this interval just before $c_1$, that is, just before the interval
$[c_1{\uparrow}\overline d_1]$.  We now have a new ordered set $B'$ with $2m+2 + 4m_{i+1}$ elements.

We have to specify the new labelling $B' \to A$.  On $c_{[1{\uparrow}m_{i}]}$, we change the labels from
$t_{i}$ to $t_{i+1}$.  On $\overline c_{[m_{i}{\downarrow}1]}$, we change the labels from
$\overline t_i$ to $\overline t_{i+1}$.  On $d_{[1{\uparrow}m_{i+1}]}$, we change the
labels from $t_{i+1}$ to $\overline t_{i+1}$.  On $\overline d_{[m_{i+1}{\downarrow}1]}$, we keep the same 
labels, $\overline t_{i+1}$.  On the rest of $B - [c_1{\uparrow}\overline d_1]$, we keep the same labels.
We give all the elements of $a_{[1{\uparrow}m_{i+1}]}$ the label $t_i$; we give all the elements of
 $\overline a_{[m_{i+1}{\downarrow}1]}$ the label $\overline t_i$; we give all the
elements of $b_{[1{\uparrow}m_{i+1}]}$ and 
$\overline b_{[m_{i+1}{\downarrow}1]}$ the label $t_{i+1}$. 
 The labelling clearly respects the orderings of $B'$ and~$A$.

For the even-to-odd steps, it follows from Lemma~\ref{lem:nest3} that
$$\{p_{[2k{\uparrow}2k+1]}\}_{k\in[1{\uparrow}m]} \,\,\,
\supseteq \,\,\, \{\{c_i,\overline c_i\}\}_{i\in[1{\uparrow}r]} \cup \{ \{d_j,\overline d_j\}\}_{j\in[1{\uparrow}s]}.$$

Let $q_{([1{\uparrow}2m+4m_{i+1}])}$ be the $2m+4s$-tuple obtained from $p_{([1{\uparrow}2m+2])}$ as follows.
For each $j \in[1{\uparrow}m_{i+1}]$, there exists a unique $i \in [1{\uparrow}m]$ such that
$p_{[2i-1{\uparrow}2i]}=\{d_j,\overline d_j \}$.  If $p_{([2i-1{\uparrow}2i])} = (d_j,\overline d_j)$,
then it is to be expanded to $(d_j, \overline b_j, a_j, \overline a_j, b_j, \overline d_j)$.
If $p_{([2i-1{\uparrow}2i])}=(\overline d_j, d_j)$,
then it is to be expanded to $(\overline d_j,  b_j, \overline a_j,  a_j, \overline b_j,  d_j)$.
This completes the definition of $q_{([1{\uparrow}2m+4m_{i+1}])}$.
\begin{figure}[H]
\vspace{0cm}
\begin{center}
\setlength{\unitlength}{.2mm}
\begin{picture}(711,412) 
\put(0,0){\epsfig{file= 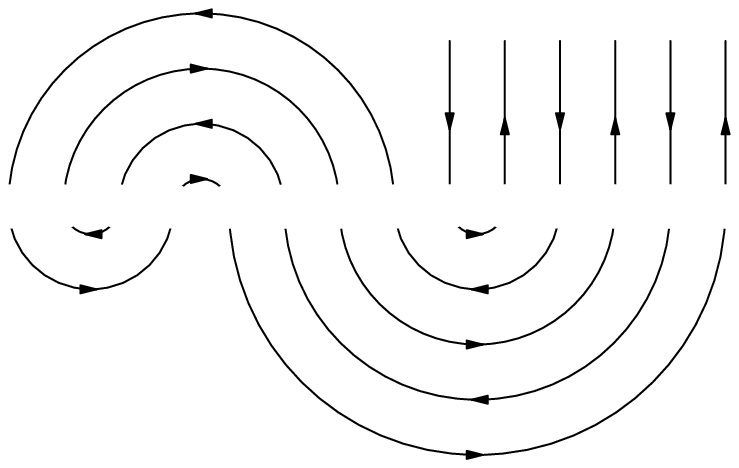,width=142.2truemm}}
\put(39,230){\makebox(0,0)[c]{$\scriptstyle \underset{t_i}{a_1}$}}
\put(89,230){\makebox(0,0)[c]{$\scriptstyle \underset{t_i}{a_2}$}}
\put(138,230){\makebox(0,0)[c]{$\scriptstyle \underset{\overline t_i}{\overline a_2}$}}
\put(187,230){\makebox(0,0)[c]{$\scriptstyle \underset{\overline t_i}{\overline a_1}$}}
\put(233,230){\makebox(0,0)[c]{$\scriptstyle \underset{t_{i+1}}{b_1}$}}
\put(283,230){\makebox(0,0)[c]{$\scriptstyle \underset{t_{i+1}}{b_2}$}}
\put(332,230){\makebox(0,0)[c]{$\scriptstyle \underset{t_{i+1}}{\overline b_2}$}}
\put(381,230){\makebox(0,0)[c]{$\scriptstyle \underset{t_{i+1}}{\overline b_1}$}}
\put(431,230){\makebox(0,0)[c]{$\scriptstyle \underset{t_{i+1}}{c_1}$}}
\put(479,230){\makebox(0,0)[c]{$\scriptstyle \underset{\overline t_{i+1}}{\overline c_1}$}}
\put(529,230){\makebox(0,0)[c]{$\scriptstyle \underset{\overline t_{i+1}}{d_1}$}}
\put(578,230){\makebox(0,0)[c]{$\scriptstyle \underset{\overline t_{i+1}}{d_2}$}}
\put(627,230){\makebox(0,0)[c]{$\scriptstyle \underset{\overline t_{i+1}}{\overline d_2}$}}
\put(676,230){\makebox(0,0)[c]{$\scriptstyle \underset{\overline t_{i+1}}{\overline d_1}$}}
\end{picture}
\caption{$(a_{([1{\uparrow}2])}, \overline a_{([2{\downarrow} 1])}, b_{([1{\uparrow}2])},
 \overline b_{([2{\downarrow}1])}, c_{([1{\uparrow}1])}, \overline c_{([1{\downarrow}1])}, 
d_{([1{\uparrow}2])}, \overline d_{([2{\downarrow}1])})$.}\label{fig:xy6}
\end{center}
\vspace{-.7cm}
\end{figure}

In passing from $\{p_{[2k{\uparrow}2k+1]}\}_{k\in[1{\uparrow}m-1]}$ 
to $\{q_{[2k{\uparrow}2k+1]}\}_{k\in[1{\uparrow}m+2m_{i+1}-1]}$,
we add $\{\{ \overline b_j, a_j\}\}_{j\in[1{\uparrow}s]} \cup \{ \{ \overline a_j, b_j\}\}_{j\in[1{\uparrow}s]}$.
In $B'$, for each $j \in [1{\uparrow}s]$,
\begin{align*}
[\overline a_j{\uparrow}b_j] &= (\overline a_{([j{\downarrow}1])}, b_{([1{\uparrow}j])}) 
\\&\text{ and the underlying set is }  \cup\{\overline a_k, b_k\}_{k\in[1{\uparrow}j]},\\
[a_j{\uparrow}\overline b_j] &= (a_{[j{\uparrow}s]}, 
\overline a_{[s{\downarrow}1]}, b_{[1{\uparrow}s]}, \overline b_{[s{\downarrow}j]})
\\&\text{ and the underlying set is }  \cup\{\overline a_k, b_k\}_{k\in[1{\uparrow}s]} \cup \cup\{\overline b_k, a_k\}_{k\in[j{\uparrow}s]}.
\end{align*}
Both of these intervals are closed under the pairing-off of 
$$\{q_{[2k{\uparrow}2k+1]}\}_{k\in[1{\uparrow}m+2m_{i+1}-1]}.$$
Thus, 
$\{q_{[2k{\uparrow}2k+1]}\}_{k\in[1{\uparrow}m+2m_{i+1}-1]}$ is also nested.

In passing from $\{p_{[2k-1{\uparrow}2k]}\}_{k\in[1{\uparrow}m]}$ to 
$\{q_{[2k-1{\uparrow}2k]}\}_{k\in[1{\uparrow}m+m_{i+1}]}$,
we delete\newline
$\{\{d_j,\overline d_j\}\}_{j\in[1{\uparrow}s]}$,
and add
$\{\{d_j,\overline b_j\}\}_{j\in[1{\uparrow}s]} \cup \{\{a_j,\overline a_j\}\}_{j\in[1{\uparrow}s]}
\cup \{\{b_j,\overline d_j\}\}_{j\in[1{\uparrow}s]}$.
In $B'$,  for each $j \in [1{\uparrow}s]$,
\begin{align*}
[a_j, \overline a_j] &= (\overline a_{([j{\downarrow}1])}, a_{([1{\uparrow}j])}) \\&\text{ and the underlying set is } 
\mathop{\cup}\limits_{k\in[1{\uparrow}j]} \{a_k,\overline a_k\},\\
[\overline b_j, d_j] &= (\overline b_{([j{\downarrow}1])}, c_{([1{\uparrow}r])}, 
\overline c_{([r{\downarrow}1])}, d_{([1{\uparrow}j])})\\
 &\text{ and the underlying set is }  \mathop{\cup}\limits_{k\in[1{\uparrow}j]} \{d_k,\overline b_k\}\cup 
\mathop{\cup}\limits_{i\in[1{\uparrow}r]}\{c_i,\overline c_i\},\\
[ b_j, \overline d_j] &= 
(b_{([j,s])}, \overline b_{([s{\downarrow}1])}, c_{([1{\uparrow}r])},
 \overline c_{([r{\downarrow}1])}, d_{([1{\uparrow}s])}, \overline d_{([s{\downarrow}j])})\\
&\text{ and the underlying set is }   \mathop{\cup}\limits_{k\in[1{\uparrow}s]}  \{d_k,\overline b_k\}
\cup  \mathop{\cup}\limits_{k\in[j{\uparrow}s]}\{ b_k, \overline d_k\}
 \cup \mathop{\cup}\limits_{i\in[1{\uparrow}r]}\{c_i,\overline c_i\}.
\end{align*}
Each of these intervals is closed under the pairing-off of $\{q_{[2k-1{\uparrow}2k]}\}_{k\in[1{\uparrow}m+2m_{i+1}]}$.
Thus, 
$\{q_{[2k-1{\uparrow}2k]}\}_{k\in[1{\uparrow}m+2m_{i+1}]}$ is nested.
\end{proof}

A similar argument shows that $\overline \sigma_i$ carries planar words to planar words. 

\begin{theorem}\label{th:permutes} The group $\B_n$ acts on the set of planar words in $\Sigma_{0,1,n}$, 
and, hence, if $n\ge1$, then every element of $t_1^{\B_n}$ is a planar word.
\end{theorem}

\begin{remark}\label{rem:square2} By combining Theorem~\ref{th:permutes} and 
Proposition~\ref{prop:square4}, we get another proof of Corollary~\ref{cor:square}.
\hfill\qed
\end{remark}

\section{The $\B_n$-orbits of the planar words in $\Sigma_{0,1,n}$}

In this section we rework~\cite[Lemma~2.3.12]{LarueThesis94} 
and in this case our argument seems to be longer than Larue's.  
The object is to show that the number of
$\B_n$-orbits in the set of all planar words in $\Sigma_{0,1,n}$ is $n+1$,
and that  $\{\Pi t_{[1{\uparrow}k]}\}_{k \in [0{\uparrow}n]}$ is a complete set of
representatives.

\begin{lemma}\label{lem:length2} Let $i$, $j$ be elements of $[1{\uparrow}n]$ such that $j \le i-1$, 
let $\phi = \Pi  \sigma_{[j{\uparrow}i-1]}$, and 
let $w$ be a planar word in~$\Sigma_{0,1,n}$. 
\begin{enumerate}[\normalfont(i)]
\vskip-0.7cm \null
\item  If  $w \in (\Pi t_{[1{\uparrow}i]} t_j {\star})$, then 
 $\abs{w^{\phi}} < \abs{w}$.
\vskip-0.7cm \null
\item If  $w \in (\Pi t_{[1{\uparrow}i]} \overline t_j{\star})$, then 
 $\abs{w^{\phi}} < \abs{w}$. 
\end{enumerate}
\end{lemma}

\begin{proof}
It is straightforward to show that $\phi$ acts as 

\centerline{\begin{tabular}
{
>{$}r<{$}  
@{\hskip0cm} 
>{$}l<{$} 
@{\hskip.3cm} 
>{$}c<{$} 
@{\hskip.7cm} 
>{$}c<{$}
@{\hskip1cm}  
>{$}r<{$}
@{\hskip 0cm} 
>{$}l<{$}
}
\setlength\extrarowheight{3pt}
&\hskip-.4cm\underline{\scriptstyle k \in [1,j-1]}&   & 
\underline{\scriptstyle k \in [j+1,i]}&   &\hskip-1cm\underline{\scriptstyle k \in [i+1,n]}
\\[.15cm]
(&t_k&  t_j& t_{k}& t_k&)^{\phi}\\  
=(&t_k & t_{i}& t_{k-1}^{t_{i}}&  t_k &).
\end{tabular}}

\bigskip
\bigskip

(i). Suppose that $w \in (\Pi t_{[1{\uparrow}i]} t_j {\star})$.

 \begin{figure}[H]
\vspace{-.3cm}
\begin{center}
\setlength{\unitlength}{.1mm}
\begin{picture}(1100,458) 
\put(0,0){\epsfig{file= 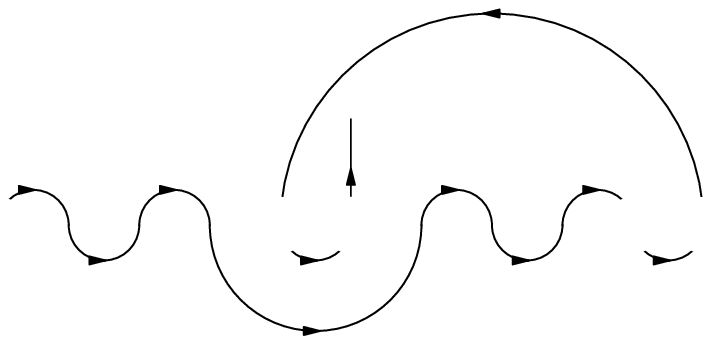,width=110truemm}}
\put(70,155){\makebox(0,0)[c]{$\overline z_1$}} 
\put(450,156){\makebox(0,0)[c]{$t_j$}} 
\put(553,162){\makebox(0,0)[c]{$\overline t_j$}} 
\put(945,156){\makebox(0,0)[c]{$t_i$}} 
\put(1040,162){\makebox(0,0)[c]{$\overline t_i$}} 
\end{picture}
\caption{$w \in (\Pi t_{[1{\uparrow}i]} t_j {\star})$, $j \le i-1$.}\label{fig:xy4}
\end{center}
\vspace{-.7cm}
\end{figure}

Since $t_it_j$ is a subword of $w$, 
every letter occurring in $w$ that belongs to $t_{[j{\uparrow}i]}\cup\overline t_{[j{\uparrow}i]}$
 belongs to a (reduced) subword of $w$ of the form $av\overline b$, where $a,b \in \{\overline t_i, t_j\}$
 and $v \in \gen{t_{[j{\uparrow}i]}}$.  Since, moreover, $w$ begins with $\Pi t_{[1{\uparrow}i]}$,
it can be shown that it is not possible to have $a= \overline t_i$ or $ b= \overline  t_i$.  Thus $a = b = t_j$.
Here, $\abs{(av\overline b)^\phi} = \abs{avb} -2$. 

We factor $w$ into syllables consisting of all such subwords together with the 
individual remaining letters, all of which lie in $t_{[1{\uparrow}j-1]}\cup t_{[i+1{\uparrow}n]}$,
and all of which are mapped to single letters by $\phi$.

Since $t_j$ occurs in $w$, we see that  $\abs{w^{\phi}} < \abs{w}$.

(ii). Suppose that $w \in (\Pi t_{[1{\uparrow}i]} \overline t_j{\star})$.

\begin{figure}[H]
\vspace{-.3cm}
\begin{center}
\setlength{\unitlength}{.1mm}
\begin{picture}(1100,407) 
\put(0,0){\epsfig{file= 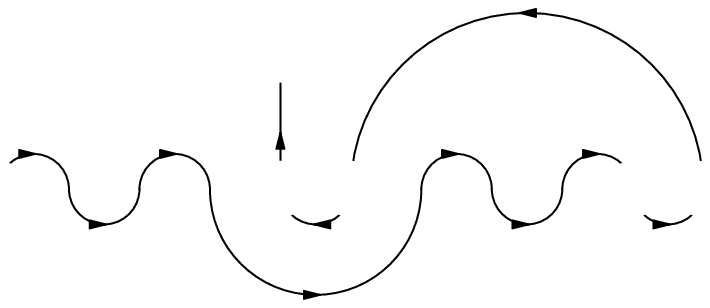,width=110truemm}}
\put(70,155){\makebox(0,0)[c]{$\overline z_1$}} 
\put(450,156){\makebox(0,0)[c]{$t_j$}} 
\put(553,162){\makebox(0,0)[c]{$\overline t_j$}} 
\put(945,156){\makebox(0,0)[c]{$t_i$}} 
\put(1040,162){\makebox(0,0)[c]{$\overline t_i$}} 
\end{picture}
\caption{$w \in (\Pi t_{[1{\uparrow}i]} \overline t_j{\star})$,  $j \le i-1$.}\label{fig:xy5}
\end{center}
\vspace{-.7cm}
\end{figure}

 Since $t_i \overline t_j$ is a subword of $w$, 
every letter occurring in $w$ that belongs to $t_{[j+1{\uparrow}i]}\cup\overline t_{[j+1{\uparrow}i]}$
 belongs to a (reduced) subword of $w$ of the form $av\overline b$, where $a,b \in \{t_j, \overline t_i\}$
 and $v \in \gen{t_{[j+1{\uparrow}i]}}$.  Since, moreover, $w$ begins with $\Pi t_{[1{\uparrow}i]}$,
it can be shown that it is not possible to have $a= \overline t_i$ or $b = \overline t_i$.  
Thus $a = b =  t_j$.
Here, $\abs{(av\overline b)^\phi} = \abs{avb} -2$. 

We factor $w$ into syllables consisting of all such subwords together with the 
individual remaining letters, all of which lie in $t_{[1{\uparrow}j]}\cup t_{[i+1{\uparrow}n]}$,
and all of which are mapped to single letters by $\phi$.

Since $t_i$ occurs in $w$, it is then clear that $\abs{w^\phi} \le \abs{w} -2$. 
\end{proof}

\begin{lemma}\label{lem:length1} Let $i$, $j$ be elements of $[1{\uparrow}n]$ such that $j \ge i+2$, 
let $\phi = \Pi \overline \sigma_{[j-1{\downarrow}i+1]}$, and 
let $w$ be a planar word in~$\Sigma_{0,1,n}$. 
\begin{enumerate}[\normalfont(i)]
\vskip-0.7cm \null
\item  If  $w \in (\Pi t_{[1{\uparrow}i]} t_j {\star})$, then 
 $\abs{w^{\phi}} \le \abs{w}$, and, moreover, if  $\abs{w^{\phi}} = \abs{w}$ then
$w^{\phi} \in (\Pi t_{[1{\uparrow}i+1]} {\star})$.
\vskip-0.7cm \null
\item If  $w \in (\Pi t_{[1{\uparrow}i]} \overline t_j{\star})$, then 
 $\abs{w^{\phi}} < \abs{w}$. 
\end{enumerate}
\end{lemma}

\begin{proof} It is straightforward to show that $\phi$ acts as 

\centerline{\begin{tabular}
{
>{$}r<{$}  
@{\hskip0cm} 
>{$}l<{$} 
@{\hskip.3cm} 
>{$}c<{$} 
@{\hskip.7cm} 
>{$}c<{$}
@{\hskip1cm}  
>{$}r<{$}
@{\hskip 0cm} 
>{$}l<{$}
}
\setlength\extrarowheight{3pt}
&\hskip-.4cm\underline{\scriptstyle k \in [1,i]}& \underline{\scriptstyle k \in [i+1,j-1]}  & 
&   &\hskip-1cm\underline{\scriptstyle k \in [j+1,n]}
\\[.15cm]
(&t_k&  t_k& t_{j}& t_k&)^{\phi}\\  
=(&t_k & t_{k+1}^{\overline t_{i+1}}& t_{i+1} &  t_k &).
\end{tabular}}

\bigskip

(i). Suppose that $w \in (\Pi t_{[1{\uparrow}i]} t_j {\star})$.

 \begin{figure}[H]
\vspace{-.3cm}
\begin{center}
\setlength{\unitlength}{.1mm}
\begin{picture}(1100,306) 
\put(0,0){\epsfig{file= 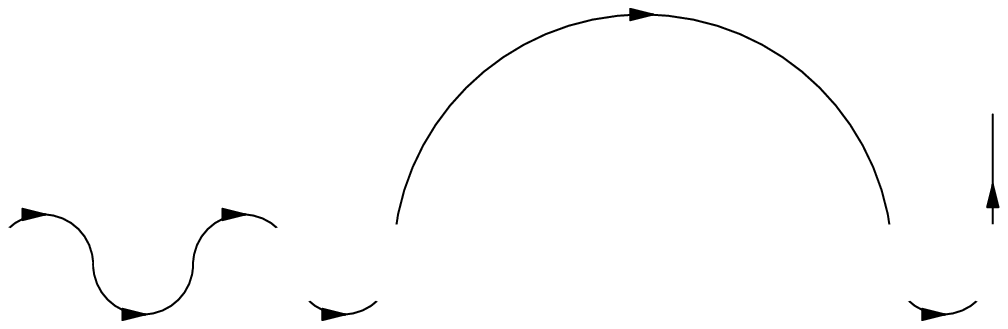,width=110truemm}}
\put(70,50){\makebox(0,0)[c]{$\overline z_1$}} 
\put(360,51){\makebox(0,0)[c]{$t_i$}} 
\put(450,57){\makebox(0,0)[c]{$\overline t_i$}} 
\put(945,51){\makebox(0,0)[c]{$t_j$}} 
\put(1040,57){\makebox(0,0)[c]{$\overline t_j$}} 
\end{picture}
\caption{$w \in (\Pi t_{[1{\uparrow}i]} t_j {\star})$, $j \ge i+2$.}\label{fig:xy2}
\end{center}
\vspace{-.7cm}
\end{figure}

Since $t_it_j$ is a subword of $w$, 
every letter occurring in $w$ that belongs to $t_{[i+1{\uparrow}j-1]}\cup\overline t_{[i+1{\uparrow}j-1]}$
 belongs to a (reduced) subword of $w$ of the form $av\overline b$, where $a,b \in \{t_i, \overline t_j\}$
 and $v \in \gen{t_{[i+1{\uparrow}j-1]}}$.  Since, moreover, $w$ begins with $\Pi t_{[1{\uparrow}i]}$,
it can be shown that it is not possible to have $a= t_i$ or $b=t_i$.  Thus $a = b = \overline t_j$.
Here, $\abs{(av\overline b)^\phi} = \abs{avb} -2$. 

We factor $w$ into syllables consisting of all such subwords together with the 
individual remaining letters, all of which lie in $t_{[1{\uparrow}i]}\cup t_{[j{\uparrow}n]}$,
and all of which are mapped to single letters by $\phi$.

It is then clear that   $\abs{w^{\phi}} \le \abs{w}$.

Moreover, if $\abs{w^{\phi}} = \abs{w}$, then $w \in \gen{t_{[1{\uparrow}i]}\cup t_{[j{\uparrow}n]}}$,
and $w^{\phi} \in (\Pi t_{[1{\uparrow}i+1]} {\star})$.

(ii). Suppose that $w \in (\Pi t_{[1{\uparrow}i]} \overline t_j{\star})$.

\begin{figure}[H]
\vspace{-.3cm}
\begin{center}
\setlength{\unitlength}{.1mm}
\begin{picture}(1100,357) 
\put(0,0){\epsfig{file= 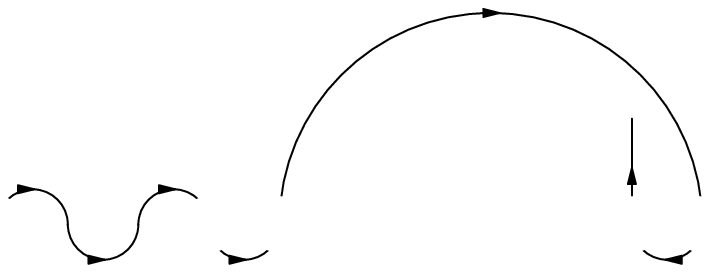,width=110truemm}}
\put(70,50){\makebox(0,0)[c]{$\overline z_1$}} 
\put(360,51){\makebox(0,0)[c]{$t_i$}} 
\put(450,57){\makebox(0,0)[c]{$\overline t_i$}} 
\put(945,51){\makebox(0,0)[c]{$t_j$}} 
\put(1040,57){\makebox(0,0)[c]{$\overline t_j$}} 
\end{picture}
\caption{$w \in (\Pi t_{[1{\uparrow}i]} \overline t_j{\star})$, $j \ge i+2$.}\label{fig:xy3}
\end{center}
\vspace{-.7cm}
\end{figure}

 Since $t_i \overline t_j$ is a subword of $w$, 
every letter occurring in $w$ that belongs to $t_{[i+1{\uparrow}j]}\cup\overline t_{[i+1{\uparrow}j]}$
 belongs to a (reduced) subword of $w$ of the form $av\overline b$, where $a,b \in \{t_i, \overline t_j\}$
 and $v \in \gen{t_{[i+1{\uparrow}j]}}$.  Since, moreover, $w$ begins with $\Pi t_{[1{\uparrow}i]}$,
it can be shown that it is not possible to have $a= t_i$ or $b=t_i$.  Thus $a = b = \overline t_j$.
Here, $\abs{(av\overline b)^\phi} = \abs{avb} -2$. 

We factor $w$ into syllables consisting of all such subwords together with the 
individual remaining letters, all of which lie in $t_{[1{\uparrow}i]}\cup t_{[j+1{\uparrow}n]}$,
and all of which are mapped to single letters by $\phi$.

Since $\overline t_j$ occurs in $w$, it is then clear that $\abs{w^\phi} \le \abs{w} -2$. 
\end{proof}

\begin{theorem}[Larue]\label{th:reps} The set $\{\Pi t_{[1{\uparrow}k]}\}_{k \in [0{\uparrow}n]}$ is a complete set of
representatives of the $\B_n$-orbits in the set of all planar words in $\Sigma_{0,1,n}$.
\end{theorem}

\begin{proof}  Let $w$ be a planar word in $\Sigma_{0,1,n}$.
We wish to show that there exists some $k \in [0{\uparrow}n]$ such that
$t_{[1{\uparrow}k]} \in w^{\B_n}$.

Let $i$ be the largest integer such that $w \in (\Pi t_{[1{\uparrow}i]}{\star})$.  

We may assume that, for all $v \in w^{\B_n}$, 
 $\abs{v} \ge \abs{w}$, and if $\abs{v} = \abs{w}$, then $v \not \in  (\Pi t_{[1{\uparrow}i+1]}{\star})$.

By Lemma~\ref{lem:length2},   for all $j \in [1{\uparrow}i-1]$, 
$w \not \in (\Pi t_{[1{\uparrow}i]}t_j{\star}) \cup (\Pi t_{[1{\uparrow}i]}\overline t_j{\star})$.

By Proposition~\ref{prop:square4}(i),   $w \not \in (\Pi t_{[1{\uparrow}i]}t_i{\star})$.

By the maximality of $i$, $w \not \in (\Pi t_{[1{\uparrow}i]} t_{i+1}{\star})$.

By Proposition~\ref{prop:square4}(iii),  $w \not \in (\Pi t_{[1{\uparrow}i]} \overline t_{i+1}{\star})$.

By Lemma~\ref{lem:length1},  for all $j \in [i+2{\uparrow}n]$, 
$w \not \in (\Pi t_{[1{\uparrow}i]}t_j{\star}) \cup (\Pi t_{[1{\uparrow}i]}\overline t_j{\star})$.

Hence, $w = \Pi t_{[1{\uparrow}i]}$, as desired.
\end{proof}

\begin{remarks} (i). Let $w$ be a planar word in $\Sigma_{0,1, n}$. 

Lemmas~\ref{lem:length2} and~\ref{lem:length1} give an effective procedure
for finding $\phi \in \B_n$ first to minimize $\abs{w^{\phi}}$, 
and then to obtain the form $w^\phi = \Pi t_{[1{\uparrow}k]}$ for some $k \in [0{\uparrow}n]$.

(ii). Let $n\ge 1$ and let $w$ be a word in $\Sigma_{0,1, n}$.  

Theorem~\ref{th:reps} shows that $w$ lies in the $\B_n$-orbit of $t_1$ if and only if
the cyclically-reduced form of $w$ lies in $t_{[1{\uparrow}n]}$ and
$w$ is planar.  Moreover, in this event,
Lemmas~\ref{lem:length2} and~\ref{lem:length1} effectively produce a 
$\phi \in \B_n$ such that $w^\phi = t_1$.

(iii). There is then an algorithm which, for any $k \in [1{\uparrow}n]$,
and any $k$-tuple $w_{([1{\uparrow}k])}$ for $\Sigma_{0,1,n}$,
decides if there exists some $\phi \in \B_n$ such that
$w^\phi_{([1{\uparrow}k])} = t_{([1{\uparrow}k])}$, and 
effectively finds such a $\phi$. We proceed as follows.  
We first convert $w_1$ to $t_1$ if possible,
and then we restrict to $\gen{\sigma_{[2{\uparrow}n-1]}}$.

It is interesting to compare this algorithm for $\B_n$  with the Whitehead algorithm
 for the much larger group $\Aut(\Sigma_{0,1,n})$.
The information provided by planarity is more detailed then the information
carried by the Whitehead graph used in the Whitehead algorithm.
\hfill\qed
\end{remarks}

We record the following.

\begin{theorem}[Larue] Let $n\ge 1$ and let $w \in \Sigma_{0,1, n}$.  
Then $w$ lies in the $\B_n$-orbit of $t_1$ if and only if
the cyclically-reduced form of $w$ lies in $t_{[1{\uparrow}n]}$ and
$w$ is planar. \hfill\qed
\end{theorem}

\vskip 1cm

\noindent{\textbf{\Large{Acknowledgments}}}

\medskip
\footnotesize

The research of both authors was 
funded by the MEC (Spain) and the EFRD (EU) through 
Projects BFM2003-06613 and MTM2006-13544.

We are grateful to Patrick Dehornoy for encouraging us to study the work of Larue, to Bert Wiest for supplying us with
photocopies of many pages of Larue's thesis, and to David Larue for kindly making his thesis available online.

We thank Mladen Bestvina and Edward Formanek for many interesting observations. 
\vskip -.5cm\null

\bibliographystyle{amsplain}

\begin{thebibliography}{30}
\vskip-0.6cm \null
\bibitem{ACampo}
Norbert A'Campo,
\newblock{\em Le groupe de monodromie du d\'eploiement des singularit\'es isol\'ees de courbes planes I.}
\newblock{ Math.\ Ann.\ \textbf{213}(1975), 1--32.}
\vskip-0.6cm \null
\bibitem{Sakuma}
Hirotaka Akiyoshi, Makoto Sakuma, Masaaki Wada and Yasushi Yamashita,
\newblock{\em Punctured torus groups and 2-bridge knot groups (I).}
\newblock{Lecture Notes in Mathematics \textbf{1909}, Springer, Berlin, 2007. xliii+252pp.}
\vskip-0.6cm \null
\bibitem{Artin25}
Emil Artin,
\newblock {\em Theorie der Z\"{o}pfe\/},  
\newblock{Abh.\ Math.\ Sem.\ Univ.\ Hamburg \textbf{4}(1925), 47--72.}
\vskip-0.6cm \null
\bibitem{Artin47}
E.\ Artin,
\newblock{\em Theory of braids\/},
\newblock{Ann.\ of Math.\ \textbf{48}(1947), 101--126.}
\vskip-0.6cm \null
\bibitem{Bacardit}
Llu\'{\i}s Bacardit Carrasco, 
\newblock{\em Representaci\'{o} de grups de trenes per automorfismes de grups,\/}
\newblock{MSc thesis, Univ.\ Aut\`{o}noma de Barcelona, 2006, 74 pp.}
\vskip-0.6cm \null
\bibitem{BirmanHilden}
Joan S.\ Birman and Hugh M.\ Hilden, 
\newblock{\em On isotopies of homeomorphisms of Riemann sufaces\/},
\newblock  Ann.\ of Math.\ \textbf{97}(1973), 424--439.
\vskip-0.6cm \null
\bibitem{Bohnenblust}
F.\ Bohnenblust, 
\newblock{\em The algebraical braid group\/},
\newblock  Ann.\ of Math.\ \textbf{48}(1947), 127--136.
\vskip-0.6cm \null
\bibitem{Chow}
Wei-Liang Chow,
\newblock{\em On the algebraical braid group\/},
\newblock Ann.\ of Math.\ \textbf{49}(1948), 654--658.
\vskip-0.6cm \null
\bibitem{Cooper}
Daryl Cooper,
\newblock{\em  Automorphisms of free groups have finitely generated fixed point sets\/},
\newblock J.\ Algebra \textbf{111}(1987), 453--456.
\vskip-0.6cm \null
\bibitem{CrispParis1}
John Crisp and Luis Paris,
\newblock{\em Artin groups of type $B$ and $D\!,$\/}
\newblock Adv.\ Geom.\ \textbf{5}(2005), 607--636.
\vskip-0.6cm \null
\bibitem{CrispParis2}
John Crisp and Luis Paris,
\newblock{\em Representations of the braid group by automorphisms 
of groups, invariants of links, and Garside groups\/},
\newblock Pacific J.\ Math.\ \textbf{221}(2005), 1--27.
\vskip-0.6cm \null
\bibitem{Dehornoy1}
Patrick Dehornoy,
\newblock{\em Sur la structure des gerbes libres\/},
\newblock C.\ R.\ Acad.\ Sci.\ Paris S\'er. I Math.\ \textbf{309}(1989), 143--148. 
\vskip-0.6cm \null
\bibitem{Dehornoy2}
Patrick Dehornoy,
\newblock{\em Free distributive groupoids\/},
\newblock J.\ Pure Appl.\ Algebra \textbf{61}(1989), 123--146.
\vskip-0.6cm \null
\bibitem{Dehornoy3}
Patrick Dehornoy,
\newblock{\em Braid groups and left distributive operations\/},
\newblock Trans.\ Amer.\ Math.\ Soc.\ \textbf{345}(1994), 115--150.
\vskip-0.6cm \null
\bibitem{Dehornoy4}
Patrick  Dehornoy,
\newblock{\em Braids and self-distributivity\/},
\newblock Progress in Mathematics \textbf{192},  Birkh\"auser Verlag, Basel, 2000.
\vskip-0.6cm \null
\bibitem{DDRW02}
Patrick Dehornoy, Ivan Dynnikov, Dale Rolfsen and Bert Wiest,
\newblock {\em Why are braids orderable?\/},
\newblock Panoramas et Synth\`eses \textbf{14}, Soc.\ Math.\ France, Paris, 2002.
\vskip-0.6cm\null
\bibitem{DicksDunwoody89}
Warren Dicks and M.\ J.\ Dunwoody,
\newblock {\em Groups acting on graphs}, 
\newblock Cambridge Stud. Adv. Math.  \textbf{17}, CUP,  Cambridge, 1989
\newline  Errata at: \texttt{http://mat.uab.cat/$\scriptstyle\sim$dicks/DDerr.html}
\vskip-0.6cm\null
\bibitem{DF05}
Warren Dicks and Edward Formanek,
\newblock{\em Algebraic mapping-class groups of orientable surfaces with boundaries\/},
pp. 57--116, in: {\em Infinite groups: geometric, combinatorial and dynamical aspects}
        (eds. Laurent Bartholdi, Tullio Ceccherini-Silberstein, Tatiana Smirnova-Nagnibeda, Andrzej Zuk), 
        Progress in Mathematics \textbf{248},  Birkh\"auser Verlag, Basel, 2005. 
\newline Errata and addenda at: \texttt{http://mat.uab.cat/$\scriptstyle\sim$dicks/Boundaries.html}
\vskip-0.6cm \null
\bibitem{Fenn}
R.\ Fenn, M.\ T.\ Greene, D.\ Rolfsen, C.\ Rourke and B. Wiest,
\newblock {\em Ordering the braid groups\/},
\newblock {Pacific J.\ Math.\ \textbf{191}(1999), 49--74.}
\vskip-0.6cm \null
\bibitem{Funk}
Jonathon Funk,
\newblock {\em The Hurwitz action and braid group orderings\/},
\newblock {Theory Appl.\ Categ.\ \textbf{9}(2001/02), 121--150.}
\vskip-0.6cm \null
\bibitem{LarueThesis94}
David Maurice Larue,
\newblock {\em Left-distributive and left-distributive idempotent algebras\/},
\newblock PhD thesis, University of Colorado, Boulder, 1994, ix + 138 pp.
\newline \texttt{http://www.mines.edu/fs$\underline{\hskip6pt}$home/dlarue/papers/dml.pdf}
\vskip-0.6cm \null
 \bibitem{Larue94}
David M.\ Larue,
\newblock {\em On braid words and irreflexivity\/},
\newblock {Algebra Universalis \textbf{31}(1994), 104--112.}
\vskip-0.6cm \null
 \bibitem{Magnus34}
Wilhelm Magnus, 
\newblock{\em \"{U}ber Automorphismen von Fundamentalgruppen berandeter Fl\"{a}chen\/},
\newblock{Math.\ Ann.\ \textbf{109}(1934), 617--646.}
\vskip-0.6cm \null
 \bibitem{Manfredini}
 Sandro Manfredini,
\newblock{\em Some subgroups of Artin's braid groups\/},
\newblock{Topology Appl. \textbf{78}(1997), 123--142.}
\vskip-0.6cm \null
 \bibitem{Markov}
A.\ Markov,
\newblock{\em Foundations of the algebraic theory of braids\/} (Russian),
\newblock{Trav.\ Inst.\ Math.\ Steklov \textbf{16}(1945), 53 pp.}
\vskip-0.6cm \null
 \bibitem{PerronVannier}
B.\ Perron and J.\ P.\ Vannier,
\newblock{\em Groupe de monodromie g\'eom\'etrique des singularit\'es simples\/} (Russian),
\newblock{Math.\ Ann.\ \textbf{306}(1996),  231--245.}
\vskip-0.6cm \null
 \bibitem{ShortWiest}
 Hamish Short and Bert Wiest, 
\newblock{\em Orderings of mapping class groups after Thurston\/},
\newblock{Enseign.\ Math.\ \textbf{46}(2000), 279--312.}
\vskip-0.6cm \null
 \bibitem{Shpilrain}
Vladimir Shpilrain,
\newblock{\em Representing braids by automorphisms\/},
\newblock{Internat.\ J.\ Algebra Comput.\ \textbf{11}(2001), 773--777.}
\vskip-0.6cm \null
 \bibitem{Wada}
Masaaki Wada, 
\newblock {\em Group invariants of links\/}, 
\newblock {Topology \textbf {31}(1992), 399--406.}
\end{thebibliography}

\medskip

\noindent\textsc{Departament de Matem\`atiques,\newline
Universitat Aut\`onoma de Barcelona,\newline
E-08193 Bella\-terra (Barcelona), Spain
}

\medskip

\noindent \emph{E-mail addresses}{:\;\;}\texttt{lluisbc@mat.uab.cat, dicks@mat.uab.cat}

\medskip

\noindent \emph{URL}{:\;\;}\texttt{http://mat.uab.cat/$\sim$dicks/}

\end{document}